\documentclass{article}

\usepackage[T1]{fontenc}

\usepackage{mathtools}
\usepackage{amssymb}
\usepackage{amsthm}
\usepackage{tikz-cd}

\usepackage{enumitem}
\usepackage{multicol}

\usepackage{bm}
\usepackage{upgreek}
\usepackage{euscript}
\usepackage{mathrsfs}
\usepackage{tensor}
\usepackage{stmaryrd}

\usepackage{scalerel}
\usepackage{letltxmacro}
\usepackage{textcomp}

\usepackage{hyperref}

\LetLtxMacro{\mathbfSpecial}{\bm}

\newtheorem*{thm*}{Theorem} 
\newtheorem*{cor*}{Corollary}
\newtheorem*{lmm*}{Lemma}
\newtheorem*{prop*}{Proposition}

\newtheorem{thm}{Theorem}[section]
\newtheorem{lmm}[thm]{Lemma}
\newtheorem{prop}[thm]{Proposition}

\theoremstyle{definition}

\newtheorem*{assumptions}{Assumptions}
\newtheorem*{warningAssumptions}{Warning on the assumptions}

\newtheorem*{deff*}{Definition}
\newtheorem*{notat*}{Notation}
\newtheorem*{rem*}{Remark}
\newtheorem*{es*}{Example}

\newtheorem{rem}[thm]{Remark}
\newtheorem{es}[thm]{Example}

\newcommand*{\virgolette}[1]{``#1''}

\newcommand*{\zpiu}{{\mathbb{Z}_{\geqslant 1}}}
\newcommand*{\nn}{\mathbb{Z}_{\geqslant 0}}
\newcommand*{\zz}{\mathbb{Z}}
\newcommand*{\qq}{\mathbb{Q}}
\newcommand*{\rr}{\mathbb{R}}

\newcommand*{\ff}{\mathbb{F}}

\newcommand*{\ooo}{\mathfrak{O}}
\newcommand*{\mmm}{\mathfrak{M}}
\newcommand*{\hhh}{\mathfrak{h}}

\DeclareMathOperator{\id}{id}

\DeclareMathOperator{\gal}{Gal}

\DeclareMathOperator{\image}{Image}

\DeclareMathOperator{\spann}{span}

\newcommand*{\defeq}{\coloneqq}

\newcommand*{\abs}[1]{\left\lvert #1 \right\rvert}

\newcommand*{\set}[2]{\left\{ #1 \;\;\middle|\;\; #2 \right\}}
\newcommand*{\ideal}[2]{\left( #1 \;\;\middle|\;\; #2 \right)}

\newcommand*{\restr}[2]{%
	\mathchoice
	{#1\big|_{#2}}
	{#1\big|_{#2}}
	{#1|_{#2}}
	{#1|_{#2}}
}

\DeclareMathOperator{\homm}{Hom}
\DeclareMathOperator{\End}{End}

\DeclareMathOperator{\ext}{Ext}

\newcommand*{\smooth}[2]{\mathsf{Rep}_{#1}^{\infty} (#2)}
\newcommand*{\vect}[1]{\mathsf{Vect}(#1)}
\newcommand*{\smoothh}[3]{\mathsf{Rep}_{#1}^{\infty, #3} (#2)}
\newcommand*{\dgmod}[1]{\mathsf{DGMod}(#1)}
\newcommand*{\modules}[1]{\mathsf{Mod}(#1)}

\newcommand*{\topgps}{\mathsf{TopGps}}

\newcommand*{\func}[5]{%
	\begin{tikzcd}[ampersand replacement=\&, row sep = 0pt]
   		\displaystyle{#1 \colon} \&[-3em] \displaystyle{#2} \ar[r] \& \displaystyle{#3} \\
   		\& \displaystyle{#4} \ar[r,mapsto] \& \displaystyle{#5}
   	\end{tikzcd}
}
\newcommand*{\functor}[7]{%
	\begin{tikzcd}[ampersand replacement=\&, row sep = 0pt]
   		\displaystyle{#1 \colon} \&[-3em] \displaystyle{#2} \ar[r] \& \displaystyle{#3}\\
   		\& \displaystyle{#4} \ar[r,mapsto] \& \displaystyle{#5}\\
   		\& \displaystyle{#6} \ar[r,mapsto] \& \displaystyle{#7}
   	\end{tikzcd}
}
\newcommand*{\funcNN}[4]{%
	\begin{tikzcd}[ampersand replacement=\&, row sep = 0pt]
   		\displaystyle{#1} \ar[r] \& \displaystyle{#2} \\
   		\displaystyle{#3} \ar[r,mapsto] \& \displaystyle{#4}
   	\end{tikzcd}
}

\newlength{\mymymywidth}

\newcommand*{\funcAboveTriple}[9]{%
	\settowidth{\mymymywidth}{\ensuremath{\displaystyle{#1}}}
	\begin{tikzcd}[ampersand replacement=\&, row sep = 0pt]
   		\displaystyle{#2} \ar[r, "\displaystyle{#1}"] \&[\mymymywidth] \displaystyle{#3}\\
   		\displaystyle{#4} \ar[r,mapsto] \& \displaystyle{#5}\\
   		\displaystyle{#6} \ar[r,mapsto] \& \displaystyle{#7}\\
   		\displaystyle{#8} \ar[r,mapsto] \& \displaystyle{#9}
   	\end{tikzcd}
}
\newcommand*{\funcInline}[3]{#1 \colon #2 \longrightarrow #3}

\newcommand*{\funcInlineNN}[2]{#1 \longrightarrow #2}

\newcommand*{\funcAbove}[3]{#2 \xrightarrow{#1} #3}

\newcommand*{\downmapsto}{\rotatebox[origin=c]{-90}{$\scriptstyle\mapsto$}\mkern2mu}

\newcommand*{\cupprod}{\mathbin{\scaledScaledSmile}}
	\newcommand{\scaledScaledSmile}{{\mathpalette\scaledSmile\relax}}
	\newcommand{\scaledSmile}[2]{\raisebox{-.3ex}{\scalebox{0.7}[1.5]{$#1\smile$}}}
\DeclareMathOperator*{\bigcupprod}{\scalerel*{\cupprod}{\textstyle\sum}}
\DeclareMathOperator{\res}{res}
\DeclareMathOperator{\cores}{cores}

\DeclareMathOperator{\SL}{SL}
\DeclareMathOperator{\GL}{GL}
\DeclareMathOperator{\PGL}{PGL}
\newcommand*{\matr}[4]{%
	\mathchoice
	{\begin{pmatrix}
			#1 & #2 \\ #3 & #4
	\end{pmatrix}}
	{\begin{psmallmatrix}
			#1 & #2 \\ #3 & #4
	\end{psmallmatrix}}
	{\begin{psmallmatrix}
			#1 & #2 \\ #3 & #4
	\end{psmallmatrix}}
	{\begin{psmallmatrix}
			#1 & #2 \\ #3 & #4
	\end{psmallmatrix}}
}

\newcommand*{\invol}{\EuScript{J}} 
\newcommand*{\fratt}[1]{\left( #1 \right)_\Phi} 
\newcommand*{\field}{\mathfrak{F}}
\DeclareMathOperator{\sh}{Sh} 
\newcommand*{\sss}{\EuScript{S}}
\newcommand*{\xx}{\mathbf{X}}
\DeclareMathOperator{\aff}{aff} 
\newcommand*{\uuu}{\EuScript{U}}
\newcommand*{\biinvol}[1]{\tensor*[^{\invol}]{\left(#1\right)}{^{\invol}}}

\DeclareMathOperator{\idd}{{\underline{id}}}

\newcommand*{\bz}[1]{\beta_{#1}^0}
\newcommand*{\bp}[1]{\beta_{#1}^+}
\renewcommand*{\bm}[1]{\beta_{#1}^-}
\newcommand*{\az}[1]{\alpha_{#1}^0}
\newcommand*{\ap}[1]{\alpha_{#1}^+}
\newcommand*{\am}[1]{\alpha_{#1}^-}

\newcommand*{\gm}{\mathbb{G}_{\operatorname{m}}}
\newcommand*{\ga}{\mathbb{G}_{\operatorname{a}}}
\newcommand*{\mmu}{\mathbfSpecial{\mu}}
\newcommand*{\GG}{\mathbf{G}}
\newcommand*{\TT}{\mathbf{T}}
\newcommand*{\UU}{\mathbf{U}}
\newcommand*{\CC}{\mathbf{C}}

\DeclareMathOperator{\pr}{pr}

\DeclareMathOperator{\val}{val}

\newcommand*{\rep}{\mathcal{R}}

\DeclareMathOperator{\exterior}{{\textstyle\bigwedge}}
\newcommand*{\mult}{\mathcal{M}}
\newcommand*{\azstar}[1]{\alpha_{#1}^{0,\star}}

\renewcommand{\text}[1]{\textnormal{#1}}

\DeclareMathOperator{\opp}{op}

\newcommand*{\chamber}{\mathfrak{C}}

\newcommand*{\quozright}[2]{#1 \backslash #2}

\newcommand*{\apart}{\mathscr{A}}

\DeclareMathOperator{\finite}{finite}

\DeclareMathOperator{\cind}{c-ind}

\title{Finite generation properties of the \texorpdfstring{pro-$p$}{pro-p} Iwahori--Hecke \texorpdfstring{$\ext$-algebra}{Ext-algebra}}
\author{Emanuele Bodon}
\date{\today}

\newcommand*{\quoz}[2]{#1 / #2}

\begin{document}

\maketitle

\begin{abstract}
The pro-$p$ Iwahori--Hecke $\ext$-algebra $E^\ast$ is a graded algebra that has been introduced and studied by Ollivier--Schneider, with the long-term goal of investigating the category of smooth mod-$p$ representations of $p$-adic reductive groups and its derived category. Its $0$\textsuperscript{th} graded piece is the pro-$p$ Iwahori--Hecke algebra studied by Vignéras and others.

In the present article, we first show that the $\ext$-algebra $E^\ast$ associated with the group $\SL_2(\field)$, $\PGL_2(\field)$ or $\GL_2(\field)$, where $\field$ is an unramified extension of $\qq_p$ with $p \neq 2,3$, is finitely generated as a (non-commutative) algebra.

We then specialize to the case of the group $\SL_2(\qq_p)$, with $p \neq 2,3$, and we show that in this case the natural multiplication map from the tensor algebra $T^\ast_{E^0} E^1$ to $E^\ast$ is surjective and that its kernel is finitely generated as a two-sided ideal. Using this fact as main input, we then show that $E^\ast$ is finitely presented as an algebra. We actually compute an explicit presentation.
\end{abstract}

\tableofcontents

\section{Introduction}

\subsection{Context}

In the framework of the local Langlands program, a fundamental problem is to obtain an understanding of categories of representations of $p$-adic reductive groups $G$, such as $G = \GL_n(\qq_p)$ or $G = \SL_n(\qq_p)$. In particular, in the case of representations over fields of characteristic $p$, such an understating is still far from complete.

Let us introduce the setting: let $p$ be a prime number, let $\field$ be a finite extension of $\qq_p$ or the field of formal Laurent series $\ff_q((X))$, where $q$ is a power of $p$. Let $G$ be the group of $\field$-rational points of a connected split reductive group defined over $\field$ (for example, $G = \GL_n(\field)$ or $G = \SL_n(\field)$). It is naturally endowed with the structure of a locally profinite topological group. Let $k$ be a field of characteristic $p$, and let $\smooth{k}{G}$ be the category of smooth representations of $G$ on $k$-vector spaces.

Let $I$ be a pro-$p$ Iwahori subgroup of $G$. For example, in the case that $G = \GL_2(\qq_p)$ or $\SL_2(\qq_p)$ we may choose
\[
	I =
	\matr{1+p\zz_p}{\zz_p}{p\zz_p}{1+p\zz_p}.
\]
(with the understanding that matrices have determinant $1$ if we are considering $\SL_2(\qq_p)$).

A fundamental tool to try to study the category $\smooth{k}{G}$ is the pro-$p$ Iwahori--Hecke algebra defined by
\[
	H \defeq \End_{\smooth{k}{G}}(k[G/I])^{\opp}.
\]
Here, $k[G/I]$ is the free $k$-vector space generated by the left cosets, endowed with the natural action of $G$ (in this way $k[G/I]$ is a smooth representation, that can be described more abstractly as the compact induction $\cind_I^G k$ of the trivial representation $k$). Furthermore, $\End_{\smooth{k}{G}}(k[G/I])$ denotes the ring of endomorphisms and $R^{\opp}$ denotes the opposite ring of a ring $R$. The ring $H$ is a (not necessarily commutative) $k$-algebra that has been studied in particular by Vignéras (see for example \cite{VignI}).

The link between $H$ and the category $\smooth{k}{G}$ is the following functor:
\begin{equation}\label{eqInvariantFunctor}
	\funcNN
		{\smooth{k}{G}}
		{\modules{H}}
		{V}
		{V^I \cong \homm_{\smooth{k}{G}}(k[G/I],V).}
\end{equation}
Here, $\modules{H}$ denotes the category of left modules over $H$ and $H$ acts on $ \homm_{\smooth{k}{G}}(k[G/I],V)$ by composition (defining a left action since we considered the opposite ring in the definition of $H$).

If $k$, instead of being of characteristic $p$, were the field of complex numbers, then the functor \eqref{eqInvariantFunctor} would be exact and, if furthermore $I$ were a Iwahori subgroup of $G$, then a fundamental result of Bernstein \cite[Corollary 3.9.ii]{Bern} would show that the functor induces an equivalence of categories when restricted to the full subcategory $\smoothh{k}{G}{I}$ of those smooth representations that are generated by their $I$-invariant vectors.

In our case instead, the functor \eqref{eqInvariantFunctor} is only left exact. The advantage of considering a pro-$p$ group when $k$ has characteristic $p$ is that a non-zero smooth representation $V$ is sent to a non-zero module $V^I$ by the functor \eqref{eqInvariantFunctor}. In particular, the subcategory $\smoothh{k}{G}{I}$ contains all irreducible representations. The specific choice of a pro-$p$ Iwahori subgroup is done in order for the algebra $H$ to be \virgolette{as simple as possible}.

In this setting, the functor \eqref{eqInvariantFunctor} restricted to $\smoothh{k}{G}{I}$ defines an equivalence of categories only in a few special cases (see \cite{OllInvariants, OllSecherreGL3, KozInvariants} and, in particular, \cite[Corollary 5.3]{KozInvariants}). Schneider \cite{Schn}, however, proved that the situation does improve when one constructs a derived version
\begin{equation}\label{eqDerivedFunctor}
	\funcInlineNN{D(\smooth{k}{G})}{D(\dgmod{\mathcal{H}^\bullet})}.
\end{equation}
of the functor \eqref{eqInvariantFunctor}. Here, $D(\smooth{k}{G})$ denotes the derived category of $\smooth{k}{G}$, while $\mathcal{H}^\bullet$ denotes a certain differential-graded algebra that plays the role of $H$ (more precisely it is the $\homm$ complex $\homm^\bullet(\mathcal{J}^\bullet, \mathcal{J}^\bullet)^{\opp}$, where $\funcInlineNN{k[\quoz{G}{K}]}{\mathcal{J}^\bullet}$ is a fixed injective resolution) and $D(\dgmod{\mathcal{H}^\bullet})$ denotes the derived category of the category of differential-graded modules over $\mathcal{H}^\bullet$. The main result of \cite{Schn} is that, under the assumption that $\field$ is a finite extension of $\qq_p$ and that $I$ is torsion-free, the functor \eqref{eqDerivedFunctor} defines an equivalence of triangulated categories (without the need to restrict to any subcategory).

The objects involved in \eqref{eqDerivedFunctor} are much more complicated than those in \eqref{eqInvariantFunctor}, and an explicit description of the differential graded algebra $\mathcal{H}^\bullet$ is not known.

As an intermediate step, one may try to study instead the $\ext$-algebra
\[
	E^\ast
	\defeq
	\ext^\ast(k[G/I],k[G/I])^{\opp}.
\]
This is the cohomology algebra of $\mathcal{H}^\bullet$, and its $0$\textsuperscript{th} graded piece is of course $H$. Such idea is pursued in the work of Ollivier--Schneider \cite{ext, newpaper}. They describe $E^\ast$ via a certain isomorphism of $k$-vector spaces
\begin{equation}\label{eqDescrEastViaShapiro}
	E^\ast \cong \bigoplus_{w \in \widetilde{W}} H^\ast(I \cap wIw^{-1},k).
\end{equation}
Here $\widetilde{W}$ is a certain subquotient of $G$ called the pro-$p$ Iwahori--Weyl group, and $H^\ast(I \cap wIw^{-1},k)$ denotes the cohomology of the profinite group $I \cap wIw^{-1}$ with respect to the trivial module $k$. This decomposition generalizes the Iwahori--Matsumoto basis of $H = E^0$ studied in \cite{VignI}.

Some of the main results by Ollivier-Schneider are the following: a description of the multiplication on $E^\ast$ via cohomological operations on the right hand side of \eqref{eqDescrEastViaShapiro} \cite[Proposition 5.3]{ext}; a more detailed description of the (left) action of $E^0$ on $E^\ast$ \cite[Proposition 5.6]{ext}; the existence of an involutive anti-automorphism \cite[Proposition 6.1]{ext}; a duality theorem in the case that $\field$ is a finite extension of $\qq_p$ and $I$ is torsion-free \cite[Proposition 7.18]{ext}; a complete description of the structure of $E^\ast$ as an $E^0$-bimodule in the case that $G = \SL_2(\qq_p)$ with $p \neq 2,3$ \cite[\S7]{newpaper}. As an interesting consequence, under the latter assumptions they show that an irreducible representation $V$ is supersingular if and only if its cohomology $H^\ast(I,V)$ is supersingular as a left $H$-module \cite[Corollary 8.12]{newpaper}.

\subsection{Main results}

The present paper is devoted to the study of finite generation properties of the $\ext$-algebra $E^\ast$, especially in the case $G = \SL_2(\qq_p)$ with $p \neq 2,3$.

In \S\ref{sectionBimod}, we study the $1$\textsuperscript{st} graded piece $E^1$ as an $E^0$-bimodule, and we prove the following result, under quite general assumptions.

\begin{prop*}[Proposition \ref{propE1fg}]
Assume that $\field$ is a finite extension of $\qq_p$. One has that $E^1$ is finitely generated as an $E^0$-bimodule.
\end{prop*}

To prove statements on the whole algebra $E^\ast$, in \S\ref{sectionSubalgE0E1} and \S\ref{sectionFinGen} we specialize to the case of \virgolette{small $G$} (as will be made precise in the statements below), and we study the multiplication map $\funcInlineNN{T^\ast_{E^0} E^1}{E^\ast}$, where $T^\ast_{E^0} E^1$ denotes the tensor algebra of the $E^0$-bimodule $E^1$. We first prove the following result.
\begin{prop*}[Proposition \ref{propAlmostGenByE1}, Remark \ref{remFiniteCodim}]
Assume that $p \neq 2,3$, that $\field$ is a (finite) unramified extension of $\qq_p$ and that $G$ is either $\SL_2(\field)$ or $\PGL_2(\field)$. One has the image of the multiplication map $\funcInlineNN{T^\ast_{E^0} E^1}{E^\ast}$ has finite codimension in $E^\ast$. If instead $G = \GL_2(\field)$, an analogous more technical result holds.
\end{prop*}

The two propositions above allow us to prove, with very little effort, that the $\ext$-algebra is finitely generated as a $k$-algebra in the above cases. This is the first main result of the paper.
\begin{thm*}[Theorem \ref{thmA1fg}]
Assume that $p \neq 2,3$, that $\field$ is a (finite) unramified extension of $\qq_p$ and that $G$ is one of $\SL_2(\field)$, $\PGL_2(\field)$ or $\GL_2(\field)$. One has that the $\ext$-algebra $E^\ast$ is finitely generated as a $k$-algebra.
\end{thm*}

The longest part of the paper (namely, \S\ref{sectionFinPres}) deals with the case $G = \SL_2(\qq_p)$ with $p \neq 2,3$. Under this assumption, the explicit description of the $\ext$-algebra of \cite{newpaper} becomes available, and we exploit such description to obtain more precise results.

We first show that the multiplication map $\funcInlineNN{T^\ast_{E^0} E^1}{E^\ast}$ is surjective, and then, with quite a long computation, we explicitly determine its kernel. We prove the following result.
\begin{thm*}[Theorem \ref{thmFinalResultKernel}]
Assume that $G=\SL_2(\qq_p)$ with $p \neq 2,3$.
The multiplication map
\[
	\funcInline{\mult}{T_{E^0}^\ast E^1}{E^\ast}
\]
is surjective and its kernel is generated by its $2$\textsuperscript{nd} and $3$\textsuperscript{rd} graded pieces as a two-sided ideal. Furthermore, the kernel is finitely generated as a two-sided ideal.
\end{thm*}
We will actually compute an explicit finite list of generators (see again Theorem \ref{thmFinalResultKernel}).

Starting from the \virgolette{presentation} of $E^\ast$ determined by the above theorem, we then compute a presentation of $E^\ast$ as a (non-commutative) $k$-algebra. This is achieved by computing a presentation of $E^0$ (the pro-$p$ Iwahori--Hecke algebra) as a $k$-algebra and a presentation of $E^1$ as an $E^0$-bimodule. Combining these two presentations with the above theorem, one then readily obtains a presentation of $E^\ast$ as a $k$-algebra. In this way, we obtain the following result.

\begin{thm*}[Theorem \ref{thmFinitePreskalg}]
Assume that $G=\SL_2(\qq_p)$ with $p \neq 2,3$.
The $\ext$-algebra $E^\ast$ is finitely presented as a $k$-algebra.
\end{thm*}
We stress the fact that we obtain this result by explicitly computing a finite presentation  (see again Theorem \ref{thmFinitePreskalg}).

\subsection{Acknowledgements}

A large portion of the material in this article was part of the author's PhD thesis \cite{thesis} under the direction of Peter Schneider. I thank him for the invaluable help and support, his fruitful and precise comments, and for suggesting such a research topic.

I am indebted with Claudius Heyer for the extremely useful discussions and his great interest in my work. I thank Rachel Ollivier for very interesting discussions and her support during the write-up of this article.

I acknowledge funding by the Deutsche Forschungsgemeinschaft (DFG, German Research Foundation) under Germany's Excellence Strategy EXC 2044 -- 390685587, Mathematics Münster: Dynamics -- Geometry -- Structure, and under the CRC 1442 Geometry: Deformations and Rigidity (Project Number 427320536).

\section{Setting and background}\label{sectionSettingBackground}

\subsection{Notation and assumptions}\label{subseNotazAssump}

We fix the following data throughout the paper.

We consider a non-archimedean local field $\field$. We denote by $p > 0$ its residue characteristic and by $q = p^f$ its residue cardinality. We recall that $\field$ is either a finite extension of $\qq_p$ or the field of Laurent series $\ff_q((X))$. Almost everywhere in this paper we will exclude the second case, but we nevertheless allow such case in this section and in a few lemmas later on for consistency with \cite{ext}, for sake of generality, and to point out a case in which our first result fails (Example \ref{esLaurent}). We denote by $\ooo$ the ring of integers of $\field$, and by $\mmm$ the maximal ideal of $\ooo$. We also fix a uniformizer $\pi$.

Let $\GG$ be a split connected reductive group over $\field$. Many results in this paper will only be valid for special choices of $\GG$, as will be explained later. We will always use the convention of denoting linear algebraic groups over $\field$ with boldface letters, and we will denote the associated group of $\field$-rational points with the corresponding non-boldface letter (e.g., in the case of $\GG$, we will have $G = \GG(\field)$). Such groups will be understood as locally profinite groups, with respect to the topology induced by $\field$.

Let us choose a split maximal torus $\TT$ of $\GG$. Let $T^0$ be the unique maximal compact subgroup of $T$ (isomorphic to $(\ooo^\times)^{\dim \TT}$), and let $T^1$ be the unique maximal pro-$p$ subgroup of $T$ (isomorphic to $(1+\mmm)^{\dim \TT}$).

Let $\mathbf{N}$ be the normalizer of $\TT$ in $\GG$. We define the (discrete) groups
\begin{align*}
	&W_0 \defeq \quoz{N}{T},
	&&W \defeq \quoz{N}{T^0},
	&&\widetilde{W} \defeq \quoz{N}{T^1}.
\end{align*}
The group $W_0$ is the usual Weyl group, the group $W$ is sometimes called extended affine Weyl group or Iwahori--Weyl group, and the group $\widetilde{W}$ is sometimes called pro-$p$ Iwahori--Weyl group.

Let us consider the semisimple (or adjoint) building associated with $\GG$ (see \cite[Definition 3.12]{VignI}) and the apartment corresponding to the choice of $\TT$. Let us choose a chamber $\chamber$ in such apartment. Let $J$ be the Iwahori subgroup of $G$ corresponding to the choice of $\chamber$ (see \cite{HainesRapp} for some possible, and equivalent, definitions), and let $I$ be the corresponding pro-$p$ Iwahori subgroup, i.e., the pro-$p$ radical of $J$.

We fix a field $k$ of characteristic $p$ (i.e., the same characteristic as the residue field of $\field$), which will be the coefficient field of the algebras that we will consider.

\begin{warningAssumptions}
Many results in the paper will be stated under stronger assumptions than those above. For \S\ref{sectionBimod}, \S\ref{sectionSubalgE0E1} and \S\ref{sectionFinGen}, any further assumption will be made explicit in every single statement. In \S\ref{sectionFinPres}, instead, the assumptions will be made explicit only at the beginning and will hold throughout.
\end{warningAssumptions}

\subsection{The apartment}\label{subsectionApartment}

In this section and the next ones, we will recall a few notions from Bruhat--Tits theory. We will only need a few facts from the treatise \cite{BT1, BT2}. For a more in-depth recollection of the what is needed in the context of the pro-$p$ Iwahori--Hecke algebra (or its $\ext$-version) see for example \cite[\S3]{VignI}. The present summary is partially inspired by those in \cite{ext} and \cite{ClaudiusThesis}.

The pair $(\GG,\TT)$ defines a root datum $\big( X^\ast(\TT), \Phi, X_\ast(\TT), \Phi^\vee \big)$, where $X^\ast(\TT)$ is the group of characters, $\Phi$ is the set of roots, $X_\ast(\TT)$ is the group of cocharacters, and $\Phi^\vee$ is the set of coroots.

Associated with every root $\alpha \in \Phi$ one has an additive subgroup $\UU_\alpha$ of $\GG$. We fix a Chevalley system $\big( \funcInline{x_\alpha}{\ga}{\UU_\alpha} \big)_{\alpha \in \Phi}$ according to the definition given in \cite[3.2.1 and 3.2.2]{BT2}. For all $\alpha \in \Phi$, we denote by $\funcInline{\varphi_\alpha}{\SL_2}{\GG}$ the associated morphism defined in \cite[3.2.1]{BT2}. In particular, we recall from loc. cit. that for all $\alpha \in \Phi$ and all $u \in \field$ one has
\begin{align*}
	x_\alpha(u) &= \varphi_{\alpha} \matr{1}{u}{0}{1},
	&
	x_{-\alpha}(u) &= \varphi_{\alpha} \matr{1}{0}{-u}{1}.
\end{align*}

For all $\alpha \in \Phi$, the homomorphism of algebraic groups
\[
	\func
		{\check{\alpha}}
		{\gm}
		{\GG}
		{x}
		{\varphi_{\alpha} \matr{x}{0}{0}{x^{-1}}}
\]
has values in $\TT$ (i.e., $\check{\alpha} \in X_\ast(\TT)$). It is called the \emph{coroot} associated with $\alpha$, and seeing $\rr \otimes_\zz X_\ast(\TT)$ as the dual of $\rr \otimes_\zz X^\ast(\TT)$, the coroot $\check{\alpha}$ is indeed the coroot associated with $\alpha$ in the sense of abstract root systems (and hence we will write $\check{\alpha} \in \Phi^\vee$).

The apartment $\apart$ corresponding to $\TT$ of the semisimple building of $\GG$ is a real affine space whose underlying vector space is
\[
	V_{\apart} \defeq (\quoz{X_\ast(\TT)}{X_\ast(\CC^\circ)}) \otimes_{\zz} \rr,
\]
where $\CC^\circ$ denotes the connected centre of $\GG$. We fix the valuation $x_0$ (in the sense of Bruhat--Tits) associated with the Chevalley system $(x_\alpha)_{\alpha \in \Phi}$ (see \cite[Examples (6.2.3) b)]{BT1}). It is an element of $\apart$, and one actually has
\[
	\apart = x_0 + V_{\apart}.
\]

We assume that the closure of the chamber $\chamber$ chosen in \S\ref{subseNotazAssump} contains $x_0$. Then the choice of $\chamber$ is equivalent to the choice of a set of positive roots $\Phi^+ \subseteq \Phi$. We denote by $\Phi^-$ the corresponding set of negative roots.

Let us consider the perfect pairing $\funcInline{\left\langle {}_-, {}_- \right\rangle}{X_\ast(\TT) \times X^\ast(\TT)}{\zz}$. It is $W_0$-equivariant for the left actions of $W_0$ on $X_\ast(\TT)$ and $X^\ast(\TT)$ (defined via the conjugation action of $W_0$ on $T$). For all $t \in T$ there exists a unique cocharacter $\mu(t) \in X_\ast(\TT)$ satisfying the condition
\begin{align*}
	&\langle \mu(t) , \chi \rangle = - \val_\field (\chi(t))
	&&\text{for all $\chi \in X^\ast(\TT)$.}
\end{align*}
The map $\funcInline{\mu}{T}{X_\ast(\TT)}$ obtained in this way is thus a surjective group homomorphism, it is $W_0$-equivariant and it induces an isomorphism $\funcInlineNN{\quoz{T}{T^0}}{X_\ast(\TT)}$. We will actually use more often the induced surjective group homomorphism
\[
	\funcInline{\nu}{\quoz{T}{T^0}}{\quoz{X_\ast(\TT)}{X_\ast(\CC^\circ)} \subseteq V_\apart.}
\]
We will often write $\nu(t)$ not only for $t \in \quoz{T}{T^0}$ but also for $t \in \quoz{T}{T^1}$.

We also consider the following perfect pairing, induced by the one above:
\[
	\funcInline{\left\langle {}_-, {}_- \right\rangle}{(\quoz{X_\ast(\TT)}{X_\ast(\CC^\circ)}) \times \spann_{\zz} \Phi}{\zz.}
\]
In particular, with this notation we have that $\langle \nu(t) , \alpha \rangle = - \val_\field (\alpha(t))$ for all $t \in \quoz{T}{T^0}$ and all $\alpha \in \Phi$.

By \cite[Lemma 56]{Steinberg}, the surjection $\funcInlineNN{N}{W_0}$ becomes split after one quotients out the subgroup of $T$ generated by $\check{\alpha}(-1)$ for $\alpha$ running over $\Phi$, and moreover an explicit splitting is obtained by sending, for all $\alpha \in \Phi$, the reflection defined by $\alpha$ to (the class of) $\varphi_\alpha \matr{0}{1}{-1}{0}$ (note that the result in loc. cit. is stated for split semisimple groups, but it might be applied to general split reductive groups by working with the derived subgroup). It thus follows that the exact sequence
\[\begin{tikzcd}
	1 \ar[r]
	&
	\quoz{T}{T^0} \ar[r]
	&
	W \ar[r]
	&
	W_0 \ar[r]
	&
	1
\end{tikzcd}\]
is split: therefore we write $W = \quoz{T}{T^0} \rtimes W_0$ and we view $W_0$ as a subgroup of $W$, fixing once and for all the splitting determined by the above rule.

The group $W$ acts on the apartment (or equivalently, on its supporting vector space) by the following rule:
\[
	\funcNN
		{W \times V_{\apart}}
		{V_{\apart}}
		{\begin{array}{c}
			{(w_0 t , x)}
			\\
			\scriptstyle\text{(for $w_0 \in W_0$, $t \in \quoz{T}{T^0}$)}
			\end{array}}
		{w_0 \cdot (x + \nu(t)).}
\]
where the action of $w_0$ on $V_{\apart}$ is defined via the usual action of the Weyl group on cocharacters.

\subsection{Affine roots and the length function}\label{subsecAffineRootsLength}

We continue recalling a few notions from the literature, following in particular \cite[\S2.1.3]{ext}.

We define the set of affine roots, the set of positive affine roots and the set of negative affine roots respectively as
\begin{align*}
	&\Phi_{\aff} \defeq \Phi \times \zz,
	&&\Phi_{\aff}^+ \defeq (\Phi \times \zz_{\geqslant 1}) \cup (\Phi^+ \times \{0\}),
	&&\Phi_{\aff}^- \defeq \Phi_{\aff} \smallsetminus \Phi_{\aff}^+.
\end{align*}
The group $W$ acts on the set of affine roots via the following rule:
\[
	\funcNN
		{W \times \Phi_{\aff}}
		{\Phi_{\aff}}
		{\begin{array}{c}
			{\big(w_0 t , (\alpha, \hhh)\big)}
			\\
			\scriptstyle\text{(for $w_0 \in W_0$, $t \in \quoz{T}{T^0}$)}
			\end{array}}
		{\big(w_0 \alpha, \hhh - \langle \nu(t), \alpha \rangle\big) = \big(w_0 \alpha, \hhh + \val_\field(\alpha(t)) \big).}
\]
The length function on $W$ is defined as
\[
\func
	{\ell}
	{W}
	{\nn}
	{w}
	{\# \set{ (\alpha, \hhh) \in \Phi_{\aff}^+ }{ w(\alpha,\hhh) \in \Phi_{\aff}^- }}
\]
(see \cite[\S 1.4]{Lusztig}). It satisfies the property $\ell(wv) \leqslant \ell(w) + \ell(v)$ for all $w,v \in W$. Further properties will be recalled in a moment.

Let $(\alpha, \hhh) \in \Phi_{\aff}$. The class of the element $s_{(\alpha, \hhh)} \defeq \varphi_\alpha \matr{0}{\pi^{\hhh}}{-\pi^{-\hhh}}{0}$ acts on the apartment as the reflection through the hyperplane
\[
	\set{x \in V_{\apart}}{\langle x, \alpha \rangle + \hhh = 0}.
\]
We define $W_{\aff}$ as the subgroup of $W$ generated by the elements $\varphi_\alpha \matr{0}{\pi^{\hhh}}{-\pi^{-\hhh}}{0}$ for $(\alpha, \hhh)$ running over $\Phi_{\aff}$. This group can be identified with the group generated by the reflections through the above hyperplanes. It forms a Coxeter system $(W_{\aff}, S_{\aff})$, where
\[
	S_{\aff} \defeq \set{ s_{(\alpha,\hhh)} }{(\alpha,\hhh) \in (\Pi \times \{0\}) \cup (\Pi_{\min} \times \{1\})},
\]
and where $\Pi$ denotes the basis of $\Phi$ corresponding to the choice of $\Phi^+$ and $\Pi_{\min}$ denotes the set of minimal roots with respect to the partial order $\leqslant$ on $\Phi$ defined by the rule $\alpha \leqslant \beta$ if and only if $\beta - \alpha$ is a linear combination with positive coefficients of elements of $\Pi$.

Recall from \cite[\S2.1.3]{ext} that the group $W$ admits a semidirect product decomposition
\[
	W = W_{\aff} \rtimes \Omega,
\]
where $\Omega$ is the subgroup of elements of length zero. Moreover, the length is constant on double cosets $(\Omega w \Omega)_{w \in W_{\aff}}$. When restricted to $W_{\aff}$, the length $\ell$ actually becomes the length function of the Coxeter system $(W_{\aff}, S_{\aff})$.

The explicit formulas to compute the length are the following:
\begin{align}
\label{eqLengthFormulaW0Right}
	\ell(t w_0)
	&=
	\sum_{\alpha \in \Phi^+ \cap w_0 \Phi^+}
	\abs{\langle \nu(t), \alpha \rangle}
	+
	\sum_{\alpha \in \Phi^+ \cap w_0 \Phi^-}
	\abs{\langle \nu(t), \alpha \rangle - 1},
\\
\label{eqLengthFormulaW0Left}
	\ell(w_0 t)
	&=
	\sum_{\alpha \in \Phi^+ \cap w_0^{-1} \Phi^+}
	\abs{\langle \nu(t), \alpha \rangle}
	+
	\sum_{\alpha \in \Phi^+ \cap w_0^{-1} \Phi^-}
	\abs{\langle \nu(t), \alpha \rangle + 1},
\end{align}
for all $w_0 \in W_0 \subseteq W$ and all $t \in \quoz{T}{T^0}$ (see \cite[Corollary 5.10]{VignI}). In particular, for all $t \in \quoz{T}{T^0}$ one has
\begin{equation}\label{eqLengthFormulaTorus}
	\ell(t)
	=
	\sum_{\alpha \in \Phi^+}
	\abs{\langle \nu(t), \alpha \rangle}
	=
	\frac{1}{2}
	\sum_{\alpha \in \Phi}
	\abs{\langle \nu(t), \alpha \rangle}.
\end{equation}

For our purposes it will be fundamental to work with the already defined group $\widetilde{W} \defeq \quoz{N}{T^1}$ instead of $W$. We will slightly abuse notation and write $\ell(w)$ not only for $w \in W$ but also for $w \in \widetilde{W}$. For a subgroup $X \subseteq W$ we will write $\widetilde{X}$ for the preimage of $X$ in $\widetilde{W}$. In particular, we will write $\widetilde{\Omega}$ and $\widetilde{W_0}$ (recall that we view $W_0 \subseteq W$).

\subsection{The \texorpdfstring{pro-$p$}{pro-p} Iwahori--Weyl group\texorpdfstring{ $\widetilde{W}$}{} and the subgroups \texorpdfstring{$I_w$}{I\_w}}\label{subsecIw}

Recall the definition $\widetilde{W} \defeq \quoz{N}{T^1}$ of the pro-$p$ Iwahori--Weyl group from \S\ref{subseNotazAssump}. Its importance stems from the following Bruhat decomposition:
\begin{equation}\label{eqBruhatDecomp}
G = \coprod_{w \in \widetilde{W}} IwI
\end{equation}
(see \cite[Proposition 3.35]{VignI}), where we have not distinguished between elements of $\widetilde{W}$ and their representatives in $N$ since $T^1 \subseteq I$, as we will recall in \S\ref{subsecIwaDecomp}.

For all $w \in \widetilde{W}$ we define the following subgroup of $I$:
\[
	I_w \defeq I \cap wIw^{-1}.
\]
Here, again, we have not distinguished between elements of $\widetilde{W}$ and their representatives in $N$. Actually, since $T^0$ normalizes $I$, this definition makes sense also for $w \in W$. The group $I_w$ is useful in the context of the Bruhat decomposition \eqref{eqBruhatDecomp} because of the $I$-equivariant bijection
\begin{equation}\label{eqBijDoubleCoset}
	\funcNN
		{\quoz{IwI}{I}}
		{\quoz{I}{I_w}}
		{iwI}
		{iI_w.}
\end{equation}

\subsection{The Iwahori decomposition}\label{subsecIwaDecomp}

Recall our fixed Chevalley system $\big( \funcInline{x_\alpha}{\ga}{\UU_\alpha} \big)_{\alpha \in \Phi}$. For all $\alpha \in \Phi$ let us consider the filtration $(\uuu_{(\alpha,m)})_{m \in \zz}$ of $U_\alpha$ defined by
\[
	\uuu_{(\alpha,m)} \defeq x_\alpha(\pi^m \ooo).
\]

The groups $\uuu_{(\alpha,1)}$ for $\alpha \in \Phi^-$, the group $T^1$ and the groups $\uuu_{(\alpha,0)}$ for $\alpha \in \Phi^+$ are all contained in the pro-$p$ Iwahori subgroup $I$ and furthermore the multiplication map
\[
\funcInlineNN
	{
		\prod_{\alpha\in \Phi^-}
		\uuu_{(\alpha, 1)}
		\times
		T^1
		\times
		\prod_{\alpha\in \Phi^+}
		\uuu_{(\alpha, 0)}
	}
	{I}
\]
is bijective (see \cite[Proposition I.2.2]{SchnSt} and \cite[Lemma 4.8 and its proof]{gorenst}; here the factors in the products are ordered in some arbitrarily chosen way). It is actually a homeomorphism, because it is a continuous bijection between compact Hausdorff spaces.

There is an analogous Iwahori decomposition for the subgroups $I_w$: for all $\alpha \in \Phi$ let us define
\begin{equation}\label{eqDefGWAlpha}
	g_w(\alpha)
	\defeq
	\min \set{ m \in \zz }{ (\alpha, m) \in \Phi_{\aff}^+ \cap w \Phi_{\aff}^+ }.
\end{equation}
Then the multiplication map induce a bijection (actually, a homeomorphism)
\begin{equation}\label{eqIwaDecompIw}
\funcInlineNN
{
	\prod_{\alpha\in \Phi^-}
	\uuu_{(\alpha, g_w(\alpha))}
	\times
	T^1
	\times
	\prod_{\alpha\in \Phi^+}
	\uuu_{(\alpha, g_w(\alpha))}
}
{I_w,}
\end{equation}
where, again, the factors in the products are ordered in some arbitrarily chosen way (see \cite[Lemma 2.3 and Remark 2.4]{ext}). Also note that the indices $g_w(\alpha)$ (and hence $I_w$) only depend on the image of $w$ in $W$ (and so we will sometimes use the notation $g_v(\alpha)$ for $v \in W$).

\subsection{The \texorpdfstring{pro-$p$}{pro-p} Iwahori--Hecke algebra}

As in the introduction, we consider the compact induction $\cind_I^G k$ from $I$ to $G$ of the trivial representation $k$. Concretely, the compact induction $\cind_I^G k$ can be described as the vector space $k[\quoz{G}{I}]$ endowed with the action defined by $g \cdot g'I \defeq gg'I$ (for all $g, g' \in G$).

As in \cite[\S2.2]{ext}, we define the pro-$p$ Iwahori--Hecke algebra to be
\[
	H \defeq \End_{\smooth{k}{G}}(k[\quoz{G}{I}])^{\opp},
\]
where $R^{\opp}$ denotes the opposite ring of a ring $R$. This definition can readily be seen to be equivalent to the one given in \cite[\S4.1]{VignI}, where $H$ is defined as $\End_{\smooth{k}{G}}(k[\quozright{I}{G}])$.

The Bruhat decomposition \eqref{eqBruhatDecomp} makes clear that one has a $k$-basis $(\tau_w)_{w \in \widetilde{W}}$ where $\tau_w$ is the characteristic function of the double coset $IwI$. This is called \emph{Iwahori--Matsumoto basis}.

A fundamental property is the following \emph{braid relation} (see \cite[Theorem 2.2]{VignI}):
\begin{align*}
	&\tau_{v} \cdot \tau_{w} = \tau_{vw}
	&&\text{for $v,w \in \widetilde{W}$ such that $\ell(vw) = \ell(v) + \ell(w)$}.
\end{align*}

In the case that $\ell(vw) \neq \ell(v) + \ell(w)$, the formula is more complicated, but it is completely determined by the above braid relation and the \emph{quadratic relations} (see \cite[Proposition 4.4 and Theorem 2.2]{VignI}). We will only introduce the quadratic relations in the case $\GG = \SL_2$ (see \eqref{eqQuadrRelSL2}), since we will not need them in the general case.

\subsection{The \texorpdfstring{$\ext$-algebra $E^\ast$}{Ext-algebra E*}}

We now introduce the \emph{pro-$p$ Iwahori--Hecke $\ext$-algebra} $E^\ast$ following \cite[\S3]{ext}. The definition, as anticipated in the introduction is:
\[
	E^\ast
	\defeq
	\ext^\ast(k[G/I],k[G/I])^{\opp},
\]
where, this time, $R^{\opp}$ denotes the opposite a graded ring $R$.

There is a canonical identification of $k$-vector spaces
\[
	E^\ast \cong H^\ast \left(I, k \left[ \quoz{G}{I} \right] \right),
\]
where $H^\ast$ denotes cohomology of profinite groups. Indeed, the abelian category $\smooth{k}{G}$ has enough injective objects (see \cite[I.5.9]{VigLMod}), and so we can fix an injective resolution $\funcInlineNN{k[\quoz{G}{I}]}{J^\bullet}$ in $\smooth{k}{G}$ and compute
\begin{align*}
	E^\ast
	&=
	\ext^\ast_{\smooth{k}{G}} \left( k \left[ \quoz{G}{I} \right], k \left[ \quoz{G}{I} \right] \right)
	\\&\cong
	H^\ast \big( \homm(k \left[ \quoz{G}{I} \right], J^\bullet) \big)
	\\&\cong
	H^\ast \big( (J^\bullet)^I \big)
	\\&\cong
	H^\ast \left(I, k \left[ \quoz{G}{I} \right] \right),
\end{align*}
where we have used Frobenius reciprocity for the compact induction and the fact that restriction functor from $\smooth{k}{G}$ to $\smooth{k}{I}$ preserves injective objects (see \cite[Chapitre I, 5.9 d)]{VigLMod}).

Then, similarly to $H$, the Bruhat decomposition \eqref{eqBruhatDecomp} induces a $k$-vector space decomposition
\begin{align*}
	&E^\ast
	\cong
	\bigoplus_{w \in \widetilde{W}} H^\ast(I, \xx(w)),
	&&\text{where $\xx(w) \defeq k[\quoz{IwI}{I}]$,}
\end{align*}
since the functor $H^\ast(I, {}_-)$ commutes with arbitrary direct sums (see \cite[Chapter I, Proposition 8]{Serre}).

Let $w \in \widetilde{W}$. From \eqref{eqBijDoubleCoset} we have a fixed isomorphism of $I$-representations $k[\quoz{IwI}{I}] \cong k[\quoz{I}{I_w}]$. Composing the induced isomorphism between cohomology spaces with the Shapiro isomorphism $H^j \left( I, k \left[ \quoz{I}{I_w} \right] \right) \cong H^j \left( I_w, k \right)$, we thus obtain an isomorphism
\[\begin{tikzcd}
\sh_w \colon
	H^\ast \left( I, \xx(w) \right)
	\ar[r, "\cong"]
	&
	H^\ast \left( I, k \left[ \quoz{I}{I_w} \right] \right)
	\ar[r, "\cong"]
	&
	H^\ast \left( I_w, k \right),
\end{tikzcd}\]
which we will still call \emph{Shapiro isomorphism}.

One of the main results of \cite{ext} consists in the description of the multiplication on $E^\ast$ in terms of cohomological operations on $\bigoplus_{w \in \widetilde{W}} H^\ast(I, \xx(w))$ (equivalently, on $\bigoplus_{w \in \widetilde{W}} H^\ast \left( I_w, k \right)$), see \cite[Proposition 5.3]{ext}. Except in the case $G = \SL_2(\qq_p)$, we will not use the full multiplicative structure of $E^\ast$, and so below, after a brief discussion on the needed cohomological operations, we will only state a partial (but interesting) result.

For all $w \in \widetilde{W}$ the cohomology space $H^\ast \left( I_w, k \right)$ has the structure of a graded-commutative $k$-algebra with respect to the cup product $\cupprod$. On the other side, the $G$-equivariant bilinear map
$
	\funcInlineNN{k \left[ \quoz{G}{I} \right] \times k \left[ \quoz{G}{I} \right]}{k \left[ \quoz{G}{I} \right]}
$
gives rise to a cup product on the cohomology space $H^\ast(I,k[\quoz{G}{I}])$ (see \cite[Chapter I, \S4]{NeukAndCo}). This cup product is associative and graded-commutative but in general not unital. The two cup products are compatible in the following sense: if we transport the cup product from $H^\ast(I,k[\quoz{G}{I}])$ to $\bigoplus_{w \in \widetilde{W}} H^\ast \left( I_w, k \right)$, the resulting product is zero when we multiply two elements lying in distinct direct summands, and, when restricted to $H^\ast \left( I_w, k \right)$, coincides with the natural cup product (see \cite[\S 3.3]{ext}).

We also recall the conjugation on cohomology of profinite groups (with trivial coefficients). Let $H$ be a locally profinite group, let $K$ be a compact subgroup and let $h \in H$. The conjugation map by $h^{-1}$
\[
	\funcNN{hKh^{-1}}{K}{x}{h^{-1}xh}
\]
defines a map
\[
	\funcInline{h_\ast}{H^i(K,k)}{H^i(hKh^{-1},k)}
\]
for all $i \in \nn$.

We can now state the above-mentioned partial result about the multiplication in $E^\ast$ (see \cite[Corollary 5.5]{ext}): let us consider $v,w \in \widetilde{W}$ such that $\ell(vw) = \ell(v) + \ell(w)$, let $i,j \in \nn$, let $\alpha \in H^i(I,\xx(v))$ and let $\beta \in H^j(I,\xx(w))$. One has
\begin{align}\label{eqCupYoneda}
	\alpha \cdot \beta = (\alpha \cdot \tau_w) \cupprod (\tau_v \cdot \beta)
\end{align}
Moreover, one has
\begin{equation}\label{eqActionRightLeftE0IfLengthsAddUp}
	\begin{aligned}
		\alpha \cdot \tau_w &\in H^i(I,\xx(vw)),
		\\
		\sh_{vw}(\alpha \cdot \tau_w) &= \res_{I_{vw}}^{I_v} \big( \sh_v(\alpha) \big),
		\\
		\tau_v \cdot \beta &\in H^i(I,\xx(vw)),
		\\
		\sh_{vw}(\tau_v \cdot \beta) &= \res_{I_{vw}}^{v I_w v^{-1}} \big( v_\ast \sh_w(\beta) \big).
	\end{aligned}
\end{equation}

\subsection{A few results on cohomology of \texorpdfstring{pro-$p$}{pro-p} groups}\label{subsecCohom}

Let $K$ be a pro-$p$ group. Let us consider the \emph{Frattini subgroup} $\overline{[K,K]K^p}$, i.e., the smallest closed subgroup of $K$ containing the commutator subgroup and the $p$-powers of the elements of $K$. We consider the \emph{Frattini quotient}
\[
	\fratt{K} \defeq \quoz{K}{\overline{[K,K]K^p}}.
\]
The first cohomology with trivial coefficients can then be computed as
\[
	H^1(K,k)
	\cong
	\homm_{\topgps}(K,k)
	\cong
	\homm_{\topgps}(\fratt{K},k),
\]
where $\homm_{\topgps}({}_-,{}_-)$ denotes continuous group homomorphisms.
If furthermore $\fratt{K}$ is finite, then we obtain
\[
	H^1(K,k)
	\cong
	\homm_{\vect{\ff_p}}(\fratt{K},k),
\]
where $\homm_{\vect{\ff_p}}({}_-,{}_-)$ denotes homomorphisms of $\ff_p$-vector spaces.

In our situation we consider the case $K = I_w$ for $w \in \widetilde{W}$. Then the Iwahori decomposition \eqref{eqIwaDecompIw} induces a surjective homomorphism of topological groups
\begin{equation}\label{eqPsiw}
\func
	{\Psi_{w}}
	{\big( \quoz{T^1}{(T^1)^p} \big) \oplus \bigoplus_{\alpha \in \Phi} \quoz{\ooo}{p\ooo}}
	{\fratt{I_w}}
	{\big( \overline{t}, (\overline{u_\alpha})_{\alpha \in \Phi} \big)}
	{\overline{t} \cdot \prod\limits_{\alpha \in \Phi} \overline{x_\alpha(\pi^{g_w(\alpha)} u_\alpha)}.}
\end{equation}

If $\field$ is a finite extension of $\qq_p$ we have that both $\quoz{\ooo}{p\ooo}$ and $\quoz{1+\mmm}{(1+\mmm)^p}$ (and hence also $\quoz{T^1}{(T^1)^p}$) are finite. Therefore,
\begin{align}\label{eqFrattiniFinite}
	&H^1(I_w,k)
	\cong
	\homm_{\vect{\ff_p}}(\fratt{I_w},k)
	&&\text{if $\field$ is a finite extension of $\qq_p$.}
\end{align}

For the general structure of the cup-product algebra $H^\ast(K,k)$ of a pro-$p$ group $K$, we will use some results due to Lazard and Serre. Let $K$ be a torsion-free analytic pro-$p$ group of dimension $n$ as an analytic manifold over $\qq_p$. One has that $K$ is a Poincaré group of dimension $n$ (see \cite[Theorem 5.1.9 and the following lines]{horizonsCohom}), meaning that the following conditions are satisfied: $H^n(K,k)$ is a one-dimensional $k$-vector space; for all $i > n$ the space $H^i(K,k)$ is zero; for all $i \in \{0,\dots,n\}$ the pairing
\[
	\funcInline{\cupprod}{H^i(K,k) \times H^{n-i}(K,k)}{H^{n}(K,k)}
\]
is non-degenerate. In the literature the definition of Poincaré pro-$p$ group is usually given only in the case $k = \ff_p$, but for for a general $k$ of characteristic $p$, the same properties follow because of the natural isomorphism $H^\ast({}_-,k) \cong H^\ast({}_-,\ff_p) \otimes_{\ff_p} k$.

A pro-$p$ group $K$ is called \emph{uniform} if it is topologically finitely generated, torsion-free and powerful (see \cite[Theorem 4.5]{SegalAndCo}); by definition, powerful means that $\overline{[K,K]} \subseteq \overline{K^p}$ if $p$ is odd and $\overline{[K,K]} \subseteq \overline{K^4}$ if $p=2$. Under these assumptions, a result by Lazard (see \cite[Theorem 5.1.5]{horizonsCohom}) states that one has natural isomorphisms of $k$-algebras
\[
	H^\ast(K,k) \cong \exterior^\ast \big(H^1(K,k) \big) \cong \exterior^\ast \big(\homm_{\vect{\ff_p}} (\fratt{K},k) \big),
\]
where $\exterior^\ast({}_-)$ denotes the exterior algebra. Note that under the above assumptions $\fratt{K}$ is finite (see \cite[1.14 Proposition]{SegalAndCo}).

\section{The first graded piece \texorpdfstring{$E^1$}{E\textasciicircum{}1} is finitely generated as an \texorpdfstring{$E^0$\mbox{-}\nobreak\hspace{0pt}bimodule}{E\textasciicircum{}0-bimodule}}\label{sectionBimod}

In this section we will prove that $E^1$ is finitely generated as an $E^0$-bimodule in the case that $\field$ is a finite extension of $\qq_p$ (Proposition \ref{propE1fg}). For the moment we do not make such assumption, and we start with a lemma on the generalized affine Weyl group $W$. The proof is inspired by the arguments in \cite[Lemma 2.11]{VignII} (which the author already reused in \cite[Lemma 3.2.1]{thesis}).

\begin{lmm}\label{lmmLengthsAddUp}
Let $w_0 \in W_0 \subseteq W$, let $r \in \nn$, let $\overline{\mathfrak{c}}$ be a closed Weyl chamber in $V_{\apart} = \left( \quoz{X_\ast(\TT)}{X_\ast(\CC^\circ)} \right) \otimes_{\zz} \rr$, let $x_1, \dots, x_r \in \quoz{T}{T^0} \subseteq W$ be such that $\nu(x_i)$ lies in $\overline{\mathfrak{c}}$ for all $i \in \{1,\dots,r\}$, let $n_1, \dots, n_r \in \nn$ and let $m_1, \dots, m_r \in \nn$ with the property that if $m_i > 0$, then $n_i > 0$ (for all $i \in \{1,\dots,r\}$). One has that
\[
	\ell \big( w_0 x_1^{n_1+m_1} \dots x_r^{n_r+m_r})
	=
	\ell \big( w_0 x_1^{n_1} \dots x_r^{n_r})
	+
	\ell \big(x_1^{m_1} \dots x_r^{m_r}).
\]
\end{lmm}

\begin{proof}
Recall the notations in \S\ref{subsecAffineRootsLength}. Let $x \defeq x_1^{n_1} \dots x_r^{n_r}$ and let $y \defeq x_1^{m_1} \dots x_r^{m_r}$. Looking at the length formula \eqref{eqLengthFormulaW0Left}, we see that
\begin{align*}
	\ell(w_0 x y)
	&=
	\sum_{\alpha \in \Phi^+ \cap w_0^{-1} \Phi^+}
		\abs{\langle \nu(x y), \alpha \rangle}
	+
	\sum_{\alpha \in \Phi^+ \cap w_0^{-1} \Phi^-}
		\abs{\langle \nu(x y), \alpha \rangle + 1}
	\\&\leqslant
	\sum_{\alpha \in \Phi^+ \cap w_0^{-1} \Phi^+}
		\big( \abs{\langle \nu(x), \alpha \rangle} + \abs{\langle \nu(y), \alpha \rangle} \big)
	\\&\qquad+
	\sum_{\alpha \in \Phi^+ \cap w_0^{-1} \Phi^-}
		\big( \abs{\langle \nu(x), \alpha \rangle + 1} + \abs{\langle \nu(y), \alpha \rangle} \big)
	\\&=
	\ell(w_0 x) + \ell(y),
\end{align*}
and that one has equality if and only if the following two conditions hold:
\begin{enumerate}[label=(\roman*)]
	\item \label{itemCondCompatI} For all $\alpha \in \Phi^+ \cap w_0^{-1} \Phi^+$ the signs of $\langle \nu(x), \alpha \rangle$ and of $\langle \nu(y), \alpha \rangle$ are compatible, meaning that their product is bigger or equal than zero;
	\item \label{itemCondCompatII} For all $\alpha \in \Phi^+ \cap w_0^{-1} \Phi^-$ the signs of $\langle \nu(x), \alpha \rangle + 1$ and of $\langle \nu(y), \alpha \rangle$ are compatible.
\end{enumerate}
The condition \ref{itemCondCompatI} is satisfied because $\nu(x)$ and $\nu(y)$ lie in a common closed Weyl chamber. To check the thesis of the lemma, it remains to show that condition \ref{itemCondCompatII} is satisfied.

Let us consider $\alpha \in \Phi^+ \cap w_0^{-1} \Phi^-$. Since all the $\nu(x_i)$'s lie in a common closed Weyl chamber, we have either that $\langle \nu(x_i), \alpha \rangle \geqslant 0$ for all $i \in \{1,\dots,r\}$ or that $\langle \nu(x_i), \alpha \rangle \leqslant 0$ for all $i \in \{1,\dots,r\}$. In the first case both $\langle \nu(x), \alpha \rangle + 1$ and $\langle \nu(y), \alpha \rangle$ are bigger or equal than $0$, and so condition \ref{itemCondCompatII} is satisfied. Now, let us assume that we are in the second case. If furthermore there exists $i_0 \in \{1,\dots,r\}$ such that $\langle \nu(x_{i_0}), \alpha \rangle < 0$ and $m_{i_0} > 0$ (and so also $n_{i_0} > 0$, by assumption), we obtain:
\begin{align*}
	\langle \nu(x), \alpha \rangle + 1
	&\leqslant
	n_{i_0} \cdot \langle \nu(x_{i_0}), \alpha \rangle + 1
	\\&\leqslant 0,
	\\
	\langle \nu(y), \alpha \rangle
	&\leqslant
	m_{i_0} \cdot \langle \nu(x_{i_0}), \alpha \rangle
	\\&<
	0,
\end{align*}
Hence, in this case condition \ref{itemCondCompatII} is satisfied. Therefore, it remains to consider the case that $\langle \nu(x_i), \alpha \rangle \leqslant 0$ for all $i \in \{1,\dots,r\}$ and that $m_i = 0$ for all $i \in \{1,\dots,r\}$ such that $\langle \nu(x_{i}), \alpha \rangle < 0$. These conditions force $\langle \nu(y), \alpha \rangle$ to be $0$, and so condition \ref{itemCondCompatII} is satisfied.
\end{proof}

Recall from \eqref{eqDefGWAlpha} the indices $g_w(\alpha)$ that appear in the Iwahori decomposition \eqref{eqIwaDecompIw}. In the following lemma we will study a certain property of these indices that will be crucial to prove Proposition \ref{propE1fg}. The lack of symmetry in the two conditions below reflect the lack of symmetry in the formulas \eqref{eqActionRightLeftE0IfLengthsAddUp} (where in one case only restriction appears and in the other case restriction and conjugation appear). The proof of the lemma is similar to \cite[Lemma 3.2.5]{thesis}.

\begin{lmm}\label{lmmgw}
Let $w_0 \in W_0 \subseteq W$, let $r \in \nn$, let $\overline{\mathfrak{c}}$ be a closed Weyl chamber in $V_{\apart} = \left( \quoz{X_\ast(\TT)}{X_\ast(\CC^\circ)} \right) \otimes_{\zz} \rr$, let $x_1, \dots, x_r \in \quoz{T}{T^0} \subseteq W$ be such that $\nu(x_i)$ lies in $\overline{\mathfrak{c}}$ for all $i \in \{1,\dots,r\}$, let $n_1, \dots, n_r \in \nn$ and let $m_1, \dots, m_r \in \nn$ with the property that if $m_i > 0$, then $n_i > 0$ (for all $i \in \{1,\dots,r\}$). Finally, let $\alpha \in \Phi$.
One has that at least one of the following two conditions is true:
\begin{itemize}
\item $g_{w_0 x_1^{n_1+m_1} \dots x_r^{n_r+m_r}} (\alpha) = g_{w_0 x_1^{n_1} \dots x_r^{n_r}} (\alpha)$;
\item $g_{w_0 x_1^{n_1+m_1} \dots x_r^{n_r+m_r}} (\alpha) = g_{w_0 x_1^{n_1} \dots x_r^{n_r}} (\alpha) - \langle \nu(x_1^{m_1} \dots x_r^{m_r}), w_0^{-1} \alpha \rangle$.
\end{itemize}
\end{lmm}

\begin{proof}
For brevity, let us define
\begin{align*}
	a &\defeq g_{w_0 x_1^{n_1+m_1} \dots x_r^{n_r+m_r}} (\alpha),
	\\
	b &\defeq g_{w_0 x_1^{n_1} \dots x_r^{n_r}} (\alpha),
	\\
	c &\defeq g_{w_0 x_1^{n_1} \dots x_r^{n_r}} (\alpha) - \langle \nu(x_1^{m_1} \dots x_r^{m_r}), w_0^{-1} \alpha \rangle.
\end{align*}
We have to check that either $a = b$ or $a = c$.

For all $\beta \in \Phi$ let us set
\[
	\varepsilon_\beta
	\defeq
	\begin{cases}
		0 &\text{if $\beta \in \Phi^+$,}
		\\
		1 &\text{if $\beta \in \Phi^-$.}
	\end{cases}
\]
Recalling the definition of $g_{({}_-)}(\alpha)$ from \eqref{eqDefGWAlpha}, for all $x \in \quoz{T}{T^0}$ we compute:
\begin{align*}
	g_{w_0x}(\alpha)
	&=
	\min \set{ m \in \zz }{ (\alpha, m) \in \Phi_{\aff}^+ \cap w_0x \Phi_{\aff}^+ }
	\\&=
	\min \set
		{ m \in \zz }
		{ \begin{matrix*}[l]
			(\alpha, m) \in \Phi_{\aff}^+,
			\\
			\left( w_0^{-1} \alpha, m + \left\langle \nu(x) , w_0^{-1} \alpha \right\rangle \right) \in \Phi_{\aff}^+
		\end{matrix*} }
	\\&=
	\max \left\{ \varepsilon_\alpha, \varepsilon_{w_0^{-1} \alpha} - \left\langle \nu(x) , w_0^{-1} \alpha \right\rangle \right\}.
\end{align*}
Plugging in, we obtain
\begin{align*}
	a &= \max \left\{ \varepsilon_\alpha, \varepsilon_{w_0^{-1} \alpha} - \left\langle \nu(x_1^{n_1+m_1} \dots x_r^{n_r+m_r}) , w_0^{-1} \alpha \right\rangle \right\},
	\\
	b &= \max \left\{ \varepsilon_\alpha, \varepsilon_{w_0^{-1} \alpha} - \left\langle \nu(x_1^{n_1} \dots x_r^{n_r}) , w_0^{-1} \alpha \right\rangle \right\},
	\\
	c &= \max \left\{ \varepsilon_\alpha, \varepsilon_{w_0^{-1} \alpha} - \left\langle \nu(x_1^{n_1} \dots x_r^{n_r}) , w_0^{-1} \alpha \right\rangle \right\} - \langle \nu(x_1^{m_1} \dots x_r^{m_r}), w_0^{-1} \alpha \rangle.
\end{align*}
We distinguish three cases on the basis of the sign of $\left\langle \nu(x_1^{m_1} \dots x_r^{m_r}) , w_0^{-1} \alpha \right\rangle$.
\begin{itemize}
\item
	Let us consider the case $\left\langle \nu(x_1^{m_1} \dots x_r^{m_r}) , w_0^{-1} \alpha \right\rangle < 0$. Using the fact that the $\nu(x_i)$'s lie in a common closed Weyl chamber, we deduce that $\left\langle \nu(x_i) , w_0^{-1} \alpha \right\rangle \leqslant 0$ for all $i \in \{1,\dots,r\}$ and that $\left\langle \nu(x_{i_0}) , w_0^{-1} \alpha \right\rangle < 0$ and $m_{i_0} > 0$ for at least one $i_0 \in \{1,\dots,r\}$. By assumption, this implies that $n_{i_0} > 0$. We thus see that $\langle \nu(x_1^{n_1} \dots x_r^{n_r}), w_0^{-1} \alpha \rangle < 0$. Of course, we also have that $\left\langle \nu(x_1^{n_1+m_1} \dots x_r^{n_r+m_r}) , w_0^{-1} \alpha \right\rangle < 0$, and we conclude that
	\begin{align*}
		a
		&=
		\varepsilon_{w_0^{-1} \alpha} - \left\langle \nu(x_1^{n_1+m_1} \dots x_r^{n_r+m_r}) , w_0^{-1} \alpha \right\rangle,
		\\
		c
		&=
		\varepsilon_{w_0^{-1} \alpha} - \left\langle \nu(x_1^{n_1} \dots x_r^{n_r}) , w_0^{-1} \alpha \right\rangle - \langle \nu(x_1^{m_1} \dots x_r^{m_r}), w_0^{-1} \alpha \rangle,
	\end{align*}
	and hence that $a=c$.
\item
	Let us consider the case $\left\langle \nu(x_1^{m_1} \dots x_r^{m_r}) , w_0^{-1} \alpha \right\rangle > 0$. With the same argument as before, we see that $\langle \nu(x_1^{n_1} \dots x_r^{n_r}), w_0^{-1} \alpha \rangle > 0$, and, of course, also that $\left\langle \nu(x_1^{n_1+m_1} \dots x_r^{n_r+m_r}) , w_0^{-1} \alpha \right\rangle > 0$. We thus conclude that
	\[
		a = \varepsilon_\alpha = b.
	\]
\item
	It remains to consider the case $\left\langle \nu(x_1^{m_1} \dots x_r^{m_r}) , w_0^{-1} \alpha \right\rangle = 0$, but under such assumption we clearly have $a=b=c$.
\qedhere
\end{itemize}
\end{proof}

Before stating the main result of this section, we will need the following two elementary lemmas.

\begin{lmm}\label{lmmStupiFiniteSet}
Let $r \in \nn$, let $F$ be a finite set, and let $\funcInline{\varphi}{\nn^r}{F}$ be a map of sets. One has that here exists a finite subset $N \subseteq \nn^r$ such that for all $(n_1, \dots n_r) \in \nn^n$ there exists $(m_1, \dots m_r) \in N$ with the property that:
\begin{itemize}
\item
	$\varphi(m_1, \dots m_r) = \varphi(n_1, \dots n_r)$,
\item
	for all $i \in \{1,\dots,r\}$ one has $m_i \leqslant n_i$,
\item
	for all $i \in \{1,\dots,r\}$ such that $n_i > 0$, one has $m_i > 0$.
\end{itemize}
\end{lmm}

\begin{proof}
Omitted.
\end{proof}

\begin{lmm}\label{lmmVectorSpaces}
Let $V$ be a finite-dimensional vector space over a field. Let $\rho$ and $\lambda$ be endomorphisms of $V$ such that there exists $n \in \zpiu$ satisfying the following properties:
\begin{align*}
	\rho^{n+1} &= \rho^n,
	\\
	\ker(\rho^n) \cap \ker(\lambda) &= 0,
	\\
	\rho^n \circ \lambda &= \lambda \circ \rho^n,
\end{align*}
One has that the subspace $\image(\rho) + \image(\lambda)$ is the full $V$.
\end{lmm}

\begin{proof}
Let us choose a complement $V'$ of $\ker(\rho^n)$ in $V$. It suffices to prove that the endomorphism
\[
	\funcNN
		{V = \ker(\rho^n) \oplus V'}
		{V}
		{(x,y)}
		{\lambda(x) + \rho(y)}
\]
is injective (hence surjective). So, let $(x,y)$ be in the kernel. Using the first and third assumptions of the lemma, we obtain
\begin{align*}
	0
	&=
	\rho^n(\lambda(x)) + \rho^{n+1}(y)
	\\&=
	\lambda(\rho^n(x)) + \rho^{n}(y)
	\\&=
	\rho^{n}(y).
\end{align*}
Therefore, $y \in \ker(\rho^n) \cap V'$ and so $y=0$. But now $x \in \ker(\rho^n) \cap \ker(\lambda) = 0$, and so we conclude that also $x=0$.
\end{proof}

We are now able to prove the main result of \S\ref{sectionBimod}.

\begin{prop}\label{propE1fg}
Assume that $\field$ is a finite extension of $\qq_p$.
One has that $E^1$ is finitely generated as an $E^0$-bimodule.
\end{prop}

\begin{proof}
We consider the Weyl chambers of $(\quoz{T}{T^0}) \otimes_{\zz} \rr \cong X_\ast(\TT) \otimes_{\zz} \rr$ (note that in Lemmas \ref{lmmLengthsAddUp} and \ref{lmmgw} we have instead considered the Weyl chambers of $V_{\apart}$).

For all closed Weyl chambers $\overline{\mathfrak{c}}$ in $(\quoz{T}{T^0}) \otimes_{\zz} \rr$ we fix a finite set of generators $x_{\overline{\mathfrak{c}}, 1}, \dots, x_{\overline{\mathfrak{c}}, r}$ of the monoid $(\quoz{T}{T^0}) \cap \overline{\mathfrak{c}}$ (as we can do by Gordan's Lemma; moreover, we choose $r$ to be independent of $\overline{\mathfrak{c}}$, as we can, just to simplify the notation). We also fix representatives $\dot{x}_{\overline{\mathfrak{c}}, 1}, \dots, \dot{x}_{\overline{\mathfrak{c}}, r} \in \widetilde{W}$ for $x_{\overline{\mathfrak{c}}, 1}, \dots, x_{\overline{\mathfrak{c}}, r}$.

Let $w \in \widetilde{W}$. Recall from \eqref{eqPsiw} the map
\[
	\func
		{\Psi_{w}}
		{\big( \quoz{T^1}{(T^1)^p} \big) \oplus \bigoplus_{\alpha \in \Phi} \quoz{\ooo}{p\ooo}}
		{\fratt{I_w}}
		{\big( \overline{t}, (\overline{u_\alpha})_{\alpha \in \Phi} \big)}
		{\overline{t} \cdot \prod\limits_{\alpha \in \Phi} \overline{x_\alpha(\pi^{g_w(\alpha)} u_\alpha)}.}
\]
We argued that, since $\field$ is a finite extension of $\qq_p$, the map $\Psi_{w}$ is a surjective homomorphism of \emph{finite-dimensional} $\ff_p$-vector spaces.
We also consider the induced isomorphism
\[
	\funcInline
		{\overline{\Psi_{w}}}
		{\frac{\big( \quoz{T^1}{(T^1)^p} \big) \oplus \bigoplus_{\alpha \in \Phi} \quoz{\ooo}{p\ooo}}{\ker(\Psi_{w})}}
		{\fratt{I_w}}
\]
and the dual isomorphism
\[
	\funcInline
		{\overline{\Psi_{w}}^\vee}
		{H^1(I,\xx(w))}
		{\homm_{\vect{\ff_p}} \left( \frac{\big( \quoz{T^1}{(T^1)^p} \big) \oplus \bigoplus_{\alpha \in \Phi} \quoz{\ooo}{p\ooo}}{\ker(\Psi_{w})}, k \right)}
\]
(compare \eqref{eqFrattiniFinite}).

Let $\widetilde{W_0}$ denote the preimage of $W_0$ via the quotient map $\funcInlineNN{\widetilde{W}}{W}$ (recall that we view $W_0$ as a subgroup of $W$). For all $w_0 \in \widetilde{W_0}$ and all closed Weyl chambers $\overline{\mathfrak{c}}$ in $(\quoz{T}{T^0}) \otimes_{\zz} \rr$ we consider the function
\[
	\funcNN
		{\nn^r}
		{\left\{ \begin{matrix*}[c]
			\text{sub-$\ff_p$-vector spaces of}
			\\
			\big( \quoz{T^1}{(T^1)^p} \big) \oplus \bigoplus_{\alpha \in \Phi} \quoz{\ooo}{p\ooo}
		\end{matrix*} \right\}}
		{(n_1,\dots,n_r)}
		{\ker \big( \Psi_{w_0 \dot{x}_{\overline{\mathfrak{c}}, 1}^{n_1} \cdots \dot{x}_{\overline{\mathfrak{c}}, r}^{n_r}} \big).}
\]
By Lemma \ref{lmmStupiFiniteSet}, there exists a finite set $N_{w_0, \overline{\mathfrak{c}}}$ with the property that, for all $(n_1, \dots, n_r) \in \nn^r$, there exists $(m_1, \dots, m_r) \in N_{w_0, \overline{\mathfrak{c}}}$ such that:
\begin{align}
	&\text{$\ker \big( \Psi_{w_0 \dot{x}_{\overline{\mathfrak{c}}, 1}^{m_1} \cdots \dot{x}_{\overline{\mathfrak{c}}, r}^{m_r}} \big) = \ker \big( \Psi_{w_0 \dot{x}_{\overline{\mathfrak{c}}, 1}^{n_1} \cdots \dot{x}_{\overline{\mathfrak{c}}, r}^{n_r}} \big)$,}
	\label{cond1}
\\
	&\text{for all $i \in \{1,\dots,r\}$ one has $m_i \leqslant n_i$,}
	\label{cond2}
\\
	&\text{for all $i \in \{1,\dots,r\}$ such that $n_i > 0$, one has $m_i > 0$.}
	\label{cond3}
\end{align}
We consider the following \emph{finite} set:
\[
	\mathcal{G}
	\defeq
	\set{
		w_0 \dot{x}_{\overline{\mathfrak{c}}, 1}^{m_1} \cdots \dot{x}_{\overline{\mathfrak{c}}, r}^{m_r}
	}{
		\begin{matrix*}[l]
			w_0 \in \widetilde{W_0},
			\\
			\overline{\mathfrak{c}} \subseteq (\quoz{T}{T^0}) \otimes_{\zz} \rr \text{ closed Weyl chamber,}
			\\
			(m_1, \dots, m_r) \in N_{w_0, \overline{\mathfrak{c}}}
		\end{matrix*}
	}.
\]
To prove the proposition, it suffices to show that the subspace
\[
	\bigoplus_{v \in \mathcal{G}} H^1(I,\xx(v))
\]
generates $E^1$ as an $E^0$-bimodule, as we said that $\mathcal{G}$ is finite and that $H^1(I,\xx(v))$ is finite-dimensional over $k$.

Let $w \in \widetilde{W}$. We can write it in the form $w = w_0 \dot{x}_{\overline{\mathfrak{c}}, 1}^{n_1} \cdots \dot{x}_{\overline{\mathfrak{c}}, r}^{n_r}$ for some $w_0 \in \widetilde{W_0},$ some closed Weyl chamber $\overline{\mathfrak{c}} \subseteq (\quoz{T}{T^0}) \otimes_{\zz} \rr$ and some $r$-tuple $(n_1, \dots, n_r) \in \nn^r$. Let us pick an $r$-tuple $(m_1, \dots, m_r) \in N_{w_0, \overline{\mathfrak{c}}}$ satisfying properties \eqref{cond1}, \eqref{cond2} and \eqref{cond3}. By \eqref{cond2}, for all $i \in \{1,\dots,r\}$ we have that $n_i-m_i \in \nn$, and from \eqref{cond3} we see that if $n_i-m_i > 0$ then $m_i > 0$. Therefore, Lemma \ref{lmmLengthsAddUp} yields that
\[
	\ell \big( w_0 \dot{x}_{\overline{\mathfrak{c}}, 1}^{n_1} \cdots \dot{x}_{\overline{\mathfrak{c}}, r}^{n_r} \big)
	=
	\ell \big( w_0 \dot{x}_{\overline{\mathfrak{c}}, 1}^{m_1} \cdots \dot{x}_{\overline{\mathfrak{c}}, r}^{m_r} \big)
	+
	\ell \big( \dot{x}_{\overline{\mathfrak{c}}, 1}^{n_1 - m_1} \cdots \dot{x}_{\overline{\mathfrak{c}}, r}^{n_r - m_r} \big).
\]
To ease notation, let us set $v \defeq w_0 \dot{x}_{\overline{\mathfrak{c}}, 1}^{m_1} \cdots \dot{x}_{\overline{\mathfrak{c}}, r}^{m_r}$ and $y \defeq \dot{x}_{\overline{\mathfrak{c}}, 1}^{n_1 - m_1} \cdots \dot{x}_{\overline{\mathfrak{c}}, r}^{n_r - m_r}$. Note that $v \in \mathcal{G}$. With the new notation, the above equation becomes
\begin{equation}\label{eqAddUp1}
	\ell(w) = \ell(vy) = \ell(v) + \ell(y).
\end{equation}
Since $y \in \quoz{T}{T^1}$, we have that $\ell(y) = \ell(vyv^{-1})$ (as can be seen from \eqref{eqLengthFormulaTorus}), and so the above equation can also be rewritten as
\begin{equation}\label{eqAddUp2}
	\ell(w) = \ell(vyv^{-1}v) = \ell(vyv^{-1}) + \ell(v).
\end{equation}

Although we will need this only at a much later point, we also apply Lemma \ref{lmmgw} (whose assumptions are satisfied, again by \eqref{cond2} and \eqref{cond3}). It tells us that for all $\alpha \in \Phi$ at least one of the following two equalities must be true:
\begin{align*}
	g_{w_0 x_1^{n_1} \dots x_r^{n_r}} (\alpha) &= g_{w_0 x_1^{m_1} \dots x_r^{m_r}} (\alpha),
\\
	g_{w_0 x_1^{n_1} \dots x_r^{n_r}} (\alpha) &= g_{w_0 x_1^{m_1} \dots x_r^{m_r}} (\alpha) - \langle \nu(x_1^{n_1-m_1} \dots x_r^{n_r-m_r}), w_0^{-1} \alpha \rangle.
\end{align*}
With the simplified notation we have introduced, these two equalities can be rewritten as
\begin{equation}\label{eqgw}\begin{aligned}
	g_{w} (\alpha) &= g_{v} (\alpha),
\\
	g_{w} (\alpha) &= g_{v} (\alpha) - \langle \nu(y), w_0^{-1} \alpha \rangle = g_{v} (\alpha) - \langle \nu(v y v^{-1}), \alpha \rangle.
\end{aligned}\end{equation}

Now let us consider an element $\beta_v \in H^1(I,\xx(v))$. By \eqref{eqAddUp1} and \eqref{eqAddUp2}, we can apply formulas \eqref{eqActionRightLeftE0IfLengthsAddUp} to compute the following products:
\begin{equation}\label{eqSimpleForm}\begin{aligned}
		\beta_v \cdot \tau_y &\in H^1(I,\xx(vy)) = H^1(I,\xx(w))
		\\
		\sh_{w}(\beta_v \cdot \tau_y) &= \res_{I_{w}}^{I_v} \big( \sh_v(\beta_v) \big).
	\\
		\tau_{vyv^{-1}} \cdot \beta_v &\in H^1(I,\xx(vyv^{-1}v)) = H^1(I,\xx(w))
		\\
		\sh_{w}(\tau_{vyv^{-1}} \cdot \beta_v) &= \res_{I_{w}}^{vyv^{-1} I_v (vyv^{-1})^{-1}} \big( (vyv^{-1})_\ast \sh_v(\beta_v) \big).
\end{aligned}\end{equation}
If we prove that the elements of the form $\beta_v \cdot \tau_y$ and $\tau_{vyv^{-1}} \cdot \beta_v$, for $\beta_v$ running over $H^1(I,\xx(v))$, generate the full $k$-vector space $H^1(I,\xx(w))$, then the thesis of the proposition is proved.

From \eqref{eqSimpleForm} we see that for all $\beta_v \in H^1(I,\xx(v))$ and for all $\alpha \in \Phi$ we have that
\begin{equation}\label{eqbohh}\begin{aligned}
	\restr{\sh_{w}(\beta_v \cdot \tau_y)}{T^1} &= \restr{\sh_v(\beta_v)}{T^1},
	\\
	\restr{\sh_{w}(\beta_v \cdot \tau_y)}{\uuu_{\alpha,g_w(\alpha)}} &= \restr{\sh_v(\beta_v)}{\uuu_{\alpha,g_w(\alpha)}},
\\
	\restr{\sh_{w}(\tau_{vyv^{-1}} \cdot \beta_v)}{T^1}
	&=
	\restr{\sh_v(\beta_v) \big( (vyv^{-1})^{-1} \cdot {}_- \cdot vyv^{-1} \big)}{T^1}
	\\&=
	\restr{\sh_v(\beta_v)}{T^1},
	\\
	\restr{\sh_{w}(\tau_{vyv^{-1}} \cdot \beta_v)}{\uuu_{\alpha,g_w(\alpha)}}
	&=
	\restr{\sh_v(\beta_v) \big( (vyv^{-1})^{-1} \cdot {}_- \cdot vyv^{-1} \big)}{\uuu_{\alpha,g_w(\alpha)}}
	\\&=
	\restr{\sh_v(\beta_v) \big( \alpha(vyv^{-1})^{-1} \cdot {}_- \big)}{\uuu_{\alpha,g_w(\alpha)}}.
\end{aligned}\end{equation}
To ease notation, we set $({}_-)^\vee \defeq \homm_{\vect{\ff_p}}({}_-, k)$ and $K \defeq \ker(\Psi_v) = \ker(\Psi_w)$. We consider the following picture:
\[\begin{tikzcd}[column sep = 1.8em, row sep = 3em]
	H^1(I,\xx(v))
	\ar[d, "{\overline{\Psi_{v}}^\vee}", "\cong"']
	\ar[r, shift left, "{{}_- \cdot \tau_{y}}"]
	\ar[r, shift right, "{\tau_{vyv^{-1}} \cdot {}_-}"']
	&
	H^1(I,\xx(w))
	\ar[d, "{\overline{\Psi_{w}}^\vee}", "\cong"']
	\\
	\left( \frac{( \quoz{T^1}{(T^1)^p} ) \oplus \bigoplus_{\alpha \in \Phi} \quoz{\ooo}{p\ooo}}{K} \right)^\vee
	\ar[r, shift left, dashed, "{\rho}"]
	\ar[r, shift right, dashed, "{\lambda}"']
	&
	\left( \frac{( \quoz{T^1}{(T^1)^p} ) \oplus \bigoplus_{\alpha \in \Phi} \quoz{\ooo}{p\ooo}}{K} \right)^\vee.
\end{tikzcd}\]
We want to describe explicitly the map $\rho$ (respectively, the map $\lambda$) making the square with the map ${}_- \cdot \tau_{y}$ (respectively, with the map $\tau_{vyv^{-1}} \cdot {}_-$) commute. Once we do this, to prove the proposition it remains to prove that the subspace $\image(\rho) + \image(\lambda)$ is the full space. Let us consider
\[
	\gamma \in \left( \frac{( \quoz{T^1}{(T^1)^p} ) \oplus \bigoplus_{\alpha \in \Phi} \quoz{\ooo}{p\ooo}}{K} \right)^\vee,
\]
and let us compute $\rho(\gamma)$ and $\lambda(\gamma)$. To ease notation, we set $\beta_v \defeq \big( \overline{\Psi_{v}}^\vee \big)^{-1}(\gamma)$, getting that
\begin{equation}\label{eqMultExpli}\begin{aligned}
	\gamma &= \overline{\Psi_{v}}^\vee (\beta_v) = \overline{\sh_v(\beta_v)} \circ \overline{\Psi_{v}},
	\\
	\rho(\gamma) &= \overline{\Psi_{w}}^\vee (\beta_v \cdot \tau_y) = \overline{\sh_w (\beta_v \cdot \tau_y)} \circ \overline{\Psi_{w}},
	\\
	\lambda(\gamma) &= \overline{\Psi_{w}}^\vee (\tau_{vyv^{-1}} \cdot \beta_v) = \overline{\sh_w (\tau_{vyv^{-1}} \cdot \beta_v)} \circ \overline{\Psi_{w}}.
\end{aligned}\end{equation}
Here, by $\overline{\sh_v(\beta_v)}$ we mean the element of $(\fratt{I_v})^\vee$ corresponding to the element $\sh_v(\beta_v) \in H^1(I_v,k)$ (and similarly in the other two lines).

Let $t \in T^1$ and let $\overline{t}$ be its image in $\frac{( \quoz{T^1}{(T^1)^p} ) \oplus \bigoplus_{\alpha \in \Phi} \quoz{\ooo}{p\ooo}}{K}$. By \eqref{eqMultExpli} we have
\begin{align*}
	\gamma(\overline{t}) & = \sh_v(\beta_v) (t),
	\\
	\rho(\gamma)(\overline{t}) & = \sh_w (\beta_v \cdot \tau_y) (t),
	\\
	\lambda(\gamma)(\overline{t}) & = \sh_w (\tau_{vyv^{-1}} \cdot \beta_v)(t).
\end{align*}
From \eqref{eqbohh} we then see that
\begin{equation}\begin{aligned}
	\rho(\gamma)(\overline{t}) &= \gamma(\overline{t}),
	\\
	\lambda(\gamma)(\overline{t}) &= \gamma(\overline{t}).
\end{aligned}\end{equation}

Now let $\alpha \in \Phi$ and let $u_\alpha \in \ooo$. Let us denote by $\overline{u_\alpha}$ be the image of $u_\alpha$ in $\frac{( \quoz{T^1}{(T^1)^p} ) \oplus \bigoplus_{\alpha' \in \Phi} \quoz{\ooo}{p\ooo}}{K}$, when $\quoz{\ooo}{p\ooo}$ is embedded in $\bigoplus_{\alpha' \in \Phi} \quoz{\ooo}{p\ooo}$ via the inclusion on the summand indexed by $\alpha$. By \eqref{eqMultExpli} we have
\begin{align*}
	\gamma(\overline{u_\alpha}) & = \sh_v(\beta_v) \big( x_\alpha(\pi^{g_v(\alpha)} u_\alpha) \big),
	\\
	\rho(\gamma)(\overline{u_\alpha}) & = \sh_w (\beta_v \cdot \tau_y) \big( x_\alpha(\pi^{g_w(\alpha)} u_\alpha) \big),
	\\
	\lambda(\gamma)(\overline{u_\alpha}) & = \sh_w (\tau_{vyv^{-1}} \cdot \beta_v) \big( x_\alpha(\pi^{g_w(\alpha)} u_\alpha) \big).
\end{align*}
Combining this with \eqref{eqbohh}, we then see that
\begin{equation}\begin{aligned}
	\rho(\gamma)(\overline{u_\alpha})
	&=
	\sh_v (\beta_v) \big( x_\alpha(\pi^{g_w(\alpha)} u_\alpha) \big)
	\\&=
	\gamma \big( \pi^{g_w(\alpha)-g_v(\alpha)} \overline{u_\alpha} \big),
	\\
	\lambda(\gamma)(\overline{u_\alpha})
	&=
	\sh_v(\beta_v) \big( \alpha(vyv^{-1})^{-1} \cdot x_\alpha(\pi^{g_w(\alpha)} u_\alpha) \big)
	\\&=
	\gamma \big( \alpha(vyv^{-1})^{-1} \pi^{g_w(\alpha)-g_v(\alpha)} \overline{u_\alpha} \big).
\end{aligned}\end{equation}

For all $\alpha \in \Phi$, we set $r_\alpha \defeq \pi^{g_w(\alpha)-g_v(\alpha)}$ and $l_\alpha \defeq \alpha(vyv^{-1})^{-1} \pi^{g_w(\alpha)-g_v(\alpha)}$. We have obtained that the following diagrams are commutative:
\[\begin{tikzcd}[column sep = 8em, row sep = 3em]
	\left( \frac{( \quoz{T^1}{(T^1)^p} ) \times \prod_{\alpha \in \Phi} \quoz{\ooo}{p\ooo}}{K} \right)^\vee
	\ar[r, shift left, "{\rho}"]
	\ar[r, shift right, "{\lambda}"']
	\ar[d, hook]
	&
	\left( \frac{( \quoz{T^1}{(T^1)^p} ) \times \prod_{\alpha \in \Phi} \quoz{\ooo}{p\ooo}}{K} \right)^\vee
	\ar[d, hook]
	\\
	( \quoz{T^1}{(T^1)^p} )^\vee \oplus \bigoplus\limits_{\alpha \in \Phi} (\quoz{\ooo}{p\ooo})^\vee
	\ar[r, shift left, "{\tilde{\rho} \defeq \id \oplus \bigoplus\limits_{\alpha \in \Phi} (r_\alpha \cdot {}_-)^\vee}"]
	\ar[r, shift right, "{\tilde{\lambda} \defeq \id \oplus \bigoplus\limits_{\alpha \in \Phi} (l_\alpha \cdot {}_-)^\vee}"']
	&
	( \quoz{T^1}{(T^1)^p} )^\vee \oplus \bigoplus\limits_{\alpha \in \Phi} (\quoz{\ooo}{p\ooo})^\vee
\end{tikzcd}\]
If we check that there exists $n \in \zpiu$ such that
\begin{align*}
	\rho^{n+1} &= \rho^n,
	\\
	\ker(\rho^n) \cap \ker(\lambda) &= 0,
	\\
	\rho^n \circ \lambda &= \lambda \circ \rho^n,
\end{align*}
then the assumptions of Lemma \ref{lmmVectorSpaces} are satisfied and so $\image(\rho) + \image(\lambda)$ is the full space $\left( \frac{( \quoz{T^1}{(T^1)^p} ) \times \prod_{\alpha \in \Phi} \quoz{\ooo}{p\ooo}}{K} \right)^\vee$, as we wanted to show to conclude the proof of the proposition. But if we check that there exists $n \in \zpiu$ satisfying the above properties for $\tilde{\rho}$ in place of $\rho$ and for $\tilde{\lambda}$ in place of $\lambda$, then a fortiori the above properties remain true for $\rho$ and $\lambda$.

The property $\tilde{\rho}^n \circ \tilde{\lambda} = \tilde{\lambda} \circ \tilde{\rho}^n$ is completely trivial for all $n \in \zpiu$. The property $\tilde{\rho}^{n+1} = \tilde{\rho}^n$ is true for $n$ big enough because, for all $\alpha \in \Phi$, the constant $r_\alpha$ is either $1$ or a positive power of $\pi$, from which it follows that the sequence $(r_\alpha^n)_{n \in \zpiu}$ is either definitely $1$ or $0$ in $\quoz{\ooo}{p\ooo}$.

It remains to check the property $\ker(\tilde{\rho}^n) \cap \ker(\tilde{\lambda}) = 0$ (we will actually check it for all $n \in \zpiu$). Let us consider an arbitrary element
\[
	f
	=
	(t^\vee, (u_\alpha^\vee)_{\alpha \in \Phi})
	\in
	( \quoz{T^1}{(T^1)^p} )^\vee \oplus \bigoplus\limits_{\alpha \in \Phi} (\quoz{\ooo}{p\ooo})^\vee;
\]
we have to check either $\tilde{\rho}^n(f)$ or $\tilde{\lambda}(f)$ is non-zero. If $t^\vee$ is non-zero, then actually both $\tilde{\rho}^n(f)$ and $\tilde{\lambda}(f)$ are non-zero. So we can now assume that $u_\alpha^\vee$ is nonzero for some $\alpha \in \Phi$. Recall from \eqref{eqgw} that at least one of the following equalities is true:
\begin{align*}
	g_{w} (\alpha) &= g_{v} (\alpha),
\\
	g_{w} (\alpha) &= g_{v} (\alpha) - \langle \nu(v y v^{-1}), \alpha \rangle.
\end{align*}
The first equality means that $r_\alpha = 1$ (and hence that $r_\alpha^n = 1$), while the second equality means that $l_\alpha \in \ooo^\times$, because
\begin{align*}
	\val_\field(l_\alpha)
	&=
	\val_\field \big( \alpha(vyv^{-1})^{-1} \big) + g_w(\alpha) - g_v(\alpha)
	\\&=
	\langle \nu(vyv^{-1}), \alpha \rangle + g_w(\alpha) - g_v(\alpha).
\end{align*}
In particular, we get that either $\tilde{\rho}^n(f)$ or $\tilde{\lambda}(f)$ is non-zero.
\end{proof}

\begin{es}\label{esLaurent}
Assume that $\field = \ff_q((X))$ and that $\GG = \gm$. We show that $E^1$ is not finitely generated as an $E^0$-bimodule.
\end{es}

\begin{proof}
Since $\GG = \gm$, the length function on $\widetilde{W}$ is constantly zero and any conjugation by an element of $G$ is trivial. Therefore, looking at the formulas \eqref{eqCupYoneda} and \eqref{eqActionRightLeftE0IfLengthsAddUp} for the multiplication on $E^\ast$ we see that we have an isomorphism of graded $k$-algebras
\[
	E^\ast \cong E^0 \otimes_k H^\ast(I,k),
\]
where the multiplication on $H^\ast(I,k)$ is the cup product. Since left and right action of $E^0$ are the same, saying that $E^1$ is finitely generated as an $E^0$-bimodule is the same as saying that $H^1(I,k)$ is finitely generated as a $k$-vector space. This is not the case: indeed, since $\field = \ff_q((X))$, the pro-$p$ group $1+\mmm$ is isomorphic to $\prod_{i \in \nn} \zz_p$ (see \cite[Chapter I, (6.2), Proposition]{FesVost}), and therefore
\[
	H^1(I,k)
	=
	H^1(1+\mmm,k)
	\\\cong
	\textstyle H^1 \big( \prod_{i \in \nn} \zz_p,k \big),
\]
which has infinite dimension as a $k$-vector space.
\end{proof}

\section{The subalgebra generated by \texorpdfstring{$E^0$}{E\textasciicircum{}0} and \texorpdfstring{$E^1$}{E\textasciicircum{}1} in some special cases}\label{sectionSubalgE0E1}

The purpose of this section is to prove Proposition \ref{propAlmostGenByE1}, i.e., to prove that, when $G$ is \virgolette{small}, the multiplication map
\[
	\funcInlineNN{T^\ast_{E^0} E^1}{E^\ast}
\]
is \virgolette{almost surjective} (for the precise statement see the proposition as well as Remark \ref{remFiniteCodim}).
To prove the proposition, we need to specialize the setting introduced in Section \ref{sectionSettingBackground} to the case that $\GG$ has semisimple rank $1$.

Let us assume that $\GG$ has semisimple rank $1$, or, equivalently, that the root system is of type $A_1$. Let $\alpha_0$ denote the positive root. Recalling the notation and facts on the Chevalley system reviewed in \S\ref{subsectionApartment} and on the structure of the group $W$ in \S\ref{subsecAffineRootsLength}, we define
\begin{align}\label{eqDefs0s1Gneneral}
	&s_0 \defeq \overline{\varphi_{\alpha_0} \matr{0}{1}{-1}{0}},
	&&s_1 \defeq \overline{\varphi_{-\alpha_0} \matr{0}{\pi}{-\pi^{-1}}{0}},
\end{align}
where $\overline{?}$ denotes the image of $?$ in $\widetilde{W}$. Strictly speaking this definition is not compatible with the one we will give at the beginning of \S\ref{sectionFinPres} (see there for details), but this will not cause problems.

With this definition, we have $S_{\aff} = \{\overline{s_0}, \overline{s_1}\}$, where $\overline{s_i}$ denotes the image of $s_i$ in $W$ and every element of $W_{\aff}$ can be expressed as a string of the elements $\overline{s_0}$ and $\overline{s_1}$, with no two identical letters one after the other. This representation is actually unique and we have
\begin{align*}
	\ell((\overline{s_0} \cdot \overline{s_1})^i) &= i,
	&
	\ell((\overline{s_1} \cdot \overline{s_0})^i) &= i,
	\\
	\ell(\overline{s_0} \cdot (\overline{s_1} \cdot \overline{s_0})^i) &= i,
	&
	\ell(\overline{s_1} \cdot (\overline{s_0} \cdot \overline{s_1})^i) &= i,
\end{align*}
for all $i \in \nn$ (compare \cite[\S3.1]{emb} in the case $\GG = \SL_2$). Note that in the above list the string representation of the elements of $W$ is not really unique, the trivial element being repeated twice.

We also deduce that every element of $W$ can be represented in a unique way as a product of an element of $\Omega$ and a string of the elements $\overline{s_0}$ and $\overline{s_1}$, without repetitions as above, and every element of $\widetilde{W}$ can be represented in a unique way as a product of an element of $\widetilde{\Omega}$ and a string of the elements ${s_0}$ and ${s_1}$, with no two identical letters one after the other.

\begin{lmm}\label{lmmSgpA1}
Assume that $\GG$ has semisimple rank $1$. With notation as above, for all $w \in \widetilde{W}$ the following facts hold.
\begin{itemize}
\item 
	If $\ell(s_0 w)= \ell(w) + 1$, then the multiplication map defines a homeomorphism
	\[
		\funcAbove{\cong}
			{
				x_{-\alpha_0} (\mmm^{\ell(w)+1})
				\times
				T^1
				\times
				x_{\alpha_0}  (\ooo)
			}
			{I_w.}
	\]
\item 
	If $\ell(s_1 w)= \ell(w) + 1$, then the multiplication map defines a homeomorphism
	\[
		\funcAbove{\cong}
			{
				x_{-\alpha_0} (\mmm)
				\times
				T^1
				\times
				x_{\alpha_0}  (\mmm^{\ell(w)})
			}
			{I_w.}
	\]
\end{itemize}
\end{lmm}

\begin{proof}
This is a variation of \cite[Equations (25) and (26)]{emb}, which only deals with the case $\GG = \SL_2$. Recall from \S\ref{subsecIw} that $I_w$ only depends on the image of $w$ in $W$, so we consider $w \in W$ rather than $w \in \widetilde{W}$. Moreover, for all $v \in W_{\aff}$ and all $\omega \in \Omega$ one has $I_{v\omega} = I_v$ (see \cite[\S2.1.5]{ext}), and the condition $\ell(\overline{s_i}v\omega) = \ell(v\omega) + 1$ is satisfied if and only if the condition $\ell(\overline{s_i}v) = \ell(v) + 1$ is. Therefore, by the above argument, one only has to consider the cases $w = ({s_0} \cdot {s_1})^i)$, $w = ({s_1} \cdot {s_0})^i$, $w = {s_0} \cdot ({s_1} \cdot {s_0})^i$, $w = {s_1} \cdot ({s_0} \cdot {s_1})^i)$ for some $i \in \nn$, and those cases can be dealt with explicitly as in \cite[Equations (25) and (26)]{emb}.
\end{proof}

\begin{lmm}\label{lmmCommutatorSgpA1}
Assume that $p \neq 2$ and that $\GG$ has semisimple rank $1$. With notation as above, for all $w \in \widetilde{W}$ the following facts hold.
\begin{itemize}
\item 
	If $\ell(s_0 w)= \ell(w) + 1$, then the following equalities hold:
	\begin{align*}
		\overline{[I_w,I_w]}
		&=
		[I_w,I_w]
		\\&=
		\varphi_{\alpha_0}
		\left(
			\matr{1+\mmm^{\ell(w)+1}}{\mmm}{\mmm^{\ell(w)+2}}{1+\mmm^{\ell(w)+1}}
			\cap
			\SL_2(\field)
		\right)
		\\&=
			x_{-\alpha_0} (\mmm^{\ell(w)+2})
			\cdot
			\check{\alpha}_0(1+\mmm^{\ell(w)+1}) 
			\cdot
			x_{\alpha_0}  (\mmm).
	\end{align*}
\item 
	If $\ell(s_1 w)= \ell(w) + 1$, then the following equalities hold:
	\begin{align*}
		\overline{[I_w,I_w]}
		&=
		[I_w,I_w]
		\\&=
		\varphi_{\alpha_0}
		\left(
			\matr{1+\mmm^{\ell(w)+1}}{\mmm^{\ell(w)+1}}{\mmm^{2}}{1+\mmm^{\ell(w)+1}}
			\cap
			\SL_2(\field)
		\right)
		\\&=
			x_{-\alpha_0} (\mmm^{2})
			\cdot
			\check{\alpha}_0(1+\mmm^{\ell(w)+1}) 
			\cdot
			x_{\alpha_0}  (\mmm^{\ell(w)+1}).
	\end{align*}
\end{itemize}
\end{lmm}

\begin{proof}
This is a straightforward variation of \cite[Proposition 3.62]{emb}, which only deals with the case $\GG = \SL_2$.
\end{proof}

\begin{lmm}\label{lmmFrattiniQuot}
Assume that $p \neq 2$ and that $\GG$ has semisimple rank $1$. With notation as above, for all $w \in \widetilde{W}$ the following facts hold.
\begin{itemize}
\item 
	If $\ell(s_0 w)= \ell(w) + 1$, then one has the group isomorphism
	\[
		\begin{tikzcd}[row sep = 0pt, column sep = 0.5em]
			{
				\quoz{\ooo}{\mmm}
				\times
				\textstyle \frac{T^1}{\check{\alpha}_0(1+\mmm^{\ell(w)+1}) \cdot (T^1)^p}
				\times
				\quoz{\ooo}{\mmm}
			}
			\ar[r]
			&
			{\fratt{I_w}}
			\\
			{\left(\overline{c}, \overline{t}, \overline{b}\right)}
			\ar[r, mapsto]
			&
			{\textstyle \overline{
				x_{-\alpha_0} (-\pi^{\ell(w)+1} c)
				\cdot
				t
				\cdot
				x_{\alpha_0}  (b)
			}.}
		\end{tikzcd}
	\]
\item 
	If $\ell(s_1 w)= \ell(w) + 1$, then one has the group isomorphism
	\[
		\begin{tikzcd}[row sep = 0pt, column sep = 0.5em]
			{
				\quoz{\ooo}{\mmm}
				\times
				\textstyle \frac{T^1}{\check{\alpha}_0(1+\mmm^{\ell(w)+1}) \cdot (T^1)^p}
				\times
				\quoz{\ooo}{\mmm}
			}
			\ar[r]
			&
			{\fratt{I_w}}
			\\
			{\left(\overline{c}, \overline{t}, \overline{b}\right)}
			\ar[r, mapsto]
			&
			{\textstyle \overline{
				x_{-\alpha_0} (-\pi c)
				\cdot
				t
				\cdot
				x_{\alpha_0} (\pi^{\ell(w)} b)
			}.}
		\end{tikzcd}
	\]
\end{itemize}
\end{lmm}

\begin{proof}
This follows from the previous lemma. The minus signs in \virgolette{$-\pi^{\ell(w)+1} c$} and \virgolette{$-\pi c$} are of course completely inessential but we keep them as they seem more natural (for example, this choice is compatible with \cite[Proposition 3.62.ii]{emb} and $x_{-\alpha_0} (-u) = \varphi_{\alpha_0} \matr{1}{0}{u}{1}$).
\end{proof}

\begin{rem}\label{remIsoParentesi}
Assume that $p \neq 2$ and that $\GG$ has semisimple rank $1$. For all $w \in \widetilde{W}$ let us define
\[
	T^1_w
	\defeq
	\frac{T^1}{\check{\alpha}_0(1+\mmm^{\ell(w)+1}) \cdot (T^1)^p}.
\]
Fixing the isomorphism
\[
	\quoz{\ooo}{\mmm} \times T^1_w \times \quoz{\ooo}{\mmm} \cong \fratt{I_w}
\]
defined in the previous lemma, and using the notation $({}_-)^\vee \defeq \homm_{\vect{\ff_p}}({}_-, k)$, we have an induced isomorphism
\begin{equation}\label{eqIsoE1AsATriple}
	\funcInline
		{({}_-, {}_-, {}_-)_w}
		{(\ooo/\mmm)^\vee \oplus (T^1_w)^\vee \oplus (\ooo/\mmm)^\vee}
		{H^1(I,\xx(w)).}
\end{equation}
\end{rem}

For the next lemma, recall from \S\ref{subsecCohom} the definition of Poincaré pro-$p$ group and of uniform pro-$p$ group.

\begin{lmm}\label{lmmPoincAndUniform}
Assume that $p \neq 2,3$, that $\field$ is an unramified extension of $\qq_p$ and that $\GG$ has semisimple rank $1$. One has that $I$ is torsion-free (and so $I$, as well as all the subgroup $I_w$ with $w \in \widetilde{W}$ are Poincaré pro-$p$ groups) and that $I_w$ is a uniform pro-$p$ group for all $w \in \widetilde{W}$ such that $\ell(w) \geqslant 1$.
\end{lmm}

\begin{proof}
We claim that $G$ is $p$-torsion-free and that $I_w$ is a powerful pro-$p$ group for all $w \in \widetilde{W}$ such that $\ell(w) \geqslant 1$. Taking these two claims for granted, the two theses of the lemma follow: indeed $I$ is $p$-torsion-free and hence torsion-free, and furthermore, if $\ell(w) \geqslant 1$, saying that the pro-$p$ group $I_w$ is uniform amounts to saying that it is topologically finitely generated, powerful and torsion-free. We have the last two properties, and the fact that it is topologically finitely generated follows from the fact $\ooo$ and $1+\mmm$ are topologically finitely generated (since $\field$ is an extension of $\qq_p$) and from the Iwahori decomposition of $I_w$.

Let us consider the derived group $\GG'$ of $\GG$ (which must be isomorphic either to $\SL_2$ or to $\PGL_2$) and the semisimple cover $\SL_2$ of $\GG'$: we have central isogenies
\[\begin{tikzcd}
	\gm^{\dim \CC^\circ} \times \SL_2
	\cong
	\CC^\circ \times \SL_2
	\ar[r]
	&
	\CC^\circ \times \GG'
	\ar[r]
	&
	\GG.
\end{tikzcd}\]
Let $\mmu$ be the kernel of the composite central isogeny; it must be isomorphic a central subgroup of $\SL_2$: i.e., either to the trivial group or to the group of $2$-units $\mmu_2$. Considering the exact sequence induced by Galois cohomology, we obtain an exact sequence
\[\begin{tikzcd}
	\mmu(\field)
	\ar[r]
	&
	(\field^\times)^{\dim \CC^\circ} \times \SL_2(\field)
	\ar[r]
	&
	G
	\ar[r]
	&
	H^1(\gal(\overline{\qq_p}/\field), \mmu(\overline{\qq_p})).
\end{tikzcd}\]
The group $(\field^\times)^{\dim \CC^\circ} \times \SL_2(\field)$ does not have $p$-torsion because $\field$ is an extension of $\qq_p$ without nontrivial $p$-th roots of unity and because $p \neq 2,3$ (an element of $\SL_2(\field)$ of order $p$ would have characteristic polynomial divisible by the $p$-th cyclotomic polynomial, which is absurd). The groups $\mmu(\field)$ and $H^1(\gal(\overline{\qq_p}/\field), \mmu(\overline{\qq_p}))$ are (finite) $2$-groups. It then follows formally that $G$ is $p$-torsion-free: indeed, assume by contradiction that we have a non-trivial $p$-torsion element $x \in G$: it lifts to an element of $y \in (\field^\times)^{\dim \CC^\circ} \times \SL_2(\field)$ since $H^1(\gal(\overline{\qq_p}/\field), \mmu(\overline{\qq_p}))$ is $p$-torsion free. The element $y^p$ must lift to an element $z \in \mmu(\field)$. Since $z^{2^r} = 1$ for some $r \in \zpiu$, it follows that $y^{2^r}$ is a $p$-torsion element in $(\field^\times)^{\dim \CC^\circ} \times \SL_2(\field)$. It cannot be trivial, because it maps to $x^{2^r}$.

Now it remains to prove that $I_w$ is a powerful pro-$p$ group for all $w \in \widetilde{W}$ such that $\ell(w) \geqslant 1$. We refer to the description of $I_w$ and of $\overline{[I_w,I_w]}$ we gave, respectively, in Lemmas \ref{lmmSgpA1} \ref{lmmCommutatorSgpA1} and for brevity we only consider the case $\ell(s_0 w)= \ell(w) + 1$. It is enough to check that the elements of the form $\varphi_{\alpha_0} \matr{1}{0}{\mmm^{\ell(w)+2}}{1}$, of the form $\check{\alpha}_0(1+\mmm^{\ell(w)+1})$ and of the form $\varphi_{\alpha_0}  \matr{1}{\mmm}{0}{1}$ are contained in the subgroup topologically generated by the $p$-powers (we will actually show that all such elements are $p$-powers). Since $\field$ is an unramified extension of $\qq_p$, we have that
\[
	\varphi_{\alpha_0} \matr{1}{0}{\mmm^{\ell(w)+2}}{1}
	=
	\left( \varphi_{\alpha_0} \matr{1}{0}{\mmm^{\ell(w)+1}}{1} \right)^p,
\]
and so we see that the elements in $\varphi_{\alpha_0} \matr{1}{0}{\mmm^{\ell(w)+2}}{1}$ are $p$-powers of elements in $I_w$. The same argument shows that the elements in $\varphi_{\alpha_0}  \matr{1}{\mmm}{0}{1}$ are $p$-powers of elements in $I_w$. It remains to work with $\check{\alpha}_0(1+\mmm^{\ell(w)+1})$. Since $\field$ is an unramified extension of $\qq_p$ with $p \neq 2$, the $p$-adic logarithm induces an isomorphism $1+\mmm \cong \mmm$, under which $1+\mmm^2$ corresponds to $\mmm^2$ (which is equal to $p \mmm$, using again that $\field$ is unramified). Therefore the elements in $1+\mmm^2$ are $p$-powers of elements in $1+\mmm$ and so, since $\ell(w) \geqslant 1$, the elements in $\check{\alpha}_0(1+\mmm^{\ell(w)+1})$ are $p$-powers of elements in $\check{\alpha}_0(1+\mmm) \subseteq I_w$. This concludes the proof that $I_w$ is a powerful pro-$p$ group.
\end{proof}

\begin{rem}\label{remMultipleCupprod}
Let $n \in \zpiu$ and let $w_1, \dots, w_n \in \widetilde{W}$ with the propriety that $\ell(w_1 \cdots w_n) = \ell(w_1) + \dots + \ell(w_n)$. For all $i \in \{1,\dots,n\}$ let $j_i \in \nn$ and let $\beta_i \in H^{j_i}(I,\xx(w_i))$. The following formula holds:
\[
	\beta_1 \cdots \beta_n
	=
	(\beta_1 \cdot \tau_{w_2 \cdots w_n})
	\cupprod
	\dots
	\cupprod
	(\tau_{w_1 \cdots w_{i-1}} \cdot \beta_i \cdot \tau_{w_{i+1} \cdots w_n})
	\cupprod
	\dots
	\cupprod
	(\tau_{w_1 \cdots w_{n-1}} \cdot \beta_n).
\]
\end{rem}

\begin{proof}
Applying the formula relating (the opposite of) the Yoneda product with the cup product \eqref{eqCupYoneda}, we have
\[
	\beta_1 \cdots \beta_n
	=
	(\beta_1 \cdots \beta_{n-1} \cdot \tau_{w_n})
	\cupprod
	(\tau_{w_1 \cdots w_{n-1}} \cdot \beta_n).
\]
So, by induction on $n$, it suffices to check that the map
\[
	\funcInline
		{{}_- \cdot \tau_{w_n}}
		{H^\ast(I,\xx(w_1 \cdots w_{n-1}))}
		{H^\ast(I,\xx(w_1 \cdots w_{n}))}
\]
commutes with cup products. Recall from \eqref{eqActionRightLeftE0IfLengthsAddUp} that we have a commutative diagram
\[\begin{tikzcd}[column sep = large, row sep = large]
	H^\ast(I,\xx(w_1 \cdots w_{n-1}))
	\ar[r, "{{}_- \cdot \tau_{w_n}}"]
	\ar[d, "\sh_{w_1 \cdots w_{n-1}}", "\cong"']
	&
	H^\ast(I,\xx(w_1 \cdots w_{n}))
	\ar[d, "\sh_{w_1 \cdots w_{n}}", "\cong"']
	\\
	H^\ast(I_{w_1 \cdots w_{n-1}}, k)
	\ar[r, "{\res^{I_{w_1 \cdots w_{n-1}}}_{I_{w_1 \cdots w_{n}}}}"]
	&
	H^\ast(I_{w_1 \cdots w_{n}}, k),
\end{tikzcd}\]
and the claim follows since restriction and Shapiro isomorphism commute with cup products.
\end{proof}

\begin{lmm}\label{lmmExplicitFormulasE1}
Assume that $p \neq 2$ and that $\GG$ has semisimple rank $1$. Let $w \in \widetilde{W}$, let $c^-, c^+ \in \homm_{\vect{\ff_p}}(\quoz{\ooo}{\mmm}, k)$ and let $c^0 \in \homm_{\vect{\ff_p}}(T^1_w,k)$. With notation as at the beginning of this subsection and as in Remark \ref{remIsoParentesi}, the following formulas hold.
\begin{enumerate}[label=(\roman*)]
\item \label{item1ExplFormulasE1}
	If $v \in \widetilde{W}$ is such that $\ell(v) \geqslant 1$ and $\ell(s_0 w v) = 1 + \ell(w) + \ell(v)$, one has
	\[
		\big( c^-,c^0,c^+ \big)_w \cdot \tau_v
		=
		\big(0,c^0(\pr({}_-)),c^+\big)_{wv},
	\]
	where $\funcInline{\pr}{T^1_{wv}}{T^1_{w}}$ is the quotient map.
\item \label{item2ExplFormulasE1}
	If $\ell(s_0w) = \ell(w) + 1$, one has
	\[
		\tau_{s_0} \cdot \big( c^-,c^0,c^+ \big)_w
		=
		\big( 0,c^0(w_0^{-1} \pr({}_-) w_0),-c^- \big)_{s_0w},
	\]
	where $\funcInline{\pr}{{T^1_{s_0w}}}{{T^1_w}}$ is the quotient map, and where $w_0$ is any element of $N \smallsetminus T$.
\item \label{item3ExplFormulasE1}
	If $v \in \widetilde{W}$ is such that $\ell(v) \geqslant 1$ and $\ell(s_1 w v) = 1 + \ell(w) + \ell(v)$, one has
	\[
		\big(c^-,c^0,c^+\big)_w \cdot \tau_v
		=
		\big(c^-,c^0(\pr({}_-)),0\big)_{wv},
	\]
	where $\funcInline{\pr}{{T^1_{wv}}}{{T^1_{w}}}$ is the quotient map.
\item \label{item4ExplFormulasE1}
	If $\ell(s_1w) = \ell(w) + 1$, one has
	\[
		\tau_{s_1} \cdot \big(c^-,c^0,c^+\big)_w
		=
		\big(-c^+,c^0(w_0^{-1} \pr({}_-) w_0),0\big)_{s_1w},
	\]
	where $\funcInline{\pr}{{T^1_{s_1w}}}{{T^1_w}}$ is the quotient map, and where $w_0$ is any element of $N \smallsetminus T$.
\end{enumerate}
\end{lmm}

\begin{proof}
We prove the four statements using the description of the first cohomology of finitely generated pro-$p$ groups as
\[
	H^1({}_-, k)
	\cong
	\homm_{\vect{\ff_p}} \big( \fratt{{}_-}, k \big).
\]
and the description of the relevant Frattini quotients given in Lemma \ref{lmmFrattiniQuot}.
\begin{enumerate}
\item[\ref{item1ExplFormulasE1}]
	Let us assume that $\ell(v) \geqslant 1$ and $\ell(s_0 w v) = 1 + \ell(w) + \ell(v)$ (in particular, if follows that $\ell(wv) = \ell(w) + \ell(v)$, that $\ell(s_0wv) = \ell(wv) + 1$ and that $\ell(s_0w) = \ell(w) + 1$) and let us prove the first formula.
	
	We want to compute $\beta \cdot \tau_v$ where $\beta \defeq \big( c^-,c^0,c^+ \big)_w$. Since lengths add up, we can apply the formula \eqref{eqActionRightLeftE0IfLengthsAddUp}, obtaining
	\begin{align*}
		\beta \cdot \tau_v &\in H^1(I,\xx(wv))
		&&\text{ and }
		&\sh_{wv}(\alpha \cdot \tau_v) &= \res_{I_{wv}}^{I_w} \big( \sh_w(\beta) \big).
	\end{align*}
	
	Therefore, we have to compute the action of the restriction on cohomology, which corresponds to computing the map induced by the inclusion on the Frattini quotients. Such action can be computed via the following commutative diagram (where we use the description of the Frattini quotients given in Lemma \ref{lmmFrattiniQuot}: note that the conditions $\ell(s_0wv) = \ell(wv) + 1$ and $\ell(s_0w) = \ell(w) + 1$ tell us which of the descriptions in Lemma \ref{lmmFrattiniQuot} we have to consider; also note that the condition $\ell(v) \geqslant 1$ is used to say that $\left(\overline{c}, 0, 0\right)$ is mapped to zero by the map on the left):
	\[\begin{tikzcd}[row sep = 3.5em, column sep = 16em]
		{
			\quoz{\ooo}{\mmm}
			\times
			{T^1_{wv}}
			\times
			\quoz{\ooo}{\mmm}
		}
		\ar[r, "{
			\left(\overline{c}, \overline{t}, \overline{b}\right)
			\mapsto
			\overline{
				\varphi_{\alpha_0} \matr{1}{0}{\pi^{\ell(wv)+1} c}{1}
				\cdot
				t
				\cdot
				\varphi_{\alpha_0}  \matr{1}{b}{0}{1}
			}
		}"']
		\ar[d, "{\substack{
			\left(\overline{c}, \overline{t}, \overline{b}\right)
			\\
			\downmapsto
			\\
			\left(0, \overline{t}, \overline{b}\right)
		}}"']
		&
		{\fratt{I_{wv}}}
		\ar[d, "{\substack{\text{ind. by} \\ I_{wv} \hookrightarrow I_{w}}}"'  pos=0.6]
		\\
		{
			\quoz{\ooo}{\mmm}
			\times
			{T^1_{w}}
			\times
			\quoz{\ooo}{\mmm}
		}
		\ar[r, "{
			\left(\overline{c}, \overline{t}, \overline{b}\right)
			\mapsto
			\overline{
				\varphi_{\alpha_0} \matr{1}{0}{\pi^{\ell(w)+1} c}{1}
				\cdot
				t
				\cdot
				\varphi_{\alpha_0}  \matr{1}{b}{0}{1}
			}
		}"']
		&
		{\fratt{I_{w}}}.
	\end{tikzcd}\]
	Dualizing (i.e., applying $\homm_{\vect{\ff_p}} \big( \fratt{{}_-}, k \big)$), we get the formula in the statement.

\item[\ref{item2ExplFormulasE1}]
	Now, let us assume that $\ell(s_0w) = \ell(w) + 1$ and let us prove the second formula.
	
	We first remark that since lengths add up we can apply the formula \eqref{eqActionRightLeftE0IfLengthsAddUp}, obtaining
	\begin{align*}
		\tau_{s_0} \cdot \beta &\in H^i(I,\xx(wv)),
		\\
		\sh_{s_0w}(\tau_{s_0} \cdot \beta) &= \res_{I_{s_0w}}^{s_0 I_w s_0^{-1}} \big( (s_0)_\ast \sh_w(\beta) \big),
	\end{align*}
	where $\beta \defeq \big(c^-,c^0,c^+\big)_w$.
	
	This time we have to compute the action of the inclusion map and of conjugation on the level of the Frattini quotients. We do this via the following commutative diagram (which uses again the description of the Frattini quotients given in Lemma \ref{lmmFrattiniQuot}: note that the description of $\fratt{s_0 I_{w} s_0^{-1}}$ follows from the description of $\fratt{I_w}$):
	\[\begin{tikzcd}[row sep = 3.5em, column sep = 13em]	
		{
			\quoz{\ooo}{\mmm}
			\times
			{T^1_{s_0w}}
			\times
			\quoz{\ooo}{\mmm}
		}
		\ar[r, "{
			\left(\overline{c}, \overline{t}, \overline{b}\right)
			\mapsto
			\overline{
				\varphi_{\alpha_0} \matr{1}{0}{\pi c}{1}
				\cdot
				t
				\cdot
				\varphi_{\alpha_0}  \matr{1}{\pi^{\ell(s_0w)} b}{0}{1}
			}
		}"']
		\ar[d, "{\substack{
			\left(\overline{c}, \overline{t}, \overline{b}\right)
			\\
			\downmapsto
			\\
			\left(0, \overline{t}, \overline{b}\right)
		}}"']
		&
		{\fratt{I_{s_0w}}}
		\ar[d, "{\substack{\text{ind. by} \\ I_{s_0w} \hookrightarrow s_0 I_{w} s_0^{-1}}}"' pos=0.6]
		\\
		{
			\quoz{\ooo}{\mmm}
			\times
			{T^1_{w}}
			\times
			\quoz{\ooo}{\mmm}
		}
		\ar[r, "{
			\left(\overline{c}, \overline{t}, \overline{b}\right)
			\mapsto
			\overline{
				\varphi_{\alpha_0} \matr{1}{0}{c}{1}
				\cdot
				t
				\cdot
				\varphi_{\alpha_0}  \matr{1}{\pi^{\ell(w)+1} b}{0}{1}
			}
		}"']
		\ar[d, "{\substack{
			\left(\overline{c}, \overline{t}, \overline{b}\right)
			\\
			\downmapsto
			\\
			\left(-\overline{b}, \overline{w_0^{-1} t w_0}, -\overline{c}\right)
		}}"']
		&
		{\fratt{s_0 I_{w} s_0^{-1}}}
		\ar[d, "{{\varphi_{\alpha_0}\matr{0}{1}{-1}{0}}^{-1} \cdot {}_- \cdot \varphi_{\alpha_0}\matr{0}{1}{-1}{0}}"' pos=0.6]
		\\
		{
			\quoz{\ooo}{\mmm}
			\times
			{T^1_{w}}
			\times
			\quoz{\ooo}{\mmm}
		}
		\ar[r, "{
			\left(\overline{c}, \overline{t}, \overline{b}\right)
			\mapsto
			\overline{
				\varphi_{\alpha_0} \matr{1}{0}{\pi^{\ell(w)+1} c}{1}
				\cdot
				t
				\cdot
				\varphi_{\alpha_0}  \matr{1}{b}{0}{1}
			}
		}"']
		&
		{\fratt{I_{w}}}.
	\end{tikzcd}\]
	Dualizing, we get the formula in the statement.

\item[\ref{item3ExplFormulasE1}]
	The proof is analogous to \ref{item1ExplFormulasE1}. The commutative diagram in this case is
	\[\begin{tikzcd}[row sep = 3.5em, column sep = 15em]
		{
			\quoz{\ooo}{\mmm}
			\times
			{T^1_{wv}}
			\times
			\quoz{\ooo}{\mmm}
		}
		\ar[r, "{
			\left(\overline{c}, \overline{t}, \overline{b}\right)
			\mapsto
			\overline{
				\varphi_{\alpha_0} \matr{1}{0}{\pi c}{1}
				\cdot
				t
				\cdot
				\varphi_{\alpha_0}  \matr{1}{\pi^{\ell(wv)} b}{0}{1}
			}
		}"']
		\ar[d, "{\substack{
			\left(\overline{c}, \overline{t}, \overline{b}\right)
			\\
			\downmapsto
			\\
			\left(\overline{c}, \overline{t}, 0\right)
		}}"']
		&
		{\fratt{I_{wv}}}
		\ar[d, "{\substack{\text{ind. by} \\ I_{wv} \hookrightarrow I_{w}}}"' pos=0.6]
		\\
		{
			\quoz{\ooo}{\mmm}
			\times
			{T^1_{w}}
			\times
			\quoz{\ooo}{\mmm}
		}
		\ar[r, "{
			\left(\overline{c}, \overline{t}, \overline{b}\right)
			\mapsto
			\overline{
				\varphi_{\alpha_0} \matr{1}{0}{\pi c}{1}
				\cdot
				t
				\cdot
				\varphi_{\alpha_0}  \matr{1}{\pi^{\ell(w)} b}{0}{1}
			}
		}"']
		&
		{\fratt{I_{w}}}.
	\end{tikzcd}\]
	
\item[\ref{item4ExplFormulasE1}]
	The proof is analogous to \ref{item1ExplFormulasE1}. The commutative diagram in this case is
	\[\begin{tikzcd}[row sep = 3.5em, column sep = 13.5em]
		{
			\quoz{\ooo}{\mmm}
			\times
			{T^1_{s_1w}}
			\times
			\quoz{\ooo}{\mmm}
		}
		\ar[r, "{
			\left(\overline{c}, \overline{t}, \overline{b}\right)
			\mapsto
			\overline{
				\varphi_{\alpha_0} \matr{1}{0}{\pi^{\ell(s_1w)+1} c}{1}
				\cdot
				t
				\cdot
				\varphi_{\alpha_0}  \matr{1}{b}{0}{1}
			}
		}"']
		\ar[d, "{\substack{
			\left(\overline{c}, \overline{t}, \overline{b}\right)
			\\
			\downmapsto
			\\
			\left(\overline{c}, \overline{t}, 0\right)
		}}"']
		&
		{\fratt{I_{s_1w}}}
		\ar[d, "{\substack{\text{ind. by} \\ I_{s_1w} \hookrightarrow s_1 I_{w} s_1^{-1}}}"'  pos=0.6]
		\\
		{
			\quoz{\ooo}{\mmm}
			\times
			{T^1_{w}}
			\times
			\quoz{\ooo}{\mmm}
		}
		\ar[r, "{
			\left(\overline{c}, \overline{t}, \overline{b}\right)
			\mapsto
			\overline{
				\varphi_{\alpha_0} \matr{1}{0}{\pi^{\ell(w)+2} c}{1}
				\cdot
				t
				\cdot
				\varphi_{\alpha_0}  \matr{1}{\pi^{-1} b}{0}{1}
			}
		}"']
		\ar[d, "{\substack{
			\left(\overline{c}, \overline{t}, \overline{b}\right)
			\\
			\downmapsto
			\\
			\left(-\overline{b}, \overline{w_0^{-1} t w_0}, -\overline{c}\right)
		}}"']
		&
		{\fratt{s_1 I_{w} s_1^{-1}}}
		\ar[d, "{{\varphi_{\alpha_0}\matr{0}{-\pi^{-1}}{\pi}{0}}^{-1} \cdot {}_- \cdot \varphi_{\alpha_0}\matr{0}{-\pi^{-1}}{\pi}{0}}"' pos=0.6]
		\\
		{
			\quoz{\ooo}{\mmm}
			\times
			{T^1_{w}}
			\times
			\quoz{\ooo}{\mmm}
		}
		\ar[r, "{
			\left(\overline{c}, \overline{t}, \overline{b}\right)
			\mapsto
			\overline{
				\varphi_{\alpha_0} \matr{1}{0}{\pi c}{1}
				\cdot
				t
				\cdot
				\varphi_{\alpha_0}  \matr{1}{\pi^{\ell(w)} b}{0}{1}
			}
		}"']
		&
		{\fratt{I_{w}}}.
	\end{tikzcd}\]
\qedhere
\end{enumerate}
\end{proof}

\begin{prop}\label{propAlmostGenByE1}
Assume that $p \neq 2,3$, that $\field$ is an unramified extension of $\qq_p$ and that $\GG$ has semisimple rank $1$. Let us denote by $f$ the residue degree of $\field$ over $\qq_p$. Let us consider the tensor algebra $T^\ast_{E^0} E^1$ associated with the $E^0$-bimodule $E^1$ and the natural multiplication map $\funcInlineNN{T^\ast_{E^0} E^1}{E^\ast}$. One has that
\[
	\image \big( \funcInlineNN{T^\ast_{E^0} E^1}{E^\ast} \big)
	\supseteq
	\bigoplus_{\substack{w \in \widetilde{W} \\ \text{s.t. $\ell(w) \geqslant f \cdot \dim \TT$}}}
		H^\ast(I,\xx(w)).
\]
\end{prop}

\begin{proof}
We first note that $f \cdot \dim \TT = \dim_{\ff_p} \quoz{T^1}{(T^1)^p}$. Indeed, we have
\[
	\quoz{T^1}{(T^1)^p}
	\cong 
	\big( \quoz{1+\mmm}{(1+\mmm)^p} \big)^{\oplus \dim \TT}
	\cong
	\big( \quoz{\mmm}{p\mmm} \big)^{\oplus \dim \TT}
	\cong
	\big( \quoz{\ooo}{\mmm} \big)^{\oplus \dim \TT},
\]
where in the second step we have used the $p$-adic logarithm/exponential (which we may apply since $\field$ is an unramified extension of $\qq_p$ and $p \neq 2$) and in the third step we have further used that $\field$ is an unramified extension of $\qq_p$. The claimed formula $f \cdot \dim \TT = \dim_{\ff_p} \quoz{T^1}{(T^1)^p}$ now follows.

Now, for the moment let us consider $a, b, c \in \nn$ and arbitrary elements
\begin{align*}
	\mathfrak{a}_1, \dots \mathfrak{a}_{a} &\in \homm_{\vect{\ff_p}}(\quoz{\ooo}{\mmm}, k),
	\\
	\mathfrak{b}_1, \dots \mathfrak{b}_{b} &\in \homm_{\vect{\ff_p}}(\quoz{T^1}{(T^1)^p}, k),
	\\
	\mathfrak{c}_1, \dots \mathfrak{c}_{c} &\in \homm_{\vect{\ff_p}}(\quoz{\ooo}{\mmm}, k).
\end{align*}
In the next computations it will be handy to work with the following notation: for all $n \in \nn$ we define $s_{\overline{n}}$ to be $s_0$ if $n$ is even and $s_1$ if $n$ is odd.

Let $m \in \nn$ be an odd number with $m \geqslant b-1$. With the notation of Remark \ref{remIsoParentesi}, we make the following computation, whose steps will be explained below.
\begin{equation}\label{eqFirstBigComputation}\begin{aligned}
	&\prod_{i=1}^{a} (\mathfrak{a}_i,0,0)_1
	\cdot
	\prod_{i=1}^{b} (0,\mathfrak{b}_i,0)_{s_{\overline{i-1}}}
	\cdot
	\tau_{s_{\overline{b}} \cdots s_{\overline{m}}}
	\cdot
	\prod_{i=1}^{c} (0,0,\mathfrak{c}_i)_1
\\&\qquad=
	\bigcupprod_{i=1}^{a} \big( (\mathfrak{a}_i,0,0)_1 \cdot \tau_{s_{\overline{0}} \cdots s_{\overline{m}}} \big)
	\cupprod
	\bigcupprod_{i=1}^{b}
		\big(
			\tau_{s_{\overline{0}} \cdots s_{\overline{i-2}}}
			\cdot
			(0,\mathfrak{b}_i,0)_{s_{\overline{i-1}}}
			\cdot
			\tau_{s_{\overline{i}} \cdots s_{\overline{m}}}
		\big)
	\\&\qquad\qquad\cupprod
	\tau_{s_{\overline{0}} \cdots s_{\overline{m}}}
	\cupprod
	\bigcupprod_{i=1}^{c} \big( \tau_{s_{\overline{0}} \cdots s_{\overline{m}}} \cdot (0,0,\mathfrak{c}_i)_1 \big)
\\&\qquad=
	\bigcupprod_{i=1}^{a} (\mathfrak{a}_i,0,0)_{s_{\overline{0}} \cdots s_{\overline{m}}} 
	\cupprod
	\bigcupprod_{i=1}^{b}
			\big( 0,\mathfrak{b}_i(w_0^{-(i-1)} {}_- w_0^{i-1}),0 \big)_{s_{\overline{0}} \cdots s_{\overline{m}}}
	\\&\qquad\qquad\cupprod
	\bigcupprod_{i=1}^{c} (0,0,\mathfrak{c}_i)_{s_{\overline{0}} \cdots s_{\overline{m}}}.
\end{aligned}\end{equation}
Here, $s_{\overline{b}} \cdots s_{\overline{m}}$ means the empty product in the case $m = b-1$. To prove the first equality it suffices to apply the formula of Remark \ref{remMultipleCupprod}. Also note that cup product by $\tau_{s_{\overline{0}} \cdots s_{\overline{m}}}$ is the identity on $H^\ast(I,\xx(s_{\overline{0}} \cdots s_{\overline{m}}))$, and so it can be ignored. Regarding the second equality, the fact that
$
(\mathfrak{a}_i,0,0)_1 \cdot \tau_{s_{\overline{0}} \cdots s_{\overline{m}}}
=
(\mathfrak{a}_i,0,0)_{s_{\overline{0}} \cdots s_{\overline{m}}} 
$
follows from Lemma \ref{lmmExplicitFormulasE1} part \ref{item3ExplFormulasE1}; the fact that
\[
	\tau_{s_{\overline{0}} \cdots s_{\overline{i-2}}}
	\cdot
	(0,\mathfrak{b}_i,0)_{s_{\overline{i-1}}}
	\cdot
	\tau_{s_{\overline{i}} \cdots s_{\overline{m}}}
=
	\big( 0,\mathfrak{b}_i(w_0^{-(i-1)} {}_- w_0^{i-1}),0 \big)_{s_{\overline{0}} \cdots s_{\overline{m}}}
\]
follows from Lemma \ref{lmmExplicitFormulasE1} (applying iteratively parts \ref{item2ExplFormulasE1} and \ref{item4ExplFormulasE1} to compute the action of the left multiplication by $\tau_{s_{\overline{0}} \cdots s_{\overline{i-2}}}$ and applying a single time parts \ref{item1ExplFormulasE1} or \ref{item3ExplFormulasE1}, depending on the parity of $i$, to compute the action of the right multiplication by $\tau_{s_{\overline{i}} \cdots s_{\overline{m}}}$.
Finally, the fact that
$
\tau_{s_{\overline{0}} \cdots s_{\overline{m}}} \cdot (0,0,\mathfrak{c}_i)_1
=
(0,0,\mathfrak{c}_i)_{s_{\overline{0}} \cdots s_{\overline{m}}}
$
follows from Lemma \ref{lmmExplicitFormulasE1} (applying iteratively parts \ref{item2ExplFormulasE1} and \ref{item4ExplFormulasE1}): here is the only point where we use that $m$ is odd.

We have thus checked the computation \eqref{eqFirstBigComputation}. Before deducing some consequences, let us make some completely similar computations.

Let $m \in \nn$ be an even number with $m \geqslant b-1$. The following computation is exactly as \eqref{eqFirstBigComputation} except that the factors on the right are different because now $m$ is even (recall that in \eqref{eqFirstBigComputation} we have used that $m$ was odd only in the very final step).
\begin{equation}\label{eqSecondBigComputation}\begin{aligned}
	&\prod_{i=1}^{a} (\mathfrak{a}_i,0,0)_1
	\cdot
	\prod_{i=1}^{b} (0,\mathfrak{b}_i,0)_{s_{\overline{i-1}}}
	\cdot
	\tau_{s_{\overline{b}} \cdots s_{\overline{m}}}
	\cdot
	\prod_{i=1}^{c} (\mathfrak{c}_i,0,0)_1
\\&\qquad=
	\bigcupprod_{i=1}^{a} (\mathfrak{a}_i,0,0)_{s_{\overline{0}} \cdots s_{\overline{m}}} 
	\cupprod
	\bigcupprod_{i=1}^{b}
			\big( 0,\mathfrak{b}_i(w_0^{-(i-1)} {}_- w_0^{i-1}),0 \big)_{s_{\overline{0}} \cdots s_{\overline{m}}}
	\\&\qquad\qquad\cupprod
	\bigcupprod_{i=1}^{c} (0,0,-\mathfrak{c}_i)_{s_{\overline{0}} \cdots s_{\overline{m}}}.
\end{aligned}\end{equation}

Now, let $m \in \zpiu$ be an odd number with $m \geqslant b$. We make the following computation
\begin{equation}\label{eqThirdBigComputation}\begin{aligned}
	&\prod_{i=1}^{a} (0,0,\mathfrak{a}_i)_1
	\cdot
	\prod_{i=1}^{b} (0,\mathfrak{b}_i,0)_{s_{\overline{i}}}
	\cdot
	\tau_{s_{\overline{b+1}} \cdots s_{\overline{m}}}
	\cdot
	\prod_{i=1}^{c} (0,0,\mathfrak{c}_i)_1
\\&\qquad=
	\bigcupprod_{i=1}^{a} \big( (0,0,\mathfrak{a}_i)_1 \cdot \tau_{s_{\overline{1}} \cdots s_{\overline{m}}} \big)
	\cupprod
	\bigcupprod_{i=1}^{b}
		\big(
			\tau_{s_{\overline{1}} \cdots s_{\overline{i-1}}}
			\cdot
			(0,\mathfrak{b}_i,0)_{s_{\overline{i}}}
			\cdot
			\tau_{s_{\overline{i+1}} \cdots s_{\overline{m}}}
		\big)
	\\&\qquad\qquad\cupprod
	\tau_{s_{\overline{1}} \cdots s_{\overline{m}}}
	\cupprod
	\bigcupprod_{i=1}^{c} \big( \tau_{s_{\overline{1}} \cdots s_{\overline{m}}} \cdot (0,0,\mathfrak{c}_i)_1 \big)
\\&\qquad=
	\bigcupprod_{i=1}^{a} (0,0,\mathfrak{a}_i)_{s_{\overline{1}} \cdots s_{\overline{m}}} 
	\cupprod
	\bigcupprod_{i=1}^{b}
			\big( 0,\mathfrak{b}_i(w_0^{-(i-1)} {}_- w_0^{i-1}),0 \big)_{s_{\overline{1}} \cdots s_{\overline{m}}}
	\\&\qquad\qquad\cupprod
	\bigcupprod_{i=1}^{c} (-\mathfrak{c}_i,0,0)_{s_{\overline{1}} \cdots s_{\overline{m}}}.
\end{aligned}\end{equation}

Now, let $m \in \zpiu$ be an even number with $m \geqslant b$. We make the following computation
\begin{equation}\label{eqFourthBigComputation}\begin{aligned}
	&\prod_{i=1}^{a} (0,0,\mathfrak{a}_i)_1
	\cdot
	\prod_{i=1}^{b} (0,\mathfrak{b}_i,0)_{s_{\overline{i}}}
	\cdot
	\tau_{s_{\overline{b+1}} \cdots s_{\overline{m}}}
	\cdot
	\prod_{i=1}^{c} (\mathfrak{c}_i,0,0)_1
\\&\qquad=
	\bigcupprod_{i=1}^{a} (0,0,\mathfrak{a}_i)_{s_{\overline{1}} \cdots s_{\overline{m}}} 
	\cupprod
	\bigcupprod_{i=1}^{b}
			\big( 0,\mathfrak{b}_i(w_0^{-(i-1)} {}_- w_0^{i-1}),0 \big)_{s_{\overline{1}} \cdots s_{\overline{m}}}
	\\&\qquad\qquad\cupprod
	\bigcupprod_{i=1}^{c} (\mathfrak{c}_i,0,0)_{s_{\overline{1}} \cdots s_{\overline{m}}}.
\end{aligned}\end{equation}

Now let us go back to \eqref{eqFirstBigComputation} and let us draw some consequences from it. Recall that all of $a$, $b$, $c$ and all of $\mathfrak{a}_i$, $\mathfrak{b}_i$, $\mathfrak{c}_i$ were arbitrary, and $m \in \nn$ was an odd integer with $m \geqslant b-1$.

Let $m \in \nn$ be an arbitrary odd natural number. Recall that (since $m+1 \geqslant 1$) the pro-$p$ group $I_{s_{\overline{0}} \cdots s_{\overline{m}}}$ is uniform (Lemma \ref{lmmPoincAndUniform}). Therefore, the cup product algebra on $H^\ast(I, \xx(s_{\overline{0}} \cdots s_{\overline{m}})) \cong H^\ast(I_{s_{\overline{0}} \cdots s_{\overline{m}}},k)$ is the exterior algebra on $H^1$, as recalled in \S\ref{subsecCohom}. The computation \eqref{eqFirstBigComputation} tells us that the elements of the following form lie in the image of the multiplication map $\funcInlineNN{T^\ast_{E^0} E^1}{E^\ast}$:
\begin{equation}\label{eqCupObtained}
	\bigcupprod_{i=1}^{a} (\mathfrak{a}_i,0,0)_{s_{\overline{0}} \cdots s_{\overline{m}}} 
	\cupprod
	\bigcupprod_{i=1}^{b}
			\big( 0,\mathfrak{b}_i',0 \big)_{s_{\overline{0}} \cdots s_{\overline{m}}}
	\cupprod
	\bigcupprod_{i=1}^{c} (0,0,\mathfrak{c}_i)_{s_{\overline{0}} \cdots s_{\overline{m}}},
\end{equation}
for $a, c \in \nn$, for $b \in \nn$ with $b \leqslant m+1$, for $\mathfrak{a}_i, \mathfrak{c}_{i} \in \homm_{\vect{\ff_p}}(\quoz{\ooo}{\mmm}, k)$ and for $\mathfrak{b}_i' \in \homm_{\vect{\ff_p}}(\quoz{T^1}{(T^1)^p}, k)$. Now let $m \geqslant f \cdot \dim \TT - 1$: in this way, since $\quoz{T^1}{(T^1)^p}$ has dimension $f \cdot \dim \TT$ as an $\ff_p$-vector space, a $k$-basis of $\homm_{\vect{\ff_p}}(\quoz{T^1}{(T^1)^p}, k)$ has cardinality less or equal than $m+1$. Then the elements of the form \eqref{eqCupObtained} (with $b \leqslant m+1$) span the cup product algebra/exterior algebra $H^\ast(I, \xx(s_{\overline{0}} \cdots s_{\overline{m}}))$.

All in all, this shows that the image of the multiplication map $\funcInlineNN{T^\ast_{E^0} E^1}{E^\ast}$ contains $H^\ast(I, \xx(s_{\overline{0}} \cdots s_{\overline{m}}))$ for $m \in \nn$ odd with $m \geqslant f \cdot \dim \TT - 1$. Reasoning exactly in the same way for the other computations we made, we obtain that the image of the multiplication map $\funcInlineNN{T^\ast_{E^0} E^1}{E^\ast}$ contains the following vector spaces:
\begin{align*}
	&H^\ast(I, \xx(s_{\overline{0}} \cdots s_{\overline{m}}))
	&&\text{for $m \in \nn$ odd with $m \geqslant f \cdot \dim \TT - 1$,}
	\\
	&H^\ast(I, \xx(s_{\overline{0}} \cdots s_{\overline{m}}))
	&&\text{for $m \in \nn$ even with $m \geqslant f \cdot \dim \TT - 1$,}
	\\
	&H^\ast(I, \xx(s_{\overline{1}} \cdots s_{\overline{m}}))
	&&\text{for $m \in \zpiu$ odd with $m \geqslant f \cdot \dim \TT$,}
	\\
	&H^\ast(I, \xx(s_{\overline{1}} \cdots s_{\overline{m}}))
	&&\text{for $m \in \zpiu$ even with $m \geqslant f \cdot \dim \TT$.}
\end{align*}

By the discussion at the beginning of this subsection, every element of $\widetilde{W}$ of length greater or equal than $f \cdot \dim \TT$ can be represented as a product $\omega w$, where $\omega \in \widetilde{\Omega}$ (i.e., $\omega$ has length zero) and $w$ is of the form $s_{\overline{0}} \cdots s_{\overline{m}}$ or $s_{\overline{1}} \cdots s_{\overline{m}}$ as above.

With this notation, we thus know that the image of the multiplication map $\funcInlineNN{T^\ast_{E^0} E^1}{E^\ast}$ contains $H^\ast(I,\xx(w))$, and to show that it also contains $H^\ast(I,\xx(\omega w))$ it suffices to recall that, since $\omega$ has length zero, the formula \eqref{eqActionRightLeftE0IfLengthsAddUp} implies that multiplication on the left by $\tau_\omega$ induces an isomorphism between $H^\ast(I,\xx(w))$ and $H^\ast(I,\xx(\omega w))$.
\end{proof}

\begin{rem}\label{remFiniteCodim}
Assume that $p \neq 2,3$, that $\GG$ is $\SL_2$ or $\PGL_2$ and that $\field$ is an unramified extension of $\qq_p$. One has that the image of the natural multiplication map $\funcInlineNN{T^\ast_{E^0} E^1}{E^\ast}$ has finite codimension in $E^\ast$.
\end{rem}

\begin{proof}
For all $w \in \widetilde{W}$ the $k$-vector space $H^\ast(I,\xx(w)) \cong H^\ast(I_w,k)$ is finite-dimensional since $I_w$ is a Poincaré group (Lemma \ref{lmmPoincAndUniform}). If one shows that, since $\GG$ is semisimple, there are only finitely many elements of $\widetilde{W}$ of any given length, then by Proposition \ref{propAlmostGenByE1} we conclude.

It suffices to show the claim for $W$ in place of $\widetilde{W}$ by the definition of length and since $W$ is a quotient of $\widetilde{W}$ of finite index. Recall from \S\ref{subsecAffineRootsLength} that the length $\ell$ is constant on double cosets $(\Omega w \Omega)_{w \in W_{\aff}}$. Therefore, it suffices to show that $\Omega$ is finite. But $\Omega$ is isomorphic to $X_\ast(\TT)/\spann_{\zz} \Phi^\vee$ (see \cite[\S1.5]{Lusztig}, with the difference that in the construction of $W$ in loc.cit. roots are swapped with coroots and characters with cocaracters with respect to our conventions). Since $\GG$ is semisimple, $X_\ast(\TT)/\spann_{\zz} \Phi^\vee$ is indeed finite (and isomorphic to the fundamental group).
\end{proof}

\section{The \texorpdfstring{$\ext$-algebra}{Ext-algebra} \texorpdfstring{$E^\ast$}{E*} is finitely generated in some special cases}\label{sectionFinGen}

Putting together the results obtained in Proposition \ref{propE1fg} and Proposition \ref{propAlmostGenByE1}, we are able to show that, under the assumptions under which those results have been proven, the $\ext$-algebra $E^\ast$ is finitely generated as a (non-commutative) $k$-algebra.

\begin{thm}\label{thmA1fg}
Assume that $p \neq 2,3$, that $\field$ is an unramified extension of $\qq_p$ and that $\GG$ has semisimple rank $1$. One has that $E^\ast$ is finitely generated as a $k$-algebra.
\end{thm}

\begin{proof}
Proposition \ref{propAlmostGenByE1} shows that
\[
	\image \big( \funcInlineNN{T^\ast_{E^0} E^1}{E^\ast} \big)
	\supseteq
	\bigoplus_{\substack{w \in \widetilde{W} \\ \text{s.t. $\ell(w) \geqslant f \cdot \dim \TT$}}}
		H^\ast(I,\xx(w)).
\]
Therefore $E^\ast$ is generated as a $k$-algebra by $E^0$, $E^1$ and
\[
	\bigoplus_{\substack{w \in \widetilde{W} \\ \text{s.t. $\ell(w) < f \cdot \dim \TT$}}} H^\ast(I,\xx(w)).
\]
By \eqref{eqActionRightLeftE0IfLengthsAddUp}, for all $w \in \widetilde{W}$ and all $\omega \in \widetilde{\Omega}$, multiplication on the left by $\tau_\omega$ induces an isomorphism between $H^\ast(I,\xx(w))$ and $H^\ast(I,\xx(\omega w))$. Furthermore, $E^1$ is finitely generated as an $E^0$-bimodule by Proposition \ref{propE1fg} (in this setting a more explicit proof could also be provided starting from the formulas in Lemma \ref{lmmExplicitFormulasE1}).

This shows that $E^\ast$ is generated as a $k$-algebra by $E^0$, a finite subset of $E^1$ and
\[
	\bigoplus_{\substack{w \in \widetilde{W_{\aff}} \\ \text{s.t. $\ell(w) < f \cdot \dim \TT$}}} H^\ast(I,\xx(w)).
\]
The last set is finite-dimensional as a $k$-vector space because there are only finitely many elements of bounded length in $W_{\aff}$ (and so in $\widetilde{W_{\aff}}$) and because for all $w \in \widetilde{W}$ the $k$-vector space $H^\ast(I,\xx(w)) \cong H^\ast(I_w,k)$ is finite-dimensional since $I_w$ is a Poincaré group (Lemma \ref{lmmPoincAndUniform}).

Therefore, to show the thesis of the proposition it suffices to notice that $E^0$ is finitely generated as a $k$-algebra. This is clear from the fact that $E^0$ is a finitely generated module over its centre, which is a finitely generated (commutative) $k$-algebra (see \cite[Theorem 1.3]{VignII}), but it can also be shown more directly from the braid relations, the fact that $\widetilde{\Omega}$ is a finitely generated group, and the fact that $S_{\aff}$ is finite.
\end{proof}

\section{The \texorpdfstring{$\ext$-algebra}{Ext-algebra} \texorpdfstring{$E^\ast$}{E*} is finitely presented when \texorpdfstring{$G=\SL_2(\qq_p)$}{G = SL\_2(Q\_p)} with \texorpdfstring{$p \neq 2,3$}{p ≠ 2,3}}\label{sectionFinPres}

\begin{assumptions}
For the whole section we make the following assumptions. We consider $\field \defeq \qq_p$ and $G \defeq \SL_2(\qq_p)$, and we assume that $p \neq 2,3$. Furthermore, we fix the uniformizer $\pi$ to be $p$. We consider the ($\qq_p$-split maximal) torus $\TT$ of diagonal matrices, and we fix the following pro-$p$ Iwahori subgroup $I$:
\begin{equation}\label{eqIwaSL2}
	I = \matr{1+p\zz_p}{\zz_p}{p\zz_p}{1+p\zz_p} \cap \SL_2(\qq_p).
\end{equation}
(the corresponding choices of the positive root and the Chevalley system will be made explicit in \S\ref{subsecNotationSL2}).
\end{assumptions}

\subsection{The \texorpdfstring{$\ext$-algebra}{Ext-algebra} in the case \texorpdfstring{$G=\SL_2(\qq_p)$}{G = SL\_2(Q\_p)}: notation and recollection of  and results}\label{subsecNotationSL2}

Under the current assumptions, the $\ext$-algebra $E^\ast$ has been studied extensively in \cite{newpaper}. We need to recall quite a few facts from loc. cit. and from \cite{ext}.

First of all, it follows from \cite[\S7.2]{ext} that $E^\ast$ is supported in the following degrees:
\[
	E^\ast = E^0 \oplus E^1 \oplus E^2 \oplus E^3.
\]

With reference to \S\ref{subsectionApartment}, we have a Chevalley system such that, denoting the positive root by $\alpha_0$, the map $\varphi_{\alpha_0}$ is the identity of $\SL_2$ and $\varphi_{-\alpha_0}$ is the inverse transpose. With reference to \S\ref{subsecAffineRootsLength}, in this setting the group $\Omega$ is trivial. Following \cite[\S2.3.1]{newpaper}, we define
\begin{align}\label{eqDefs0s1Special}
	&s_0 \defeq \overline{\matr{0}{1}{-1}{0}},
	&&s_1 \defeq \overline{\matr{0}{-\pi^{-1}}{\pi}{0}},
\end{align}
where $\overline{?}$ denotes the image of $?$ in $\widetilde{W}$. This is \virgolette{almost} compatible with the definition we gave in \eqref{eqDefs0s1Gneneral}, the only difference being that the two definitions of $s_1$ differ by multiplication by the central element $\overline{\matr{-1}{0}{0}{-1}}$. However, this will not cause any problems. For example, in \cite[Proposition 7.1]{ext} some general formulas are stated only for the specific choices \eqref{eqDefs0s1Gneneral} (see the comment at the end of the statement), but the only difference being multiplication by a central element, the formulas remain valid for the lifts \eqref{eqDefs0s1Special}. We make the choice \eqref{eqDefs0s1Special} to preserve compatibility of notation with \cite[\S2.3.1]{newpaper}.

We also define
\begin{equation}\label{eqcminus1}
	c_{-1} \defeq \overline{\matr{-1}{0}{0}{-1}} = s_0^2 = s_1^2.
\end{equation}

Since the group $\Omega$ is trivial, every element of $\widetilde{W}$ can be represented in a unique way as a product of an element of $T^0/T^1$ and a string of the elements ${s_0}$ and ${s_1}$, with no two identical letters one after the other (compare the beginning of \S\ref{sectionSubalgE0E1}). In other words, the following is a basis of $E^0$:
\begin{align*}
	&\tau_{\omega} &&\text{for $\omega \in T^0/T^1$,}
	\\
	&\tau_{\omega s_0(s_1s_0)^i} &&\text{for $\omega \in T^0/T^1$ and $i \in \nn$,}
	\\
	&\tau_{\omega s_1(s_0s_1)^i} &&\text{for $\omega \in T^0/T^1$ and $i \in \nn$,}
	\\
	&\tau_{\omega(s_0s_1)^i} &&\text{for $\omega \in T^0/T^1$ and $i \in \zpiu$,}
	\\
	&\tau_{\omega(s_1s_0)^i} &&\text{for $\omega \in T^0/T^1$ and $i \in \zpiu$.}
\end{align*}

As in \cite[Equation (31)]{newpaper}, we parametrize the group $T^0/T^1 \cong \ff_p^\times$ as follows: for all $u \in \ff_p^\times$ we define
\begin{equation}\label{eqOmegaU}
	\omega_u \defeq \overline{\matr{[u]^{-1}}{0}{0}{[u]}},
\end{equation}
where $[?]$ denotes the Teichmüller lift of $?$ and $\overline{?}$ denotes the class of $?$ in $T^0/T^1$.

In \cite[\S6]{ext}, it is shown that the map $\invol_w$ defined by the commutativity of the diagram
\begin{equation}\label{eqqDefOfTheAntiInvol}
\begin{tikzcd}
	H^\ast(I,\xx(w))
	\arrow[r, "{\invol_w}", "{\cong}"']
	\arrow[d, "{\sh_w}"', "{\cong}"]
	&
	H^\ast(I,\xx(w^{-1}))
	\arrow[d, "{\sh_{w^{-1}}}", "{\cong}"']
	\\
	H^\ast(I_w,k)
	\arrow[r, "{(w^{-1})_\ast}"', "{\cong}"]
	&
	H^\ast(I_{w^{-1}},k).
\end{tikzcd}
\end{equation}
gives rise to a map $\funcInline{\invol \defeq \bigoplus_{w \in \widetilde{W}} \invol_w}{E^\ast}{E^\ast}$ that is an involutive anti-automorphism of $E^\ast$ (anti-automorphism meaning that it is an automorphism of $k$-vector spaces and that it satisfies the equation $\invol(\alpha \cdot \beta) = (-1)^{\abs{\alpha} \cdot \abs{\beta}} \invol(\beta) \cdot \invol(\alpha)$ for all homogeneous $\alpha \in E^{\abs{\alpha}}$ and $\beta \in E^{\abs{\beta}}$).

In \cite[\S2.2.6]{newpaper} a certain involutive automorphism $\Gamma_\varpi$ of $E^\ast$ is constructed as follows. One defines $\varpi \defeq \matr{0}{1}{p}{0} \in \GL_2(\qq_p)$ and one considers the map $\Gamma_{\varpi,w}$ defined by the commutativity of the diagram
\[
\begin{tikzcd}
	H^\ast(I,\xx(w))
	\ar[d, "{\sh_w}"', "{\cong}"]
	\ar[r, "{\Gamma_{\varpi,w}}"]
	&
	H^\ast(I,\xx(\varpi w \varpi^{-1}))
	\ar[d, "{\sh_{\varpi w \varpi^{-1}}}", "{\cong}"']	
	\\
	H^\ast(I_w,k)
	\ar[r, "{\varpi_\ast}"]
	&
	H^\ast(I_{\varpi w \varpi^{-1}},k).
\end{tikzcd}
\]
Defining $\funcInline{\Gamma_\varpi \defeq \bigoplus_{w \in \widetilde{W}} \Gamma_{\varpi,w}}{E^\ast}{E^\ast}$, one has that $\Gamma_\varpi$ is an involutive automorphism of $E^\ast$ (as a graded $k$-algebra).
Moreover, the involutions $\Gamma_\varpi$ and $\invol$ commute (see again \cite[\S2.2.6]{newpaper}):
\begin{equation}\label{eqCCCNInvolCommute}
	\Gamma_\varpi \circ \invol = \invol \circ \Gamma_\varpi.
\end{equation}

The Poincaré duality of \S\ref{subsecCohom} for the groups $I_w$ implies that there is a duality between $H^i(I,\xx(w))$ and $H^{3-i}(I,\xx(w))$ for all $i \in \{0,1,2,3\}$ and all $w \in \widetilde{W}$. In \cite[Proposition 7.18]{ext}, such duality is extended to a suitably defined \virgolette{duality} of $E^0$-bimodules involving $E^i$ and $E^{3-i}$. We recall a few details. One fixes (and we will make a specific choice later) an isomorphism of $1$-dimensional vector spaces
\begin{equation}\label{eqEta}
	\funcInline{\eta}{H^3(I,k)}{k}
\end{equation}
as well as  the $G$-equivariant map
\begin{equation}\label{eqDefSSS}
	\func
		{\sss}
		{k \left[ \quoz{G}{I} \right]}
		{k}
		{f}
		{\sum_{\overline{g} \in \quoz{G}{I}} f(g)}
\end{equation}
and its induced map
\[
	\sss^i \defeq \funcInline{H^i(I, \sss)}{E^i = H^i\left(I, k \left[ \quoz{G}{I} \right]\right)}{H^i(I, k)}.
\]
Furthermore, one defines the \emph{finite twisted dual} $\biinvol{(E^{i})^{\vee, \finite}}$ of $E^{i}$ as follows: as a $k$-vector space it is defined as
\[
	\bigoplus_{w \in \widetilde{W}} H^{i}(I,\xx(w))^\vee \subseteq (E^{i})^{\vee},
\]
where $({}_-)^\vee \defeq \homm_{\vect{k}}({}_-, k)$. One may check that this vector space is actually a sub-$E^0$-bimodule of the full dual $(E^{i})^{\vee}$, but we endow $\biinvol{(E^{i})^{\vee, \finite}}$ with the following \virgolette{twisted} action instead:
\begin{align*}
	h \cdot \varphi &\defeq \varphi(\invol(h) \cdot {}_-),
	&
	\varphi \cdot h &\defeq \varphi( {}_- \cdot \invol(h))
\end{align*}
for all $h \in E^0$ and all $\varphi \in \biinvol{(E^{i})^{\vee, \finite}}$, thus justifying the complicated notation.
Having introduced these notations, we can recall the statement of the above-mentioned duality theorem; namely, the map
\[
	\func
		{\Delta^i}
		{E^i}
		{\biinvol{(E^{3-i})^{\vee, \finite}}}
		{\alpha}
		{\Big( \beta \mapsto \big( \eta \circ \sss^d \big) (\alpha \cupprod \beta) \Big)}
\]
is a well-defined isomorphism of $E^0$-bimodules. The link with Poincaré duality of the groups $I_w$ is as follows: The Poincaré duality of \S\ref{subsecCohom} defines an isomorphism of $k$-vector spaces
\begin{equation}\label{eqDualityIso}
	\func
		{\Delta^i_w}
		{H^i(I_w,k)}
		{\big( H^{3-i}(I_w,k) \big)^\vee}
		{\alpha}
		{\big( \beta \mapsto \eta_w (\alpha \cupprod \beta) \big),}
\end{equation}
once one fixes an isomorphism of $1$-dimensional $k$-vector spaces
\[
	\funcInline{\eta_w}{H^3(I_w,k)}{k}
\]
and we make the choice $\eta_w \defeq \eta \circ \cores^{I_w}_I$ (recalling from \cite[Chapter 1, Proposition 30 (4)]{Serre} that the corestriction is an isomorphism on the top graded piece). With these choices the following diagram of solid arrows commutes (see \cite[Diagram in the proof of Lemma 7.5]{ext}):
\begin{equation}\label{eqDiagramDualities}\begin{tikzcd}
	H^i(I_w,k)
	\ar[d, "\sh_w", "\cong"']
	\ar[r, "{\Delta^i_w}"]
	&
	H^{3-i}(I_w,k)^\vee
	\ar[d, leftarrow, "\sh_w^\vee"', "\cong"]
	\\
	H^i(I,\xx(w))
	\ar[d, hook]
	\ar[r, dashed]
	&
	H^{3-i}(I,\xx(w))^\vee
	\ar[d, hook]
	\\
	E^i
	\ar[r, "\Delta^i"]
	&
	\biinvol{(E^{3-i})^{\vee, \finite}}
\end{tikzcd}\end{equation}
One can thus define a dashed map making the whole diagram commute either as the unique map making the upper square commute or as the restriction of $\Delta^i$.

Following \cite[\S4.2.3]{newpaper}, we want to fix $k$-bases for the graded pieces of $E^\ast$ \virgolette{in a compatible way}.
Denoting $({}_-)^\vee \defeq \homm_{\vect{\ff_p}}({}_-,k)$ let us recall from \eqref{eqIsoE1AsATriple} that for all $w \in \widetilde{W}$ we have fixed a specific isomorphism
\begin{equation}\label{eqIsoE1AsATripleSL2Qp}
	\funcInline
		{({}_-, {}_-, {}_-)_w}
		{(\ff_p)^\vee \oplus (T^1_w)^\vee \oplus \ff_p^\vee}
		{H^1(I,\xx(w)).}
\end{equation}
Using the $p$-adic logarithm and exponential, one sees that $(1+\mmm)^p = 1+\mmm^2$, and therefore
\begin{align*}
	T^1_w
	&\cong
	\quoz{(1+\mmm)}{\big( (1+\mmm^{\ell(w)+1})\cdot (1+\mmm)^p \big)}
	\\&\cong
	\begin{cases}
		\ff_p &\text{if $\ell(w) \geqslant 1$,}
		\\
		0 &\text{if $\ell(w) = 0$.}
	\end{cases}
\end{align*}
Henceforth, in the case that $\ell(w) \geqslant 1$ we fix the following isomorphism, induced by the $p$-adic logarithm:
\begin{equation}
	\label{eqDefIota}
	\func
		{\iota}
		{T^1_w \cong \quoz{(1+\mmm)}{(1+\mmm^2)}}
		{\quoz{\ooo}{\mmm} = \ff_p}
		{\overline{1+px}}
		{\overline{x}.}
\end{equation}

Dualizing the isomorphism \eqref{eqIsoE1AsATripleSL2Qp} we obtain an isomorphism
\[
	\funcAbove{\cong}{H^1(I,\xx(w))^\vee}{\ff_p \oplus T^1_w \oplus \ff_p}
\]
We combine the inverse of this isomorphism with the inverse of the isomorphism $\funcInlineNN{H^2(I,\xx(w))}{H^1(I,\xx(w))^\vee}$ defined in the diagram \eqref{eqDiagramDualities}, obtaining a composite isomorphism
\[\begin{tikzcd}
	\ff_p \oplus T^1_w \oplus \ff_p
	\ar[r, "\cong"]
	&
	H^1(I,\xx(w))^\vee
	\ar[r, "\cong"]
	&
	H^2(I,\xx(w)),
\end{tikzcd}\]
for which we will use the notation
\[
	\funcInline{({}_-,{}_-,{}_-)_w}{\ff_p \oplus T^1_w \oplus \ff_p}{H^2(I,\xx(w))}
\]
similarly to the isomorphism describing $H^1(I,\xx(w))$.

Now, let us fix an element $\mathbf{c} \in \homm_{\vect{\ff_p}} \left( \ff_p, k \right)$ and an element $\mathbfSpecial{\alpha} \in \ff_p$ with the property that
\[
	\mathbf{c}(\mathbfSpecial{\alpha}) = 1.
\]
For all $w \in \widetilde{W}$ we define the following $k$-basis of $H^1(I,\xx(w))$:
\begin{align*}
	\bm{w} &\defeq (\mathbf{c},0,0)_w,
	\\
	\bp{w} &\defeq (0,0,\mathbf{c})_w,
	\\
	\bz{w} &\defeq (0,\mathbf{c}\iota,0)_w \qquad\qquad\text{if $\ell(w) \geqslant 1$,}
\end{align*}
and for all $w \in \widetilde{W}$ we define the following $k$-basis of $H^2(I,\xx(w))$:
\begin{align*}
	\am{w} &\defeq (\mathbfSpecial{\alpha},0,0)_w,
	\\
	\ap{w} &\defeq (0,0,\mathbfSpecial{\alpha})_w,
	\\
	\az{w} &\defeq (0,\iota^{-1}(\mathbfSpecial{\alpha}),0)_w \qquad\qquad\text{if $\ell(w) \geqslant 1$.}
\end{align*}

Furthermore, we define the basis $(\phi_v)_{v \in \widetilde{W}}$ of $E^3$ as the dual basis of $(\tau_v)_{v \in \widetilde{W}}$ (dual with respect to the duality isomorphism \eqref{eqDualityIso} with $i = 3$), meaning that for all $w \in \widetilde{W}$ the element $\phi_w$ is characterized by the properties
\begin{align*}
	\left( \eta \circ \sss^d \right) (\phi_w \cupprod \tau_w) &= 1,
	\\
	\left( \eta \circ \sss^d \right) (\phi_w \cupprod \tau_v) &= 0 &&\text{for all $v \in \widetilde{W} \smallsetminus \{w\}$.}
\end{align*}

As anticipated above, we make a specific choice for $\funcInline{\eta}{H^3(I,k)}{k}$. In \cite[Lemma 4.5]{newpaper}, it is shown that there exists a (necessarily unique) choice of $\eta$ such that the following property holds:
\begin{align}\label{eqCond2}
	&\bm{w} \cupprod \bz{w} \cupprod \bp{w} = \phi_w &&\text{for all $w \in \widetilde{W}$ such that $\ell(w) \geqslant 1$.}
\end{align}
Henceforth, we will always work with this fixed choice of $\eta$.

It is possible to show that the following relations hold (see \cite[Lemma 5.3]{newpaper}):
\begin{align}
\az{w} &= \bp{w} \cupprod \bm{w}
&&\text{for all $w \in \widetilde{W}$ such that $\ell(w) \geqslant 1$,}
\notag
\\
\am{w} &= \bz{w} \cupprod \bp{w}
&&\text{for all $w \in \widetilde{W}$ such that $\ell(w) \geqslant 1$,}
\label{eqFormulaCupUniform}
\\
\ap{w} &= \bm{w} \cupprod \bz{w}
&&\text{for all $w \in \widetilde{W}$ such that $\ell(w) \geqslant 1$.}
\notag
\end{align}

We have thus fixed the following $k$-basis of $H^\ast(I,\xx(w))$:
\begin{equation}\label{eqBasis}
\begin{aligned}
	&\tau_w,&&&&
	\\
	&\bm{w}, &&\qquad\qquad\bz{w} \quad\text{(if $\ell(w) \geqslant 1$),} &&\qquad\qquad\bp{w},
	\\
	&\am{w}, &&\qquad\qquad\az{w} \quad\text{(if $\ell(w) \geqslant 1$),} &&\qquad\qquad\ap{w},
	\\
	&\phi_w.&&&&
\end{aligned}
\end{equation}

It is useful to introduce some idempotent elements of $E^0$. We denote by $\widehat{T^0/T^1}$ the group of characters $\funcInlineNN{T^0/T^1}{k^\times}$. To each $\lambda \in \widehat{T^0/T^1}$ we attach
\begin{equation}\label{eqDefELambda}
	e_\lambda \defeq - \sum_{t \in T^0/T^1} \lambda(t)^{-1} \tau_{t}.
\end{equation}
Using the braid relations and that $T^0/T^1 \cong \ff_p^\times$, one sees that $(e_\lambda)_{\lambda \in \widehat{T^0/T^1}}$ is a system of orthogonal idempotents, that they form a basis of the subalgebra $\spann_k \set{\tau_t}{t \in \quoz{T^0}{T^1}} \cong k[T^0/T^1]$ of $E^0$, and that the formula
\begin{equation}\label{eqELambdaTauT}
	e_\lambda \cdot \tau_t = \tau_t \cdot e_\lambda = \lambda(t) e_\lambda
\end{equation}
holds for all $\lambda \in \widehat{T^0/T^1}$ and all $t \in T^0/T^1$ (compare \cite[\S2.2.1]{ext}). With this notation, $e_1$ denotes the idempotent associated with the trivial character. Furthermore, we define
\begin{equation}\label{eqDefIdd}
	\func
		{\idd}
		{T^0/T^1}
		{k^\times}
		{\omega_u}
		{u,}
\end{equation}
where the notation $\omega_u$ has been introduced in \eqref{eqOmegaU}. We may thus write $e_{\idd}$ for the corresponding idempotent. More generally, every idempotent of the form \eqref{eqDefELambda} can be written as $e_{\idd^m}$ for some uniquely determined $m \in \{0,\dots,p-2\}$.

Using the basis \eqref{eqBasis}, we will now recall the formulas for the multiplication on $E^0$, for the left action of $E^0$ on $E^\ast$ (and partially, also on the right) as well as the action of the anti-involution $\invol$ and of the involutive automorphism $\Gamma_\varpi$.
\begin{itemize}
\item Braid and quadratic relations on $E^0$ (see \cite[Equations (19) and (20)]{newpaper}):

One has:
\begin{align}\label{eqQuadrRelSL2}
	&\tau_{v} \cdot \tau_w = \tau_{vw} &&\text{for $v,w \in \widetilde{W}$ such that $\ell(vw)=\ell(v)+\ell(w)$,}
\\
	&\tau_{s_i}^2 = - e_1 \cdot \tau_{s_i} = - \tau_{s_i} \cdot e_1,
	&&\text{for $i \in \{0,1\}$.}
\end{align}

\item Action of the anti-involution $\invol$ on $E^0$ (see \cite[Proposition 6.1]{ext}):

For all $w \in \widetilde{W}$ one has:
\begin{equation}\label{eqInvolTauw}
	\invol(\tau_w) = \tau_{w^{-1}}.
\end{equation}

\item Action of the involutive automorphism $\Gamma_\varpi$ on $E^0$:

For all $w \in \widetilde{W}$ one has that
\[
	\Gamma_\varpi(\tau_w) = \tau_{\varpi w \varpi^{-1}}
\]
(immediate from the definition of $\Gamma_\varpi$). With an explicit computation, this can be made explicit as follows:
\begin{equation}\label{eqCCCNOnH}\begin{aligned}
	\Gamma_\varpi(\tau_\omega) &= \tau_{\omega^{-1}} \qquad\qquad\qquad\qquad\qquad \text{for all $\omega \in \quoz{T^0}{T^1}$,}
	\\
	\Gamma_\varpi(\tau_{s_0}) &= \tau_{s_1},
	\\
	\Gamma_\varpi(\tau_{s_1}) &= \tau_{s_0}.
\end{aligned}\end{equation}

\item Left and right action of $\tau_{\omega}$ on $E^1$ for $\omega \in \quoz{T^0}{T^1}$ (see \cite[Equations (64) and (66)]{newpaper}):

Let $u \in \ff_p^\times$ and let $w \in \widetilde{W}$ (with the understanding that $\ell(w) \geqslant 1$ when we speak about $\bz{w}$). Recall the notation $\omega_u$ from \eqref{eqOmegaU}. One has:
\begin{equation}\label{eqOmegaOnE1Left}\begin{aligned}
	\tau_{\omega_{u}} \cdot \bm{w} &= u^{-2} \bm{\omega_u w},
\\
	\tau_{\omega_{u}} \cdot \bz{w} &= \bz{\omega_u w},
\\
	\tau_{\omega_{u}} \cdot \bp{w} &= u^2 \bp{\omega_u w},
\end{aligned}\end{equation}
and
\begin{equation}\label{eqOmegaOnE1Right}\begin{aligned}
	\bm{w} \cdot \tau_{\omega_{u}} &=\bm{w \omega_{u}},
\\
	\bz{w} \cdot \tau_{\omega_{u}} &= \bz{w \omega_{u}},
\\	
	\bp{w} \cdot \tau_{\omega_{u}} &= \bp{w \omega_{u}}.
\end{aligned}\end{equation}

\item Action of the idempotents on $E^1$:

Let $\lambda \in \widehat{\quoz{T^0}{T^1}}$ and let $w \in \widetilde{W}$ (with the understanding that $\ell(w) \geqslant 1$ when we speak about $\bz{w}$). Recalling the notation \eqref{eqDefELambda} and \eqref{eqDefIdd}, one has:
\begin{equation}\label{eqIdempotentsE1}\begin{aligned}
	\bm{w} \cdot e_\lambda &= e_{\lambda^{(-1)^{\ell(w)}} \idd^{-2}} \cdot \bm{w},
	\\
	\bz{w} \cdot e_\lambda &= e_{\lambda^{(-1)^{\ell(w)}}} \cdot \bz{w},
	\\
	\bp{w} \cdot e_\lambda &= e_{\lambda^{(-1)^{\ell(w)}} \idd^{2}} \cdot \bp{w}.
\end{aligned}\end{equation}
These formulas can be easily computed from formulas \eqref{eqOmegaOnE1Left} and \eqref{eqOmegaOnE1Right}, and they are also proven in \cite[Equation (69)]{newpaper} (for $\lambda = \idd^m$ for some $m \in \zz$, i.e., for every $\lambda$).

\item Left action of $\tau_{s_0}$ and $\tau_{s_1}$ on $E^1$ when lengths add up:

Let $w \in \widetilde{W}$ (with the understanding that $\ell(w) \geqslant 1$ when we speak about $\bz{w}$). One has:
\begin{equation}\label{eqFormulasTauQp}\begin{aligned}
	&\left.\begin{array}{l}
		\tau_{s_0} \cdot \bm{w}
		=
		-\bp{s_0w}
	\\
		\tau_{s_0} \cdot \bz{w}
		=
		-\bz{s_0w}
	\\
		\tau_{s_0} \cdot \bp{w}
		=
		0
	\end{array}
	\quad\right\}\qquad\text{if $\ell(s_0 w) = \ell(w) + 1$,}
	\\
	&\left.\begin{array}{l}
		\tau_{s_1} \cdot \bm{w}
		=
		0
	\\
		\tau_{s_1} \cdot \bz{w}
		=
		-\bz{s_1 w}
	\\
		\tau_{s_1} \cdot \bp{w}
		=
		-\bm{s_1 w}
	\end{array}
	\quad\right\}\qquad\text{if $\ell(s_1 w) = \ell(w) + 1$.}
\end{aligned}\end{equation}
These formulas are proved in \cite[Proposition 4.9]{newpaper}, and they also follow from Lemma \ref{lmmExplicitFormulasE1} parts \ref{item2ExplFormulasE1} and \ref{item4ExplFormulasE1}.

\item Right action of $\tau_v$ on $E^1$ when lengths add up (for $v \in \widetilde{W}$):

Let $w, v \in \widetilde{W}$ such that $\ell(wv) = \ell(w) + \ell(v)$ and such that $\ell(v) \geqslant 1$ (and furthermore let us assume that $\ell(w) \geqslant 1$ when we speak about $\bz{w}$). One has:
\begin{equation}\label{eqFormulaRighteasyQp}\begin{aligned}
	&\left.\begin{array}{l}
		\bm{w} \cdot \tau_{v}
		=
		\bm{wv}
	\\
		\bz{w} \cdot \tau_{v}
		=
		\bz{wv}
	\\
		\bp{w} \cdot \tau_{v}
		=
		0
	\end{array}
	\quad\right\}\qquad\text{if $\ell(s_1 w v) = \ell(w v) + 1$,}
	\\
	&\left.\begin{array}{l}
		\bm{w} \cdot \tau_{v}
		=
		0
	\\
		\bz{w} \cdot \tau_{v}
		=
		\bz{wv}
	\\
		\bp{w} \cdot \tau_{v}
		=
		\bp{wv}
	\end{array}
	\quad\right\}\qquad\text{if $\ell(s_0 w v) = \ell(w v) + 1$.}
\end{aligned}\end{equation}
These formulas are proved in \cite[Lemma 4.12]{newpaper}, and they also follow from Lemma \ref{lmmExplicitFormulasE1} parts \ref{item1ExplFormulasE1} and \ref{item3ExplFormulasE1}.

\item Left action of $\tau_{s_0}$ and $\tau_{s_1}$ on $E^1$ when lengths do not add up:

	Let $w \in \widetilde{W}$ be such that $\ell(s_0w) = \ell(w) - 1$. One has:
	\begin{equation}\label{eqTauS0OnE1LeftNotAddUp}\begin{aligned}
		\tau_{s_0} \cdot \bm{w}
		&=
		\begin{cases}
			-e_1\bm{w} -2e_{\idd} \bz{w} -\bp{s_0w}
			&\text{if $\ell(w) \geqslant 2$,}
			\\
			-e_1\bm{w} -2e_{\idd} \bz{w} + e_{\idd^2} \bp{w} - \bp{s_0w}
			&\text{if $\ell(w) = 1$,}
		\end{cases}
	\\
		\tau_{s_0} \cdot \bz{w}
		&=
		\begin{cases}
			-e_1\bz{w}
			&\text{if $\ell(w) \geqslant 2$,}
			\\
			-e_1\bz{w} + e_{\idd} \bp{w}
			&\text{if $\ell(w) = 1$,}
		\end{cases}
	\\
		\tau_{s_0} \cdot \bp{w}
		&=
		-e_1\bp{w}.
	\end{aligned}\end{equation}

	Let $w \in \widetilde{W}$ be such that $\ell(s_1w) = \ell(w) - 1$. One has:
	\begin{equation}\label{eqTauS1OnE1LeftNotAddUp}\begin{aligned}
		\tau_{s_1} \cdot \bm{w}
		&=
		-e_1 \bm{w}
	\\	
		\tau_{s_1} \cdot \bz{w}
		&=
		\begin{cases}
			-e_1 \bz{w}
			&\text{if $\ell(w) \geqslant 2$,}
			\\
			-e_1 \bz{w} - e_{\idd^{-1}}\bm{w}
			&\text{if $\ell(w) = 1$,}
		\end{cases}
	\\
		\tau_{s_1} \cdot \bp{w}
		&=
		\begin{cases}
			-e_1 \bp{w} + 2e_{\idd^{-1}}\bz{w} - \bm{s_1w}
			&\text{if $\ell(w) \geqslant 2$,}
			\\
			-e_1 \bp{w} + 2e_{\idd^{-1}} \bz{w} + e_{\idd^{-2}} \bm{w} - \bm{s_1w}
			&\text{if $\ell(w) = 1$,}
		\end{cases}
	\end{aligned}\end{equation}

This is proved in \cite[Proposition 4.9]{newpaper}.

\item Right action of $\tau_{s_0}$ and $\tau_{s_1}$ on $E^1$ when lengths do not add up (some cases):

	Let $v \in \widetilde{W}$ such that $\ell(s_0v) = \ell(v) - 1$. One has:
	\begin{equation}\label{eqRightBadLengthS0}
		\bz{s_0} \cdot \tau_v = -e_1\bz{v} - e_{\idd^{-1}}\bm{v}.
	\end{equation}	

	Let $v \in \widetilde{W}$ such that $\ell(s_1v) = \ell(v) - 1$. One has:
	\begin{equation}\label{eqRightBadLengthS1}
		\bz{s_1} \cdot \tau_v = -e_1\bz{v} + e_{\idd}\bp{v}.
	\end{equation}

This is proved in \cite[Lemma 4.12]{newpaper}.

\item Action of the anti-involution $\invol$ on $E^1$ (see \cite[Lemma 4.7]{newpaper}):

Let $w \in \widetilde{W}$ (with the understanding that $\ell(w) \geqslant 1$ when we speak about $\bz{w}$) and let $u_w \in \ff_p^\times$ be such that $\omega_{u_w}^{-1}w$ lies in the subgroup of $\widetilde{W}$ generated by $s_0$ and $s_1$ (it is easy to see that such $u_w$ exists and, although it is not unique, it has the property that $u_w^2$ is uniquely determined by $w$; for the notation $\omega_?$ see \eqref{eqOmegaU}). One has:
\begin{equation}
	\label{eqInvolOnE1}
\begin{aligned}
	\invol (\bm{w})
	&=
	\begin{cases}
		u_w^2 \bm{w^{-1}} &\text{if $\ell(w)$ is even,}
		\\
		- u_w^{2} \bp{w^{-1}} &\text{if $\ell(w)$ is odd,}
	\end{cases}
\\
	\invol (\bz{w})
	&=
	\begin{cases}
		\bz{w^{-1}} &\text{if $\ell(w)$ is even,}
		\\
		-\bz{w^{-1}} &\text{if $\ell(w)$ is odd,}
	\end{cases}
\\	\invol (\bp{w})
	&=
	\begin{cases}
		u_w^{-2} \bp{w^{-1}} &\text{if $\ell(w)$ is even,}
		\\
		- u_w^{-2} \bm{w^{-1}} &\text{if $\ell(w)$ is odd.}
	\end{cases}
\end{aligned}
\end{equation}

\item Action of the involutive automorphism $\Gamma_\varpi$ on $E^1$ (see \cite[Lemma 4.4]{newpaper}):

Let $w \in \widetilde{W}$ (with the understanding that $\ell(w) \geqslant 1$ when we speak about $\bz{w}$). One has:
\begin{equation}\label{eqCCCNOnE1}\begin{aligned}
	\Gamma_\varpi (\bm{w})
	&=
	\bp{\varpi w \varpi^{-1}},
	\\
	\Gamma_\varpi (\bz{w})
	&=
	-\bz{\varpi w \varpi^{-1}},
	\\
	\Gamma_\varpi (\bp{w})
	&=
	\bm{\varpi w \varpi^{-1}}.
\end{aligned}\end{equation}

\item Left and right action of $\tau_{\omega}$ on $E^2$ for $\omega \in \quoz{T^0}{T^1}$:

Let $u \in \ff_p^\times$, let $w \in \widetilde{W}$ (with the understanding that $\ell(w) \geqslant 1$ when we speak about $\az{w}$). Recall the notation $\omega_u$ from \eqref{eqOmegaU}. One has:
\begin{equation}\label{eqTauOmegaE2Left}\begin{aligned}
	\tau_{\omega_u} \cdot \am{w}
	&=
	u^2 \am{\omega_u w},
\\
	\tau_{\omega_u} \cdot \az{w}
	&=
	\az{\omega_u w},
\\
	\tau_{\omega_u} \cdot \ap{w}
	&=
	u^{-2} \ap{\omega_u w},
\end{aligned}\end{equation}
and
\begin{equation}\label{eqTauOmegaE2Right}\begin{aligned}
	\am{w} \cdot \tau_{\omega_u}
	&=
	\am{w \omega_u},
\\
	\az{w} \cdot \tau_{\omega_u}
	&=
	\az{w \omega_u},
\\
	\ap{w} \cdot \tau_{\omega_u}
	&=
	\ap{w \omega_u}.
\end{aligned}\end{equation}
For the proof of the first formula see \cite[Equation (89)]{newpaper}. The second formula can be proved exactly in the same way using the corresponding formula for $E^1$, or from the first formula by using the anti-involution (see \eqref{eqInvolOnE2} later).
	
\item Action of the idempotents on $E^2$:

Let $\lambda \in \widehat{\quoz{T^0}{T^1}}$ and let $w \in \widetilde{W}$ (with the understanding that $\ell(w) \geqslant 1$ when we speak about $\az{w}$). One has:
\begin{equation}\label{eqFormulasIdempotentsLeftRight}
\begin{aligned}
	\am{w} \cdot e_\lambda
	&=
		e_{\lambda^{(-1)^{\ell(w)}} \cdot \idd^2} \cdot \am{w},
	\\
	\az{w} \cdot e_\lambda
	&=
		e_{\lambda^{(-1)^{\ell(w)}}} \cdot \az{w}
		\qquad\qquad\qquad\text{(if $\ell(w) \geqslant 1$)},
	\\
	\ap{w} \cdot e_\lambda
	&=
		e_{\lambda^{(-1)^{\ell(w)}} \cdot \idd^{-2}} \cdot \ap{w}.
\end{aligned}
\end{equation}
These formulas can be proved using \eqref{eqTauOmegaE2Left} and \eqref{eqTauOmegaE2Right}.

\item Left action of $\tau_{s_0}$ and $\tau_{s_1}$ when lengths add up (see \cite[Proposition 5.5]{newpaper}):

Let $w \in \widetilde{W}$ (with the understanding that $\ell(w) \geqslant 1$ when we speak about $\az{w}$). One has:
\begin{equation}\label{eqExplicitDeg2AddUp}\begin{aligned}
	&\left.\begin{array}{l}
		\tau_{s_0} \cdot \am{z}
		=
		0
	\\
		\tau_{s_0} \cdot \az{z}
		=
		0
	\\
		\tau_{s_0} \cdot \ap{z}
		=
		-\am{s_0 w}
	\end{array}
	\quad\right\}\qquad\text{if $\ell(s_0 w) = \ell(w) + 1$,}
	\\
	&\left.\begin{array}{l}
		\tau_{s_1} \cdot \am{z}
		=
		-\ap{s_1 w}
	\\
		\tau_{s_1} \cdot \az{z}
		=
		0
	\\
		\tau_{s_1} \cdot \ap{z}
		=
		0
	\end{array}
	\quad\right\}\qquad\text{if $\ell(s_1 w) = \ell(w) + 1$.}
\end{aligned}\end{equation}

\item Left action of $\tau_{s_0}$ and $\tau_{s_1}$ when lengths do not add up (see \cite[Proposition 5.5]{newpaper}):

	Let $w \in \widetilde{W}$ be such that $\ell(s_0w) = \ell(w) - 1$. One has:
	\begin{equation}\label{eqTauS0OnE2LeftNotAddUp}\begin{aligned}
		\tau_{s_0} \cdot \am{w}		
		&=
		- e_1 \am{w}
	\\
		\tau_{s_0} \cdot \az{w}
		&=
		\begin{cases}
			-e_1 \az{w} + 2 e_{\idd} \am{w} - \az{s_0 w}
			&\text{if $\ell(w) \geqslant 2$,}
			\\
			-e_1 \az{w} + 2 e_{\idd} \am{w}
			&\text{if $\ell(w) = 1$,}
		\end{cases}
	\\
		\tau_{s_0} \cdot \ap{w}
		&=
		\begin{cases}
			-e_1 \ap{w} - \am{s_0 w}
			&\text{if $\ell(w) \geqslant 2$,}
			\\
			- e_1 \ap{w} - e_{\idd} \az{w} + e_{\idd^2} \am{w} - \am{s_0 w}
			&\text{if $\ell(w) = 1$,}
		\end{cases}
	\end{aligned}\end{equation}

	Let $w \in \widetilde{W}$ be such that $\ell(s_1w) = \ell(w) - 1$. One has:
	\begin{equation}\label{eqTauS1OnE2LeftNotAddUp}\begin{aligned}
		\tau_{s_1} \cdot \am{w}
		&=
		\begin{cases}
			- e_1 \am{w} - \ap{s_1 w}
			&\text{if $\ell(w) \geqslant 2$,}
			\\
			- e_1 \am{w} + e_{\idd^{-1}} \az{w} + e_{\idd^{-2}} \ap{w} - \ap{s_1 w}
			&\text{if $\ell(w) = 1$,}
		\end{cases}
	\\
		\tau_{s_1} \cdot \az{w}
		&=
		\begin{cases}
			- e_1 \az{w} -2 e_{\idd^{-1}} \ap{w} - \az{s_1 w}
			&\text{if $\ell(w) \geqslant 2$,}
			\\
			- e_1 \az{w} -2 e_{\idd^{-1}} \ap{w}
			&\text{if $\ell(w) = 1$,}
		\end{cases}
	\\
		\tau_{s_1} \cdot \ap{w}
		&=
		- e_1 \ap{w}
	\end{aligned}\end{equation}

\item Action of the anti-involution $\invol$ on $E^2$ (see \cite[Equations (86) and (87)]{newpaper}):

Let $w \in \widetilde{W}$ (with the understanding that $\ell(w) \geqslant 1$ when we speak about $\az{w}$) and let $u_w \in \ff_p^\times$ as in \eqref{eqInvolOnE1}. One has:
\begin{equation}\label{eqInvolOnE2}\begin{aligned}
	\invol (\am{w})
	&=
	\begin{cases}
		u_w^{-2} \am{w^{-1}} &\text{if $\ell(w)$ is even,}
		\\
		- u_w^{-2} \ap{w^{-1}} &\text{if $\ell(w)$ is odd,}
	\end{cases}
\\
	\invol (\az{w})
	&=
	\begin{cases}
		\az{w^{-1}} &\text{if $\ell(w)$ is even,}
		\\
		-\az{w^{-1}} &\text{if $\ell(w)$ is odd,}
	\end{cases}
\\
	\invol (\ap{w})
	&=
	\begin{cases}
		u_w^{2} \ap{w^{-1}} &\text{if $\ell(w)$ is even,}
		\\
		- u_w^{2} \am{w^{-1}} &\text{if $\ell(w)$ is odd.}
	\end{cases}
\end{aligned}\end{equation}

\item Action of the involutive automorphism $\Gamma_\varpi$ on $E^2$ (see \cite[Lemma 5.2]{newpaper}):

Let $w \in \widetilde{W}$ (with the understanding that $\ell(w) \geqslant 1$ when we speak about $\az{w}$). One has:
\begin{equation}\label{eqCCCNOnE2}\begin{aligned}
	\Gamma_\varpi (\am{w})
	&=
	\ap{\varpi w \varpi^{-1}},
\\
	\Gamma_\varpi (\az{w})
	&=
	-\az{\varpi w \varpi^{-1}},
\\
	\Gamma_\varpi (\ap{w})
	&=
	\am{\varpi w \varpi^{-1}}.
\end{aligned}\end{equation}

\item Left and right action of $E^0$ on $E^3$:

For all $\omega \in \quoz{T^0}{T^1}$, for all $j \in \{0,1\}$ and for all $w \in \widetilde{W}$ one has
\begin{equation}\label{eqFormulasTopDegSL2}\begin{aligned}
	\tau_\omega \cdot \phi_w &= \phi_{\omega w},
	\\
	\phi_w \cdot \tau_\omega &= \phi_{w \omega},
	\\
	\tau_{s_j} \cdot \phi_w
	&=
	\begin{cases}
		\phi_{s_j w} - e_1 \cdot \phi_w &\text{if } \ell(s_j w) = \ell(w) - 1,\\
		0 &\text{if } \ell(s_j w) = \ell(w) + 1,
	\end{cases}
	\\
	\phi_w \cdot \tau_{s_j}
	&=
	\begin{cases}
		\phi_{w s_j} - \phi_w \cdot e_1 &\text{if } \ell(w s_j) = \ell(w) - 1,\\
		0 &\text{if } \ell(w s_j) = \ell(w) + 1.
	\end{cases}
\end{aligned}\end{equation}
These formulas follow from the general ones in \cite[Proposition 8.2]{ext}.

\item Action of the anti-involution $\invol$ on $E^3$ (see \cite[Equation (89)]{ext}):

For all $w \in \widetilde{W}$ one has
\begin{equation}\label{eqInvolOnE3}
	\invol(\phi_w) = \phi_{w^{-1}}.
\end{equation}

\item Action of the involutive automorphism $\Gamma_\varpi$ on $E^3$ (see \cite[Lemma 4.1]{newpaper}):

For all $w \in \widetilde{W}$ one has:
\begin{equation}\label{eqCCCNOnE3}
	\Gamma_\varpi(\phi_w) = \phi_{\varpi w \varpi^{-1}}.
\end{equation}
\end{itemize}

Recall that $E^1$ is finitely-generated as an $E^0$-bimodule (Proposition \ref{propE1fg}). In the next lemma we compute a (nice) finite set of generators, along with some useful formulas.

\begin{lmm}\label{lmmDeg1GeneratedByTheFourElements}
One has the following facts.
\begin{enumerate}[label=(\roman*)]
\item
	The following formulas hold:
	\begin{align}
	\bm{1} \cdot \tau_w &= \bm{w}
	\label{eqA}
	&&\substack{\text{for $w \in \widetilde{W}$ with $\ell(s_1w) = \ell(w) + 1$,}}
	\\
	\bp{1} \cdot \tau_w &= \bp{w}
	\label{eqB}
	&&\substack{\text{for $w \in \widetilde{W}$ with $\ell(s_0w) = \ell(w) + 1$,}}
	\\
	\tau_{(s_1s_0)^i} \cdot \bm{1} &= \bm{(s_1s_0)^i}
	\label{eqC}
	&&\substack{\text{for all $i \in \nn$,}}
	\\
	\tau_{s_0(s_1s_0)^i} \cdot \bm{1} &= - \bp{s_0(s_1s_0)^i}
	\label{eqD}
	&&\substack{\text{for all $i \in \nn$,}}
	\\
	\tau_{(s_0s_1)^i} \cdot \bp{1} &= \bp{(s_0s_1)^i}
	\label{eqE}
	&&\substack{\text{for all $i \in \nn$,}}
	\\
	\tau_{s_1(s_0s_1)^i} \cdot \bp{1} &= -\bm{s_1(s_0s_1)^i}
	\label{eqF}
	&&\substack{\text{for all $i \in \nn$,}}
	\\
	\bz{s_i} \cdot \tau_{w} &= \bz{s_iw}
	\label{eqG}
	&&\substack{\text{for $i \in \{0,1\}$ and $w \in \widetilde{W}$}\\\text{with $\ell(s_iw) = \ell(w) + 1$,}}
	\\
	\tau_w \cdot \bz{s_i} &= (-1)^{\ell(w)} \bz{ws_i}
	\label{eqH}
	&&\substack{\text{for $i \in \{0,1\}$ and $w \in \widetilde{W}$}\\\text{with $\ell(ws_i) = \ell(w) + 1$.}}
	\end{align}
\item
	The following elements generate $E^1$ as an $E^0$-bimodule:
	\begin{align*}
		&\bm{1},
		&&\bp{1},
		&&\bz{s_0},
		&&\bz{s_1}.
	\end{align*}
\end{enumerate}
\end{lmm}

\begin{proof}
Let us prove the two statements.
\begin{enumerate}[label=(\roman*)]
\item
	Formulas \eqref{eqA}, \eqref{eqB} and \eqref{eqG} are immediate consequence of the formulas describing the right action of $E^0$ when lengths add up (namely, formulas \eqref{eqFormulaRighteasyQp} in the case that $\ell(w) \geqslant 1$ and formulas \eqref{eqOmegaOnE1Right} in the case that $\ell(w)=0$). Formulas \eqref{eqC}, \eqref{eqD}, \eqref{eqE} and \eqref{eqF} can be shown using the formulas describing the left action of $E^0$ when lengths add up (namely, formulas \eqref{eqFormulasTauQp}). The same is true for formula \eqref{eqH}, also recalling the left action of $\tau_\omega$ for $\omega \in \quoz{T^0}{T^1}$ of formula \eqref{eqOmegaOnE1Left}.
\item
	This is a consequence of part (i). More precisely, for all $v \in \widetilde{W}$ we want to show that the elements $\bm{v}$, $\bp{v}$ and $\bz{v}$ (the last one if $\ell(v) \geqslant 1$) lie in the sub-$E^0$ bimodule generated by the four elements in the statement. For $\bz{v}$ this is clear form formula \eqref{eqG} (or from formula \eqref{eqH}). For $\bm{v}$ and $\bp{v}$, using \eqref{eqOmegaOnE1Right} (or \eqref{eqOmegaOnE1Left}), we might assume that $v$ is of the form $(s_1s_0)^i$, $s_0(s_1s_0)^i$, $(s_0s_1)^i$ or $s_1(s_0s_1)^i$ for some $i \in \nn$. Then we can apply the formulas in part (i) to conclude.
	\qedhere
\end{enumerate}
\end{proof}

\subsection{Outline of the strategy for the first main result}\label{subsecStrategy}

Let us consider the multiplication map 
\[
	\funcInline
		{\mult}
		{T_{E^0}^\ast E^1}
		{E^\ast,}
\]
where the left hand side is the tensor algebra generated by the $E^0$-bimodule $E^1$. Proposition \ref{propAlmostGenByE1} can be used to prove surjectivity of $\mult$, even if it does not directly imply such property (see Lemmas \ref{lmmDefRep2} and \ref{lmmDefRep3}).
The \virgolette{first main result} we want to prove is the computation of the kernel of $\mult$ (Theorem \ref{thmFinalResultKernel}), hence obtaining a \virgolette{presentation} of $E^\ast$.

This computation will be quite long and will require many intermediate statements, and we now outline the strategy of the proof, working in a general abstract setting to simplify the notation.

Let $A$ be an associative $k$-algebra with $1$ (in our case $A=E^0$); let $M$ and $N$ be $A$-bimodules (in our case, at first, $M = E^1 \otimes_{E^0} E^1$ and $N=E^2$) and let $\funcInline{\mult}{M}{N}$ be a surjective homomorphism of $A$-bimodules (in our case $\mult$ will be the above multiplication map, at first restricted to $E^1 \otimes_{E^0} E^1$). To compute the kernel of $\mult$ we fix the following:
\begin{itemize}
\item generators $(a_i)_{i \in \mathbf{I}}$ of $A$ as a $k$-algebra (for a suitable index set $\mathbf{I}$);
\item generators $(m_j)_{j \in \mathbf{J}}$ of $M$ as an $A$-bimodule (for a suitable index set $\mathbf{J}$);
\item a basis $(n_l)_{l \in \mathbf{L}}$ of $N$ as a $k$-vector space (for a suitable index set $\mathbf{L}$).
\end{itemize}
We fix a splitting $\rep$ of $\mult$ as a map of $k$-vector spaces:
\[\begin{tikzcd}[ampersand replacement = \&, column sep = 6em]
	M
	\ar[r, twoheadrightarrow, "{ \mult}"]
	\&
	N
	\ar[l, bend left = 50, dashed, "{ \rep}"']
\end{tikzcd}\]
Let $M'$ be the sub-$A$-bimodule of $M$ generated by the elements $\rep(a_i n_l) - a_i \rep(n_l)$ and by the elements $\rep(n_l a_i) - \rep(n_l) a_i$ for $i \in \mathbf{I}$ and $l \in \mathbf{L}$. The maps $\mult$ and $\rep$ define maps $\overline{\mult}$ and $\overline{\rep}$:
\[\begin{tikzcd}[ampersand replacement = \&, column sep = 9em]
	\quoz{M}{M'}
	\ar[r, twoheadrightarrow, "{ \overline{\mult}}", "{\overline{m} \mapsto \mult(m)}"']
	\&
	N,
	\ar[l, bend left = 50, dashed, "{ \overline{\rep}}"', "\overline{\rep(n)} \mapsfrom n"]
\end{tikzcd}\]
with $\overline{\rep}$ being a splitting of $\overline{\mult}$, this time not only as a homomorphism of $k$-vector spaces, but as a homomorphism of $A$-bimodules. It follows that the $A$-bimodule $\ker \overline{\mult}$ (i.e., $\quoz{(\ker \mult)}{M'}$) is generated by the elements $\overline{m_j - (\rep \circ \mult)(m_j)}$ for $j \in \mathbf{J}$, and we thus conclude that $\ker \mult$ is the sub-$A$-bimodule of $M$ generated by the following elements:
\begin{itemize}
\item $\rep(a_i n_l) - a_i \rep(n_l)$, for $i \in \mathbf{I}$ and $l \in \mathbf{L}$;
\item $\rep(n_l a_i) - \rep(n_l) a_i$, for $i \in \mathbf{I}$ and $l \in \mathbf{L}$;
\item $m_j - (\rep \circ \mult)(m_j)$, for $j \in \mathbf{J}$.
\end{itemize}

\subsection{The tensor algebra}

As explained, we work with the tensor algebra $T_{E^0}^\ast E^1$ and the multiplication map
$
	\funcInline
		{\mult}
		{T_{E^0}^\ast E^1}
		{E^\ast}
$.
We denote its graded pieces by
\[
	\funcInline
		{\mult_i}
		{T_{E^0}^i E^1}
		{E^i.}
\]

Let us define maps $\Gamma_\varpi$ and $\invol$ on $T_{E^0}^\ast E^1$ that are compatible with the namesake maps on $E^\ast$. We start by defining a map
\[
	\functor
		{\Gamma_\varpi}
		{T_{E^0}^\ast E^1}
		{T_{E^0}^\ast E^1}
		{E^0 \ni x}
		{\Gamma_\varpi(x),}
		{\beta_1 \otimes \dots \otimes \beta_i}
		{\Gamma_\varpi(\beta_1) \otimes \dots \otimes \Gamma_\varpi(\beta_i).}
\]
extending the map $\Gamma_\varpi$ defined on $E^0$ and $E^1$. We see that it is an involutive automorphism the graded $k$-algebra $T_{E^0}^\ast E^1$. Similarly, we define the map
\[
	\functor
		{\invol}
		{T_{E^0}^\ast E^1}
		{T_{E^0}^\ast E^1}
		{E^0 \ni x}
		{\invol(x),}
		{\beta_1 \otimes \dots \otimes \beta_i}
		{(-1)^{\frac{(i-1)i}{2}} \invol(\beta_i) \otimes \dots \otimes \invol(\beta_1).}
\]
Note that the sign $(-1)^{\frac{(i-1)i}{2}}$ is the sign of the permutation $\begin{psmallmatrix} 1 & \dots & i \\ i & \dots & 1 \end{psmallmatrix}$. The map $\invol$ that we have just defined is an involutive anti-automorphism on $T_{E^0}^\ast E^1$.

The above definitions are given in such a way that the following diagrams commute:
\begin{align}\label{eqMultiplicationAndCCCNAndInvol}
&\begin{tikzcd}[ampersand replacement = \&]
	T_{E^0}^\ast E^1
	\ar[r, "{ \mult}"]
	\ar[d, "{ \Gamma_\varpi}"']
	\&
	E^\ast
	\ar[d, "{ \Gamma_\varpi}"]
	\\
	T_{E^0}^\ast E^1
	\ar[r, "{ \mult}"']
	\&
	E^\ast.
\end{tikzcd}
&&\begin{tikzcd}[ampersand replacement = \&]
	T_{E^0}^\ast E^1
	\ar[r, "{ \mult}"]
	\ar[d, "{ \invol}"']
	\&
	E^\ast
	\ar[d, "{ \invol}"]
	\\
	T_{E^0}^\ast E^1
	\ar[r, "{ \mult}"']
	\&
	E^\ast.
\end{tikzcd}
\end{align}

As explained in the outline of the strategy in \S\ref{subsecStrategy}, we want to compute a (quite simple) set of generators of $T_{E_0}^2 E^1$ as an $E^0$-bimodule.

We recall that the following are generators of $E^1$ as an $E^0$-bimodule (see Lemma \ref{lmmDeg1GeneratedByTheFourElements}):
\begin{align*}
	&\bm{1},
	&&\bp{1},
	&&\bz{s_0},
	&&\bz{s_1}.
\end{align*}
It follows that $E^1 \otimes_{E^0} E^1$ is generated by the following elements as a left $E^0$-module (in particular also as an $E^0$-bimodule):
\begin{align}\label{eqGeneratorsE1tensorE1OnTheLeft}
	&\begin{aligned}
		&\bm{1} \otimes \beta^{\sigma}_{w},
		&&\qquad\bp{1} \otimes \beta^{\sigma}_{w},
		&&\qquad\bz{s_0} \otimes \beta^{\sigma}_{w},
		&&\qquad\bz{s_1} \otimes \beta^{\sigma}_{w}
	\end{aligned}
	\\&\qquad
		\text{for $w \in \widetilde{W}$ and $\sigma \in \{-,0,+\}$ (with $\ell(w) \geqslant 1$ in the case $\sigma = 0$).}
		\notag
\end{align}
In the next lemma, we compute a simpler set of $E^0$-bimodule generators.

\begin{lmm}\label{lmmGeneratorsT2}
The following elements generate $E^1 \otimes_{E^0} E^1$ as an $E^0$-bimodule:
\begin{align*}
&\begin{aligned}
	&\bm{1} \otimes \bm{1}, &&\qquad\bp{1} \otimes \bm{1}, &&\qquad\bz{s_0} \otimes \bm{1}, &&\qquad\bz{s_1} \otimes \bm{1},
	\\
	&\bm{1} \otimes \bp{1}, &&\qquad\bp{1} \otimes \bp{1}, &&\qquad\bz{s_0} \otimes \bp{1}, &&\qquad\bz{s_1} \otimes \bp{1},
	\\
	&\bm{1} \otimes \bz{s_0}, &&\qquad\bp{1} \otimes \bz{s_0}, &&\qquad\bz{s_0} \otimes \bz{s_0}, &&\qquad\bz{s_1} \otimes \bz{s_0},
	\\
	&\bm{1} \otimes \bz{s_1}, &&\qquad\bp{1} \otimes \bz{s_1}, &&\qquad\bz{s_0} \otimes \bz{s_1}, &&\qquad\bz{s_1} \otimes \bz{s_1},
\end{aligned}
\\	
	&\bp{1} \otimes \bm{(s_1s_0)^i} = \bp{(s_1s_0)^i} \otimes \bm{1} &&\mkern-170mu\text{for $i \in \zpiu$,}
	\\
	&\bp{1} \otimes \bm{s_1(s_0s_1)^i} = - \bp{s_1(s_0s_1)^i} \otimes \bp{1} &&\mkern-170mu\text{for $i \in \nn$,}
	\\
	&\bm{1} \otimes \bp{(s_0s_1)^i} = \bm{(s_0s_1)^i} \otimes \bp{1} &&\mkern-170mu\text{for $i \in \zpiu$,}
	\\
	&\bm{1} \otimes \bp{s_0(s_1s_0)^i} = - \bm{s_0(s_1s_0)^i} \otimes \bm{1} &&\mkern-170mu\text{for $i \in \nn$.}
\end{align*}
\end{lmm}

Note that the equalities in the last four lines follow from Lemma \ref{lmmDeg1GeneratedByTheFourElements}.

\begin{proof}
We start from the generators in \eqref{eqGeneratorsE1tensorE1OnTheLeft}, and we recall from Lemma \ref{lmmDeg1GeneratedByTheFourElements} that
\begin{align*}
	\bm{w} &= \bm{1} \cdot \tau_{w}
	&&\text{for $w \in \widetilde{W}$ such that $\ell(s_1w)=\ell(w)+1$,}
	\\
	\bp{w} &= \bp{1} \cdot \tau_{w}
	&&\text{for $w \in \widetilde{W}$ such that $\ell(s_0w)=\ell(w)+1$,}
	\\
	\bz{s_1w} &= \bz{s_1} \cdot \tau_{w}
	&&\text{for $w \in \widetilde{W}$ such that $\ell(s_1w)=\ell(w)+1$,}
	\\
	\bz{s_0w} &= \bz{s_0} \cdot \tau_{w}
	&&\text{for $w \in \widetilde{W}$ such that $\ell(s_0w)=\ell(w)+1$.}
\end{align*}
From this (and from the behaviour of multiplication by $\tau_\omega$ for $\omega \in \quoz{T^0}{T^1}$, see \eqref{eqOmegaOnE1Right}), we obtain that the following elements generate $E^1 \otimes_{E^0} E^1$ as an $E^0$-bimodule:
\begin{align*}
	&\gamma \otimes \bm{1}
	&&\text{for $\gamma \in \left\{ \bm{1}, \bp{1}, \bz{s_0}, \bz{s_1} \right\}$,}
	\\
	&\gamma \otimes \bp{1}
	&&\text{for $\gamma$ as above,}
	\\
	&\gamma \otimes \bz{s_0}
	&&\text{for $\gamma$ as above,}
	\\
	&\gamma \otimes \bz{s_1}
	&&\text{for $\gamma$ as above.}
	\\
	&\gamma \otimes \bm{(s_1s_0)^i}
	&&\text{for $\gamma$ as above and for $i \in \zpiu$,}
	\\
	&\gamma \otimes \bm{s_1(s_0s_1)^i}
	&&\text{for $\gamma$ as above and for $i \in \zpiu$,}
	\\
	&\gamma \otimes \bp{(s_0s_1)^i}
	&&\text{for $\gamma$ as above and for $i \in \zpiu$,}
	\\
	&\gamma \otimes \bp{s_0(s_1s_0)^i}
	&&\text{for $\gamma$ as above and for $i \in \zpiu$.}
\end{align*}
The first four of these lines consist exactly of the first four lines of the claimed generators in the statement of the lemma. Now, let us look at the remaining four lines: we certainly get the families of generators in the remaining four lines of the statement of the lemma, and we have to argue that the remaining generators are superfluous.
Up to changing signs if necessary, using the formulas of Lemma \ref{lmmDeg1GeneratedByTheFourElements}, we can rewrite the last four lines as:
\begin{align*}
	&(\gamma \cdot \tau_{(s_1s_0)^i}) \otimes \bm{1}
	&&\text{for $\gamma$ as above and for $i \in \zpiu$,}
	\\
	&(\gamma \cdot \tau_{s_1(s_0s_1)^i}) \otimes \bp{1}
	&&\text{for $\gamma$ as above and for $i \in \zpiu$,}
	\\
	&(\gamma \cdot \tau_{(s_0s_1)^i}) \otimes \bp{1}
	&&\text{for $\gamma$ as above and for $i \in \zpiu$,}
	\\
	&(\gamma \cdot \tau_{s_0(s_1s_0)^i}) \otimes \bm{1}
	&&\text{for $\gamma$ as above and for $i \in \zpiu$.}
\end{align*}
From Lemma \ref{lmmDeg1GeneratedByTheFourElements}, one has
\begin{align*}
	&
	\bm{1} \cdot \tau_{(s_1s_0)^i} = 0,
	&&
	\bm{1} \cdot \tau_{s_1(s_0s_1)^i} = 0,
	&&
	\bp{1} \cdot \tau_{(s_0s_1)^i} = 0,
	&&
	\bp{1} \cdot \tau_{s_0(s_1s_0)^i} = 0.
\end{align*}
This shows that some of the remaining generators are superfluous, and now it remains to study the cases where $\gamma \in \{\bz{s_0}, \bz{s_1}\}$.

Let us start with the following case:
\begin{align*}
	(\bz{s_0} \cdot \tau_{(s_1s_0)^i}) \otimes \bm{1}
	&=
	\bz{s_0(s_1s_0)^i} \otimes \bm{1}
	\\&=
	\tau_{(s_0s_1)^i} \bz{s_0} \otimes \bm{1},
\end{align*}
where we used the formulas \eqref{eqG} and \eqref{eqH}. We see that the element $\bz{s_0} \otimes \bm{1}$ is a generator already appearing in the list in the statement of the lemma, and so the generator $(\bz{s_0} \cdot \tau_{(s_1s_0)^i}) \otimes \bm{1}$ is superfluous. The other three cases where lengths add up are similar.

Now it remains to prove that the generators
\begin{align*}
	&(\bz{s_1} \cdot \tau_{(s_1s_0)^i}) \otimes \bm{1},
	&&(\bz{s_1} \cdot \tau_{s_1(s_0s_1)^i}) \otimes \bp{1},
	\\&(\bz{s_0} \cdot \tau_{(s_0s_1)^i}) \otimes \bp{1},
	&&(\bz{s_0} \cdot \tau_{s_0(s_1s_0)^i}) \otimes \bm{1}
\end{align*}
are superfluous. To this end, we can use the formulas \eqref{eqRightBadLengthS0} and \eqref{eqRightBadLengthS1}: we spell the details for the first generator, the other three being similar. 
\begin{align*}
	(\bz{s_1} \cdot \tau_{(s_1s_0)^i}) \otimes \bm{1}
	&=
	\left( -e_1 \cdot \bz{(s_1s_0)^i} + e_{\idd} \cdot \bp{(s_1s_0)^i} \right) \otimes \bm{1}
	\\&=
	e_1 \cdot \tau_{(s_1s_0)^{i-1}s_1} \cdot \bz{s_0} \otimes \bm{1} + e_{\idd} \cdot \bp{(s_1s_0)^i} \otimes \bm{1},
\end{align*}
where, besides \eqref{eqRightBadLengthS1}, we have also used \eqref{eqH}.
\end{proof}

\subsection{The \texorpdfstring{$2$\textsuperscript{nd}}{2nd} graded piece of the kernel}

The multiplication map $\mult_2$ is surjective: indeed, Proposition \ref{propAlmostGenByE1} shows that $H^2(I,\xx(w))$ is contained in its image for all $w \in \widetilde{W}$ with $\ell(w) \geqslant 1$, and for all $\omega \in T^0/T^1$ (i.e., $\omega \in \widetilde{W}$ with $\ell(\omega) = 0$) one concludes with the following consequence of formulas \eqref{eqTauS0OnE2LeftNotAddUp} and \eqref{eqTauS1OnE2LeftNotAddUp}:
\begin{equation}\label{eqToQuoteDeg2E}\begin{aligned}
	\tau_{s_0} \cdot \ap{s_0^{-1}\omega} + \am{\omega}
	&\in
	\bigoplus_{\substack{v \in \widetilde{W}\\\text{s.t. $\ell(v)=1$}}}
		H^2(I,\xx(v)),
	\\
	\tau_{s_1} \cdot \am{s_1^{-1}\omega} + \ap{\omega}
	&\in
	\bigoplus_{\substack{v \in \widetilde{W}\\\text{s.t. $\ell(v)=1$}}}
		H^2(I,\xx(v)).
\end{aligned}\end{equation}
The purpose of this subsection is to compute the kernel of $\mult_2$, and following the outline in \S\ref{subsecStrategy} we start by constructing a section of $\mult_2$ as a homomorphism of $k$-vector space.

\begin{lmm}\label{lmmDefRep2}
The multiplication map $\funcInline{\mult_2}{T^2_{E^0} E^1}{E^2}$ is surjective, and the following is a section of it as a homomorphism of $k$-vector spaces.
\[\begin{tikzcd}[row sep = -0.6em]
\rep_2 \colon &[-3em] E^2 \ar[r] &[-1em] T_{E^0}^2E^1 = E^1 \otimes_{E^0} E^1 &[-4.5em]
\\
& \am{s_1v} \ar[r, mapsto] & - \bp{1} \otimes \bz{s_1v} &&\scriptstyle\text{if $\ell(s_1v) = \ell(v) + 1$,}
\\
& \az{s_1v} \ar[r, mapsto] & \bp{1} \otimes \bm{s_1v} &&\scriptstyle\text{if $\ell(s_1v) = \ell(v) + 1$,}
\\
& \ap{s_1v} \ar[r, mapsto] & \bz{s_1} \otimes \bp{v} &&\scriptstyle\text{if $\ell(s_1v) = \ell(v) + 1$,}
\\
& \am{s_0w} \ar[r, mapsto] & - \bz{s_0} \otimes \bm{w} &&\scriptstyle\text{if $\ell(s_0w) = \ell(w) + 1$,}
\\
& \az{s_0w} \ar[r, mapsto] & - \bm{1} \otimes \bp{s_0w} &&\scriptstyle\text{if $\ell(s_0w) = \ell(w) + 1$,}
\\
& \ap{s_0w} \ar[r, mapsto] & \bm{1} \otimes \bz{s_0w} &&\scriptstyle\text{if $\ell(s_0w) = \ell(w) + 1$,}
\\
& \am{\omega} \ar[r, mapsto] & \rep_2 \big( \am{\omega} + \tau_{s_0} \cdot \ap{s_0^{-1}\omega} \big) - \tau_{s_0} \cdot \rep_2 \big( \ap{s_0^{-1}\omega} \big) &&\scriptstyle\text{for $\omega \in \quoz{T^0}{T^1}$,}
\\
& \ap{\omega} \ar[r, mapsto] & \rep_2 \big( \ap{\omega} + \tau_{s_1} \cdot \am{s_1^{-1}\omega} \big) - \tau_{s_1} \cdot \rep_2 \big( \am{s_1^{-1}\omega} \big) &&\scriptstyle\text{for $\omega \in \quoz{T^0}{T^1}$.}
\end{tikzcd}\]
\end{lmm}
Note that the last two lines make sense: by \eqref{eqToQuoteDeg2E}, on the right hand side of such lines only the values of $\rep_2$ on \( \bigoplus_{\substack{v \in \widetilde{W}\\\text{s.t. $\ell(v)=1$}}} H^2(I,\xx(v)) \) are used, and on such subspace $\rep_2$ is defined in the preceding lines.

\begin{proof}
Combining formula \eqref{eqFirstBigComputation} in the proof of Proposition \ref{propAlmostGenByE1} with \eqref{eqFormulaCupUniform}, we obtain
\begin{align*}
	\bz{(s_0s_1)^i} \cdot \bp{1}
	&=
	\bz{(s_0s_1)^i} \cupprod \bp{(s_0s_1)^i}
	=
	\am{(s_0s_1)^i},
\\
	\bm{1} \cdot \tau_{(s_0s_1)^i} \cdot \bp{1}
	&=
	\bm{(s_0s_1)^i} \cupprod \bp{(s_0s_1)^i}
	=
	- \az{(s_0s_1)^i}.
\\
	\bm{1} \cdot \bz{(s_0s_1)^i}
	&=
	\bm{(s_0s_1)^i} \cupprod \bz{(s_0s_1)^i}
	=
	\ap{(s_0s_1)^i},
\end{align*}
We rewrite this as follows:
\begin{align*}
	\am{(s_0s_1)^i}
	&=
	\bz{(s_0s_1)^i} \cdot \bp{1}
	=
	\bz{s_0} \cdot \tau_{s_1(s_0s_1)^{i-1}} \cdot \bp{1}
	=
	- \bz{s_0} \cdot \bm{s_1(s_0s_1)^{i-1}},
\\
	\az{(s_0s_1)^i}
	&=
	- \bm{1} \cdot \tau_{(s_0s_1)^i} \cdot \bp{1}
	=
	- \bm{1} \cdot \bp{(s_0s_1)^i},
\\
	\ap{(s_0s_1)^i}
	&=
	\bm{1} \cdot \bz{(s_0s_1)^i},
\end{align*}
where we have used the formulas of Lemma \ref{lmmDeg1GeneratedByTheFourElements}.

This proves that the map $\mult_2 \circ \rep_2$ is the identity at $\am{s_0w}$, $\az{s_0w}$, $\ap{s_0w}$ for all $w$ of the form $s_1 (s_0s_1)^{i-1}$ for some $i \in \zpiu$. From the formulas \eqref{eqOmegaOnE1Right} it follows that the same is true for $w$ of the form $s_1 (s_0s_1)^{i-1} \omega$ for some $i \in \zpiu$ and some $\omega \in \quoz{T^0}{T^1}$. The proof that the same claim is true for all $w$ of the form $(s_1s_0)^{i-1}\omega$ for some $i \in \zpiu$ and some $\omega \in \quoz{T^0}{T^1}$ is similar, and this shows that the map $\mult_2 \circ \rep_2$ is the identity at $\am{s_0w}$, $\az{s_0w}$, $\ap{s_0w}$ for all $w \in \widetilde{W}$ such that $\ell(s_0w) = \ell(w) + 1$.

Similarly one shows that $\mult_2 \circ \rep_2$ is the identity at $\am{s_1v}$, $\az{s_1v}$, $\ap{s_1v}$ for all $v \in \widetilde{W}$ such that $\ell(s_1v) = \ell(v) + 1$. Finally, the fact that $\mult_2 \circ \rep_2$ is the identity at $\am{\omega}$ and $\ap{\omega}$ for all $\omega \in \Omega$ follows from the fact that $\mult_2 \circ \rep_2$ is the identity on \( \bigoplus_{\substack{v \in \widetilde{W}\\\text{s.t. $\ell(v)=1$}}} H^2(I,\xx(v)) \).
\end{proof}

For the next lemma, we use the identification of $k$-algebras
\[
	\spann_k \set{\tau_\omega}{\omega \in T^0/T^1} \cong k[T^0/T^1].
\]

\begin{lmm}\label{lmmBimodulesGpAlgebraFiniteTorus}
The map $\rep_2$ is a homomorphism of $k[\quoz{T^0}{T^1}]$-bimodules.
\end{lmm}

\begin{proof}
From the definition of $\rep_2$ and from the formulas for the right action of $k[\quoz{T^0}{T^1}]$ on $E^1$ and $E^2$ (see \eqref{eqOmegaOnE1Right} \eqref{eqTauOmegaE2Right}), it is clear that $\rep_2$ is a homomorphism of right $k[\quoz{T^0}{T^1}]$-modules.

Regarding the left action, from the formulas \eqref{eqTauOmegaE2Left} and \eqref{eqTauOmegaE2Right} describing the action of $k[\quoz{T^0}{T^1}]$ on $E^2$, we see that for all $\omega \in \quoz{T^0}{T^1}$, for all $\sigma \in \{-,0,+\}$ and for all $w \in \widetilde{W}$ (of length greater or equal than $1$ if $\sigma = 0$) we have
\[
	\tau_\omega \cdot \alpha_w^\sigma = c_{\omega, \sigma} \alpha_w^\sigma \cdot \tau_{\omega^{(-1)^{\ell(w)}}},
\]
for some $c_{\omega, \sigma} \in k^\times$. Assume now that $\ell(w) \geqslant 1$. Looking at the definition of $\rep_2$ in Lemma \ref{lmmDefRep2}, and using the formulas \eqref{eqOmegaOnE1Left} and \eqref{eqOmegaOnE1Right} describing the action of $k[\quoz{T^0}{T^1}]$ on $E^1$, we see that we also have the completely analogous result
\[
	\tau_\omega \cdot \rep_2(\alpha_w^\sigma) = c_{\omega, \sigma}' \rep_2(\alpha_w^\sigma) \cdot \tau_{\omega^{(-1)^{\ell(w)}}},
\]
for some $c_{\omega, \sigma}' \in k^\times$. Using that $\rep_2$ is a homomorphism of right $k[\quoz{T^0}{T^1}]$-modules, we thus see that
\[
	\tau_\omega \cdot \rep_2(\alpha_w^\sigma) = c_{\omega, \sigma}' c_{\omega, \sigma}^{-1} \rep_2(\tau_\omega \cdot \alpha_w^\sigma),
\]
and by applying $\mult_2$ to both sides, we see that the coefficient $c_{\omega, \sigma}' c_{\omega, \sigma}^{-1}$ must be $1$.

So far, we have shown that the maps $\tau_\omega \cdot \rep_2({}_-)$ and $\rep_2(\tau_\omega \cdot {}_-)$ coincide on the subspace \( \bigoplus_{\substack{v \in \widetilde{W}\\\text{s.t. $\ell(v) \geqslant 1$}}} H^2(I,\xx(v)) \), and looking at the definition of $\rep_2$ on the subspace \( \bigoplus_{\substack{v \in \widetilde{W}\\\text{s.t. $\ell(v)=0$}}} H^2(I,\xx(v)) \), we then deduce that the maps $\tau_\omega \cdot \rep_2({}_-)$ and $\rep_2(\tau_\omega \cdot {}_-)$ coincide on the whole $E^2$.
\end{proof}

\begin{lmm}\label{lmmCommutesCCCN}
The maps $\rep_2$ and $\Gamma_\varpi$ commute. More precisely, the following diagram is commutative:
\[\begin{tikzcd}
	E^2
	\ar[r, "{ \rep_2}"]
	\ar[d, "{ \Gamma_\varpi}"]
	&
	E^1 \otimes_{E^0} E^1
	\ar[d, "{ \Gamma_\varpi}"]
	\\
	E^2
	\ar[r, "{ \rep_2}"]
	&
	E^1 \otimes_{E^0} E^1.
\end{tikzcd}\]
\end{lmm}

\begin{proof}
Let $v \in \widetilde{W}$ with $\ell(s_1v) = \ell(v) + 1$. We compute
\begin{align*}
\Gamma_\varpi \big( \rep_2(\am{s_1v}) \big)
&=
\Gamma_\varpi \big( - \bp{1} \otimes \bz{s_1v} \big)
\\&=
\bm{1} \otimes \bz{s_0 \varpi v \varpi^{-1}},
\\
\Gamma_\varpi \big( \rep_2(\az{s_1v}) \big)
&=
\Gamma_\varpi \big( \bp{1} \otimes \bm{s_1v} \big)
\\&=
\bm{1} \otimes \bp{s_0 \varpi v \varpi^{-1}},
\\
\Gamma_\varpi \big( \rep_2(\ap{s_1v}) \big)
&=
\Gamma_\varpi \big( \bz{s_1} \otimes \bp{v} \big)
\\&=
- \bz{s_0} \otimes \bm{\varpi v \varpi^{-1}}.
\end{align*}
Here we have used the definition of $\Gamma_\varpi$ on $T^\ast_{E^0} E^1$, the formulas for the action of $\Gamma_\varpi$ on $E^1$ \eqref{eqCCCNOnE1}, and the fact that $\varpi s_1 \varpi^{-1} = s_0$.
On the other side, we compute
\begin{align*}
\rep_2 \big( \Gamma_\varpi(\am{s_1v}) \big)
&=
\rep_2 (\ap{s_0 \varpi v \varpi^{-1}}),
\\
\rep_2 \big( \Gamma_\varpi(\az{s_1v}) \big)
&=
- \rep_2 (\az{s_0 \varpi v \varpi^{-1}}),
\\
\rep_2 \big( \Gamma_\varpi(\ap{s_1v}) \big)
&=
\rep_2 (\am{s_0 \varpi v \varpi^{-1}}),
\end{align*}
where we have used the formulas for the action of $\Gamma_\varpi$ on $E^2$ \eqref{eqCCCNOnE2}. Comparing the two above computations with the definition of $\rep_2$, we see that the maps $\Gamma_\varpi \circ \rep_2$ and $\rep_2 \circ \Gamma_\varpi$ coincide at $\am{s_1v}$, $\az{s_1v}$ and $\ap{s_1v}$.

We also see that
\begin{align*}
	\Gamma_\varpi(\rep_2(\ap{\omega}))
	&=
	\Gamma_\varpi \left( \rep_2 \big( \ap{\omega} + \tau_{s_1\omega} \cdot \am{s_1^{-1}\omega} \big) - \tau_{s_1} \cdot \rep_2 \big( \am{s_1^{-1}\omega} \big) \right)
	\\&=
	\Gamma_\varpi \left( \rep_2 \big( \ap{\omega} + \tau_{s_1} \cdot \am{s_1^{-1}\omega} \big) \right) - \tau_{s_0} \cdot \Gamma_\varpi \left( \rep_2 \big( \am{s_1^{-1}\omega} \big) \right)
	\\&=
	\rep_2 \left( \Gamma_\varpi \big( \ap{\omega} + \tau_{s_1} \cdot \am{s_1^{-1}\omega} \big) \right) - \tau_{s_0} \cdot \rep_2 \left( \Gamma_\varpi \big( \am{s_1^{-1}\omega} \big) \right)
	\\&\qquad\left( \substack{\text{since $\Gamma_\varpi \circ \rep_2$ and $\rep_2 \circ \Gamma_\varpi$ coincide on $\bigoplus_{\vartheta \in \quoz{T^0}{T^1}} H^2(I,\xx(s_1 \vartheta))$}} \right)
	\\&=
	\rep_2 \big( \am{\varpi \omega \varpi^{-1}} + \tau_{s_0} \cdot \ap{s_0^{-1}\varpi \omega \varpi^{-1}} \big) - \tau_{s_0} \cdot \rep_2 \big( \ap{s_0^{-1}\varpi \omega \varpi^{-1}} \big)
	\\&=
	\rep_2 \left( \am{\varpi \omega \varpi^{-1}} \right)
	\\&=
	\rep_2(\Gamma_\varpi(\ap{\omega})),
\end{align*}
where we have used the same formulas as before (and the fact that $\Gamma_\varpi(\tau_{s_1}) = \tau_{s_0}$, see \eqref{eqCCCNOnH}).

So far, we have shown that the maps $\Gamma_\varpi \circ \rep_2$ and $\rep_2 \circ \Gamma_\varpi$ coincide on \virgolette{half} of the elements of the $k$-basis of $E^2$ used in the definition of $\rep_2$.
To conclude the proof, we remark that, since $\Gamma_\varpi$ is an involution, if the maps $\Gamma_\varpi \circ \rep_2$ and $\rep_2 \circ \Gamma_\varpi$ coincide on an element $x \in E^2$, then they coincide also on $\Gamma_\varpi(x)$.
\end{proof}

\begin{lmm}\label{lmmInvolRep2}
Let us consider the following $\invol$-invariant $k$-subspace of $E^2$:
\[
	F^1E^2
	\defeq
	\bigoplus_{\mathclap{\substack{w \in \widetilde{W} \\\text{s.t. } \ell(w) \geqslant 1}}} H^2(I,\xx(w)).
\]
One has that map $\restr{\rep_2}{F^1E^2}$ commutes with the anti-involution $\invol$. More precisely, the following diagram is commutative:
\[\begin{tikzcd}[row sep = 2.7em, column sep = 4em]
	F^1E^2
	\ar[r, "{ \restr{\rep_2}{F^1E^2}}"]
	\ar[d, "{ \invol}"']
	&
	E^1 \otimes_{E^0} E^1
	\ar[d, "{ \invol}", "{\substack{ \beta \otimes \beta' \\ \downmapsto \\ - \invol(\beta') \otimes \invol(\beta) }}"']
	\\
	F^1E^2
	\ar[r, "{ \restr{\rep_2}{F^1E^2}}"']
	&
	E^1 \otimes_{E^0} E^1.
\end{tikzcd}\]
\end{lmm}

\begin{proof}
We first recall from \eqref{eqCCCNInvolCommute} that the involutions $\invol$ and $\Gamma_\varpi$ commute on $E^\ast$. Using this property, let us see that if the maps $\invol \circ \rep_2$ and $\rep_2 \circ \invol$ coincide at a certain $\alpha \in E^2$, then they also coincide at $\Gamma_\varpi(\alpha)$: indeed, one has
\begin{align*}
	(\invol \circ \rep_2) (\Gamma_\varpi(\alpha))
	&=
	(\invol \circ \Gamma_\varpi \circ \rep_2)(\alpha)
	&&\text{as $\rep_2$ and $\Gamma_\varpi$ commute (Lemma \ref{lmmCommutesCCCN})}
	\\&=
	(\Gamma_\varpi \circ \invol \circ \rep_2)(\alpha)
	&&\begin{matrix*} \text{as $\invol$ and $\Gamma_\varpi$ commute on $E^1$} \\ \text{and hence on $E^1 \otimes_{E^0} E^1$} \end{matrix*}
	\\&=
	(\Gamma_\varpi \circ \rep_2 \circ \invol)(\alpha)
	&&\text{by assumption}
	\\&=
	(\rep_2 \circ \Gamma_\varpi \circ \invol)(\alpha)
	&&\text{as $\rep_2$ and $\Gamma_\varpi$ commute (Lemma \ref{lmmCommutesCCCN})}
	\\&=
	(\rep_2 \circ \invol) (\Gamma_\varpi(\alpha))
	&&\text{as $\invol$ and $\Gamma_\varpi$ commute on $E^2$.}
\end{align*}
Now, let us look at the definition of $\rep_2$ in Lemma \ref{lmmDefRep2}. By what we have just remarked, we only need to show that $\invol \circ \rep_2$ and $\rep_2 \circ \invol$ coincide at $\am{s_1v}$, $\az{s_1v}$ and $\ap{s_1v}$ for all $v \in \widetilde{W}$ such that $\ell(s_1v) = \ell(v) + 1$. Since $\rep_2$ is a homomorphism of $k[\quoz{T^0}{T^1}]$-bimodules (Lemma \ref{lmmBimodulesGpAlgebraFiniteTorus}) and since $\funcInline{\invol}{E^\ast}{E^\ast}$ and $\funcInline{\invol}{T^\ast_{E^0}E^1}{T^\ast_{E^0}E^1}$ are anti-multiplicative, it follows that $\invol \circ \rep_2 \circ \invol$ is a homomorphism of $k[\quoz{T^0}{T^1}]$-bimodules, and so (looking at formulas \eqref{eqTauOmegaE2Right}) we can assume that $v$ is of the form $(s_0s_1)^i$ or $s_0(s_1s_0)^i$ for some $i \in \nn$.
Using the formulas for the action of $\invol$ on $E^2$ \eqref{eqInvolOnE2} as well as the definition of $\rep_2$ and recalling the definition of $c_{-1}$ from \eqref{eqcminus1}, we compute
\begin{align*}
	\rep_2 \big(\invol( \am{s_1(s_0s_1)^i} ) \big)
	&=
	- \rep_2 \big( \ap{s_1(s_0s_1)^ic_{-1}} \big)
	=
	- \bz{s_1} \otimes \bp{(s_0s_1)^ic_{-1}}.
\end{align*}
Similarly, using the formulas for the action of $\invol$ on $E^1$ \eqref{eqInvolOnE1} (together with the definition of $\invol$ on $E^1 \otimes_{E^0} E^1$ and the definition of $\rep_2$), we compute
\begin{align*}
	\invol \big( \rep_2( \am{s_1(s_0s_1)^i} ) \big)
	&=
	- \invol \big( \bp{1} \otimes \bz{s_1(s_0s_1)^i} \big)
	=
	- \bz{s_1(s_0s_1)^ic_{-1}} \otimes \bp{1}.
\end{align*}
And, using the formulas in Lemma \ref{lmmDeg1GeneratedByTheFourElements}, we conclude that the two simple tensors we have obtained are equal, and hence that $\rep_2 \circ \invol$ and $\invol \circ \rep_2$ coincide at $\am{s_1(s_0s_1)^i}$.

Similarly, one computes that $\rep_2 \circ \invol$ and $\invol \circ \rep_2$ coincide at $\az{s_1(s_0s_1)^i}$, $\ap{s_1(s_0s_1)^i}$, $\am{(s_1s_0)^{i+1}}$, $\az{(s_1s_0)^{i+1}}$ and $\ap{(s_1s_0)^{i+1}}$.
\end{proof}

The above result leaves open the question of whether $\rep_2$ and $\invol$ could commute on the whole $E^2$. This is dealt with in the following lemma and remark.

\begin{lmm}\label{lmmInvolBadBehaviour}
One has the following equalities:
\begin{align*}
	(\rep_2 \circ \invol)(\am{1}) - (\invol \circ \rep_2)(\am{1})
	&=
	\bp{s_0} \otimes \bz{s_0^{-1}}
	+
	\bz{s_0} \otimes \bm{s_0^{-1}}
	\\&=
	- \tau_{s_0} \cdot \bm{1} \otimes \bz{s_0^{-1}}
	+
	\bz{s_0} \otimes \bm{1} \cdot \tau_{s_0^{-1}}
\\
	(\rep_2 \circ \invol)(\ap{1}) - (\invol \circ \rep_2)(\ap{1})
	&=
	- \bm{s_1} \otimes \bz{s_1^{-1}}
	-
	\bz{s_1} \otimes \bp{s_1^{-1}}
	\\&=
	\tau_{s_1} \cdot \bp{1} \otimes \bz{s_1^{-1}}
	-
	\bz{s_1} \otimes \bp{1} \cdot \tau_{s_1^{-1}}
\end{align*}
\end{lmm}

\begin{proof}
First note that the second formula can be obtained from the first one by applying $\Gamma_\varpi$: indeed $\Gamma_\varpi$ commutes with $\rep_2$ (Lemma \ref{lmmCommutesCCCN}) and with $\invol$ (see \eqref{eqCCCNInvolCommute}), and $\Gamma_\varpi(\am{1}) = \ap{1}$ (see \eqref{eqCCCNOnE2}); hence applying $\Gamma_\varpi$ on $(\rep_2 \circ \invol)(\am{1}) - (\invol \circ \rep_2)(\am{1})$ we obtain exactly $(\rep_2 \circ \invol)(\ap{1}) - (\invol \circ \rep_2)(\ap{1})$. Using the formulas for the action of $\Gamma_\varpi$ on $E^0$ \eqref{eqCCCNOnH} and on $E^1$ \eqref{eqCCCNOnE1}, we also see that the terms on the right of the equations correspond.

The equality
$
	\bp{s_0} \otimes \bz{s_0^{-1}}
	+
	\bz{s_0} \otimes \bm{s_0^{-1}}
	=
	- \tau_{s_0} \cdot \bm{1} \otimes \bz{s_0^{-1}}
	+
	\bz{s_0} \otimes \bm{1} \cdot \tau_{s_0^{-1}}
$
is immediate from the formulas \eqref{eqFormulasTauQp} and \eqref{eqFormulaRighteasyQp}. Therefore, to prove the lemma it remains to check that
\[
	(\rep_2 \circ \invol)(\am{1}) - (\invol \circ \rep_2)(\am{1})
	=
	- \tau_{s_0} \cdot \bm{1} \otimes \bz{s_0^{-1}}
	+
	\bz{s_0} \otimes \bm{1} \cdot \tau_{s_0^{-1}}
\]

Using \eqref{eqInvolOnE2} and the definition of $\rep_2$, we see that
\begin{align*}
	(\rep_2 \circ \invol)(\am{1})
	&=
	\rep_2(\am{1})
	\\&=
	\rep_2 \big( \am{1} + \tau_{s_0} \cdot \ap{s_0^{-1}} \big) - \tau_{s_0} \cdot \rep_2 ( \ap{s_0^{-1}} ).
\end{align*}
On the other side, using that the maps $\invol \circ \rep_2$ and $\rep_2 \circ \invol$ coincide on $F^1E^2$, we deduce that
\begin{equation}\label{eqInvolRep2}\begin{aligned}
	(\invol \circ \rep_2)(\am{1})
	&=
	(\invol \circ \rep_2) \big( \am{1} + \tau_{s_0} \cdot \ap{s_0^{-1}} \big)
	-
	(\invol \circ \rep_2) ( \ap{s_0^{-1}} ) \cdot \tau_{s_0^{-1}}
	\\&=
	(\rep_2 \circ \invol) \big( \am{1} + \tau_{s_0} \cdot \ap{s_0^{-1}} \big)
	-
	(\rep_2 \circ \invol) ( \ap{s_0^{-1}} ) \cdot \tau_{s_0^{-1}}
	\\&=
	\rep_2 \big( \am{1} + \invol (\tau_{s_0} \cdot \ap{s_0^{-1}}) \big)
	+
	\rep_2 ( \am{s_0} ) \cdot \tau_{s_0^{-1}},
\end{aligned}\end{equation}
where we have also used that $\invol$ is anti-multiplicative, that $\invol(\tau_{s_0}) = \tau_{s_0^{-1}}$ by \eqref{eqInvolTauw} and, again, that $\invol(\am{1}) = \am{1}$. We thus deduce that
\begin{equation}\label{eqRepInvolInvolRep}\begin{aligned}
	&(\rep_2 \circ \invol)(\am{1}) - (\invol \circ \rep_2)(\am{1})
	\\&\qquad=
	\rep_2 \big( \tau_{s_0} \cdot \ap{s_0^{-1}} - \invol (\tau_{s_0} \cdot \ap{s_0^{-1}}) \big)
	-
	\tau_{s_0} \cdot \rep_2 ( \ap{s_0^{-1}} )
	-
	\rep_2 ( \am{s_0} ) \cdot \tau_{s_0^{-1}}
\end{aligned}\end{equation}
We now focus on the first term in this sum. Using \eqref{eqTauS0OnE2LeftNotAddUp}, we compute
\[
	\tau_{s_0} \cdot \ap{s_0^{-1}}
	=
	- e_1 \ap{s_0^{-1}} - e_{\idd} \az{s_0^{-1}} + e_{\idd^2} \am{s_0^{-1}} - \am{1}.
\]
Using formulas \eqref{eqInvolOnE2} and \eqref{eqFormulasIdempotentsLeftRight} and using the fact that $\invol(e_\lambda) = e_{\lambda^{-1}}$ for all $\lambda \in \widehat{\quoz{T^0}{T^1}}$ (which follows from \eqref{eqInvolTauw}), we compute
\begin{align*}
	\invol \big( \tau_{s_0} \cdot \ap{s_0^{-1}} \big)
	&=
	\invol \big( - e_1 \ap{s_0^{-1}} - e_{\idd} \az{s_0^{-1}} + e_{\idd^2} \am{s_0^{-1}} - \am{1} \big)
	\\&=
	\am{s_0} e_1 + \az{s_0} e_{\idd^{-1}} - \ap{s_0} e_{\idd^{-2}} - \am{1}
	\\&=
	e_{\idd^2} \am{s_0}  + e_{\idd} \az{s_0}  - e_1 \ap{s_0} - \am{1}
	\\&=
	e_{\idd^2} \am{s_0^{-1}}  - e_{\idd} \az{s_0^{-1}}  - e_1 \ap{s_0^{-1}} - \am{1}.
\end{align*}
In the last line we have used the formulas for the left action of $\tau_{c_{-1}}$ on $E^2$ (see \eqref{eqOmegaOnE1Left} and recall the definition of $c_{-1}$ from \eqref{eqcminus1}) and the fact $e_\lambda \tau_{c_{-1}} = \lambda(c_{-1}) e_\lambda$ for all $\lambda \in \widehat{\quoz{T^0}{T^1}}$ (see \eqref{eqELambdaTauT}).

We have shown that $\tau_{s_0} \cdot \ap{s_0^{-1}} = \invol (\tau_{s_0} \cdot \ap{s_0^{-1}})$, and so going back to \eqref{eqRepInvolInvolRep}, we obtain
\begin{align*}
	(\rep_2 \circ \invol)(\am{1}) - (\invol \circ \rep_2)(\am{1})
	&=
	-
	\tau_{s_0} \cdot \rep_2 ( \ap{s_0^{-1}} )
	-
	\rep_2 ( \am{s_0} ) \cdot \tau_{s_0^{-1}}
	\\&=
	- \tau_{s_0} \cdot \big( \bm{1} \otimes \bz{s_0^{-1}} \big)
	-
	\big( - \bz{s_0} \otimes \bm{1} \big) \cdot \tau_{s_0^{-1}}
\end{align*}
This is the formula we had to check.
\end{proof}

\begin{rem}
It can be proved that the elements of $E^1 \otimes_{E^0} E^1$ in the previous lemma are nonzero, thus showing that $\invol \circ \rep_2 \neq \rep_2 \circ \invol$. For the details see \cite[Remark 4.5.8]{thesis}. Furthermore, in loc.cit. it is shown that one can modify the definition of $\rep_2$ on \( \bigoplus_{\substack{w \in \widetilde{W}\\\text{s.t. $\ell(w)=0$}}} H^2(I, \xx(w)) \) in such a way that one gets a new section $\rep_2'$ of $\mult_2$ that is a homomorphism of $k[\quoz{T^0}{T^1}]$-bimodules, that commutes with $\Gamma_\varpi$ and that commutes with $\invol$. However, we do not change the definition of $\rep_2$ to reduce the amount of computations.
\end{rem}

\begin{lmm}\label{lmmSomeElementsInTheKer}
The sub-$E^0$-bimodule of $E^1 \otimes_{E^0} E^1$ generated by the elements of the form
\[
	x - \rep_2(\mult_2(x)) \in \ker(\mult_2),
\]
where $x$ runs through the set of generators of $E^1 \otimes_{E^0} E^1$ as an $E^0$-bimodule computed in Lemma \ref{lmmGeneratorsT2}, is the sub-$E^0$-bimodule of $E^1 \otimes_{E^0} E^1$ generated by the following elements.
\begin{align*}
&\begin{aligned}
	&\bm{1} \otimes \bm{1}, &&\qquad\qquad\qquad\qquad\bp{1} \otimes \bm{1}, &&\qquad\qquad\qquad\qquad\bz{s_1} \otimes \bm{1},
	\\
	&\bm{1} \otimes \bp{1}, &&\qquad\qquad\qquad\qquad\bp{1} \otimes \bp{1}, &&\qquad\qquad\qquad\qquad\bz{s_0} \otimes \bp{1},
	\\
	&\bp{1} \otimes \bz{s_0}, &&\qquad\qquad\qquad\qquad\bz{s_1} \otimes \bz{s_0},
	\\
	&\bm{1} \otimes \bz{s_1}, &&\qquad\qquad\qquad\qquad\bz{s_0} \otimes \bz{s_1},
\end{aligned}
\\
	&\bz{s_0} \otimes \bz{s_0} + e_{\idd^{-1}} \cdot \bm{1} \otimes \bz{s_0} + e_{\idd} \cdot \bz{s_0} \otimes \bm{1} - e_1 \cdot \bm{1} \otimes \bp{s_0}, 
	\\
	&\bz{s_1} \otimes \bz{s_1} - e_{\idd} \cdot \bp{1} \otimes \bz{s_1} - e_{\idd^{-1}} \cdot \bz{s_1} \otimes \bp{1} - e_1 \cdot \bp{1} \otimes \bm{s_1}.
\end{align*}
In particular, all the above elements lie in $\ker(\mult_2)$.
\end{lmm}

\begin{proof}
Let us look at the generators of $E^1 \otimes_{E_0} E_1$ as an $E^0$-bimodule computed in Lemma \ref{lmmGeneratorsT2}.
From the definition of $\rep_2$ (Lemma \ref{lmmDefRep2}), we immediately see that some of these lie in the image of $\rep_2$, and so we can discard them immediately, because if $x$ is one of these elements then $x - \rep_2(\mult_2(x)) = 0$. We are left with the following elements:
\begin{align*}
	&\bm{1} \otimes \bm{1}, &&\bp{1} \otimes \bm{1}, &&//////// &&\bz{s_1} \otimes \bm{1},
	\\
	&\bm{1} \otimes \bp{1}, &&\bp{1} \otimes \bp{1}, &&\bz{s_0} \otimes \bp{1}, &&////////
	\\
	&//////// &&\bp{1} \otimes \bz{s_0}, &&\bz{s_0} \otimes \bz{s_0}, &&\bz{s_1} \otimes \bz{s_0},
	\\
	&\bm{1} \otimes \bz{s_1}, &&//////// &&\bz{s_0} \otimes \bz{s_1}, &&\bz{s_1} \otimes \bz{s_1}.
\end{align*}
We consider these remaining elements.
\begin{itemize}
\item
	The elements $\bm{1} \otimes \bm{1}$, $\bp{1} \otimes \bm{1}$, $\bm{1} \otimes \bp{1}$ and $\bp{1} \otimes \bp{1}$ are all in the kernel of $\mult_2$: indeed products are cup products by \eqref{eqCupYoneda}, and then trivially $\bm{1} \cdot \bm{1}$ and $\bp{1} \cdot \bp{1}$ are both zero; moreover, the fact that $\bp{1} \cupprod \bm{1}$ and $\bm{1} \cupprod \bp{1}$ are both zero can be shown with a simple computation using Poincaré duality (see \cite[Example 4.6]{newpaper}).
	
	In particular,  if $x$ is one of the above four elements, then, trivially,
	\[ x - \rep_2(\mult_2(x)) = x. \]
\item
	Now let us consider the four elements $\bz{s_1} \otimes \bm{1}$, $\bz{s_0} \otimes \bp{1}$, $\bp{1} \otimes \bz{s_0}$ and $\bm{1} \otimes \bz{s_1}$. To compute the product, one can use the formula \eqref{eqCupYoneda} relating cup product and (opposite of the) Yoneda product. But then we see that all such products are zero, because
	\begin{align*}
		\tau_{s_0} \cdot \bp{1} &= 0,
		&\bp{1} \cdot \tau_{s_0} &= 0,
		\\
		\tau_{s_1} \cdot \bm{1} &= 0,
		&\bm{1} \cdot \tau_{s_1} &= 0
	\end{align*}
	(see formulas \eqref{eqFormulasTauQp} and \eqref{eqFormulaRighteasyQp}).
	
	So, if $x$ is one of the above four elements, one has $x - \rep_2(\mult_2(x)) = x$.
\item
	Now let us consider the two elements $\bz{s_1} \otimes \bz{s_0}$ and $\bz{s_1} \otimes \bz{s_0}$. We compute the first product using the formula \eqref{eqCupYoneda} relating cup product and (opposite of the) Yoneda product, and the formulas \eqref{eqFormulasTauQp} and \eqref{eqFormulaRighteasyQp} for the action of $E^0$ on $E^1$ (the other product can be computed exactly in the same way, or alternatively one can use $\Gamma_\varpi$ or $\invol$):
	\begin{align*}
		\bz{s_1} \cdot \bz{s_0}
		&=
		(\bz{s_1} \cdot \tau_{s_0}) \cupprod (\tau_{s_1} \cdot \bz{s_0})
		\\&=
		\bz{s_1s_0} \cupprod (-\bz{s_1s_0})
		\\&=
		0.
	\end{align*}
	
	So, if $x$ is one of the above two elements, one has $x - \rep_2(\mult_2(x)) = x$.
\item
	It remains to consider the two elements $\bz{s_0} \otimes \bz{s_0}$ and $\bz{s_1} \otimes \bz{s_1}$; we start with the first one and then we use $\Gamma_\varpi$ for the second one. In the end of the proof of \cite[Proposition 9.5]{newpaper} the following formula is computed:
	\[
		\mathbf{x}_0 \cdot \mathbf{x}_0 = e_1 \cdot \upalpha_{s_0}^\star,
	\]
	where $\mathbf{x}_0$ is defined as
	\[
		\mathbf{x}_0 \defeq - \bz{s_0} - e_{\idd^{-1}} \bm{1}
	\]
	(see \cite[Lemma 4.17 and the subsequent lines]{newpaper}) and $\upalpha_{s_0}^\star$ is a certain element of $E^2$ (defined in \cite[Equation (123)]{newpaper}) with the property that
	\[
		\azstar{s_0} + \az{s_0} \in e_{\{\idd, \idd^{-1}\}} E^2
	\]
	(see \cite[Equation (124)]{newpaper}).
	Therefore, we deduce that
	\begin{align*}
		- e_1 \cdot \az{s_0}
		&=
		\big(-\bz{s_0} - e_{\idd^{-1}} \cdot \bm{1}\big) \cdot \big(-\bz{s_0} - e_{\idd^{-1}} \cdot \bm{1}\big)
		\\&=
		\bz{s_0} \cdot \bz{s_0} + e_{\idd^{-1}} \cdot \bm{1} \cdot \bz{s_0} + \bz{s_0} \cdot e_{\idd^{-1}} \cdot \bm{1}
		\\&\qquad+ e_{\idd^{-1}} \cdot \bm{1} \cdot e_{\idd^{-1}} \cdot \bm{1}
		\\&=
		\bz{s_0} \cdot \bz{s_0} + e_{\idd^{-1}} \cdot \bm{1} \cdot \bz{s_0} + e_{\idd} \cdot \bz{s_0} \cdot \bm{1}
		\\&\qquad+ e_{\idd^{-1}} \cdot e_{\idd^{-3}}  \cdot \bm{1} \cdot \bm{1}
		\\&=
		\bz{s_0} \cdot \bz{s_0} + e_{\idd^{-1}} \cdot \ap{s_0} - e_{\idd} \cdot \am{s_0}
	\end{align*}
	(here we have used the behaviour of the left and right action of the idempotents \eqref{eqIdempotentsE1}; moreover, in the last step one can compute products explicitly, but actually we have already computed them -- see the definition of $\rep_2$ in Lemma \ref{lmmDefRep2}). Now we can compute $\rep_2(\bz{s_0} \cdot \bz{s_0})$:
	\begin{align*}
		\rep_2(\bz{s_0} \cdot \bz{s_0})
		&=
		\rep_2 \left( -e_{\idd^{-1}} \cdot \ap{s_0} + e_{\idd} \cdot \am{s_0} - e_1 \cdot \az{s_0} \right)
		\\&=
		-e_{\idd^{-1}} \cdot \bm{1} \otimes \bz{s_0} - e_{\idd} \cdot \bz{s_0} \otimes \bm{1} + e_1 \cdot \bm{1} \otimes \bp{s_0}
	\end{align*}
	(here we have used that $\rep_2$ is a homomorphism of left $E^0$-modules, by Lemma \ref{lmmBimodulesGpAlgebraFiniteTorus}). Hence, the value of $x - \rep_2(\mult_2(x))$ for $x = \bz{s_0} \otimes \bz{s_0}$ is
	\[
		\bz{s_0} \otimes \bz{s_0}
		+ e_{\idd^{-1}} \cdot \bm{1} \otimes \bz{s_0} + e_{\idd} \cdot \bz{s_0} \otimes \bm{1} - e_1 \cdot \bm{1} \otimes \bp{s_0}.
	\]
	Now it remains to compute the value of $x - \rep_2(\mult_2(x))$ for $x = \bz{s_1} \otimes \bz{s_1}$, but since $\Gamma_\varpi(\bz{s_0}) = -\bz{s_1}$ and since $\rep_2$ is $\Gamma_\varpi$-invariant (Lemma \ref{lmmCommutesCCCN}), we see that such value can be obtained by applying $\Gamma_\varpi$ to the last displayed equality. Hence we get
	\[
		\bz{s_1} \otimes \bz{s_1}
		- e_{\idd} \cdot \bp{1} \otimes \bz{s_1} - e_{\idd^{-1}} \cdot \bz{s_1} \otimes \bp{1} - e_1 \cdot \bp{1} \otimes \bm{s_1}
	\]
	(using the formulas for the action of $\Gamma_\varpi$ on $E^0$ \eqref{eqCCCNOnH} and on $E^1$ \eqref{eqCCCNOnE1}).
	\qedhere
\end{itemize}
\end{proof}

\begin{rem}\label{remDefCandidateKernel2}
Let us define $K_2$ as the sub-$E^0$-bimodule of $E^1 \otimes_{E^0} E^1$ generated by the following elements:
\begin{align*}
&\begin{aligned}
	&\bm{1} \otimes \bm{1}, &&\qquad\qquad\qquad\qquad\bp{1} \otimes \bm{1}, &&\qquad\qquad\qquad\qquad\bz{s_1} \otimes \bm{1},
	\\
	&\bm{1} \otimes \bp{1}, &&\qquad\qquad\qquad\qquad\bp{1} \otimes \bp{1}, &&\qquad\qquad\qquad\qquad\bz{s_0} \otimes \bp{1},
	\\
	&\bp{1} \otimes \bz{s_0}, &&\qquad\qquad\qquad\qquad\bz{s_1} \otimes \bz{s_0},
	\\
	&\bm{1} \otimes \bz{s_1}, &&\qquad\qquad\qquad\qquad\bz{s_0} \otimes \bz{s_1},
\end{aligned}
\\
	&\bz{s_0} \otimes \bz{s_0} + e_{\idd^{-1}} \cdot \bm{1} \otimes \bz{s_0} + e_{\idd} \cdot \bz{s_0} \otimes \bm{1} - e_1 \cdot \bm{1} \otimes \bp{s_0}, 
\\
	&\bz{s_1} \otimes \bz{s_1} - e_{\idd} \cdot \bp{1} \otimes \bz{s_1} - e_{\idd^{-1}} \cdot \bz{s_1} \otimes \bp{1} - e_1 \cdot \bp{1} \otimes \bm{s_1},
\\
	&\bp{s_0} \otimes \bz{s_0} + \bz{s_0} \otimes \bm{s_0} = - \tau_{s_0} \cdot \bm{1} \otimes \bz{s_0} + \bz{s_0} \otimes \bm{1} \cdot \tau_{s_0},
\\
	&\bm{s_1} \otimes \bz{s_1} + \bz{s_1} \otimes \bp{s_1} = - \tau_{s_1} \cdot \bp{1} \otimes \bz{s_1} + \bz{s_1} \otimes \bp{1} \cdot \tau_{s_1}.
\end{align*}
The first twelve elements were obtained in Lemma \ref{lmmSomeElementsInTheKer} (in particular, they lie in $\ker(\mult_2)$). The last two elements are slight modifications of those obtained in Lemma \ref{lmmInvolBadBehaviour}: more precisely, recalling the definition of $c_{-1}$ from \eqref{eqcminus1}, they are respectively equal to
		\begin{align*}
			(\rep_2 \circ \invol) (\am{1}) \cdot \tau_{c_{-1}} &- (\invol \circ \rep_2) (\am{1}) \cdot \tau_{c_{-1}},
			\\
			- (\rep_2 \circ \invol) (\ap{1}) \cdot \tau_{c_{-1}} &+ (\invol \circ \rep_2) (\ap{1}) \cdot \tau_{c_{-1}}
		\end{align*}
(in particular, they lie in $\ker(\mult_2)$, too, since $\invol$ and $\mult_2$ commute by \eqref{eqMultiplicationAndCCCNAndInvol}).
\end{rem}

We want to prove that $K_2$ is actually the full $\ker(\rep_2)$. Let us start with some lemmas.

\begin{lmm}\label{lmmK2Invariant}
One has that $K_2$ is $\Gamma_\varpi$-invariant and $\invol$-invariant.
\end{lmm}

\begin{proof}
It suffices to prove that the generators listed in the definition of $K_2$ (Remark \ref{remDefCandidateKernel2}) are $\Gamma_\varpi$-invariant and $\invol$-invariant. From the formulas for the action of $\Gamma_\varpi$ \eqref{eqCCCNOnE1} and of $\invol$ \eqref{eqInvolOnE1} on $E^1$, it is immediate to see that applying $\Gamma_\varpi$ or $\invol$ to each of the first ten generators we get, up to a sign, again one of such generators. The same is true for the last two generators. It remains to study the behaviour of $\Gamma_\varpi$ and $\invol$ on the following two elements:
\begin{align*}
	x_0 &\defeq \bz{s_0} \otimes \bz{s_0} + e_{\idd^{-1}} \cdot \bm{1} \otimes \bz{s_0} + e_{\idd} \cdot \bz{s_0} \otimes \bm{1} - e_1 \cdot \bm{1} \otimes \bp{s_0}, 
	\\
	x_1 &\defeq \bz{s_1} \otimes \bz{s_1} - e_{\idd} \cdot \bp{1} \otimes \bz{s_1} - e_{\idd^{-1}} \cdot \bz{s_1} \otimes \bp{1} - e_1 \cdot \bp{1} \otimes \bm{s_1}.
\end{align*}
In the proof of Lemma \ref{lmmSomeElementsInTheKer}, the element $x_1$ was obtained by applying $\Gamma_\varpi$ to $x_0$: in other words $\Gamma_\varpi(x_0) = x_1$ and hence $\Gamma_\varpi(x_1) = x_0$. Now let us compute $\invol(x_0)$, using the formula \eqref{eqInvolOnE1} for the action of $\invol$ on $E^1$ and the formula \eqref{eqIdempotentsE1} for the action of the idempotents on $E^1$:
\begin{align*}
	\invol(x_0)
	&=
	\invol \left(\bz{s_0} \otimes \bz{s_0} + e_{\idd^{-1}} \cdot \bm{1} \otimes \bz{s_0} + e_{\idd} \cdot \bz{s_0} \otimes \bm{1} - e_1 \cdot \bm{1} \otimes \bp{s_0}\right)
	\\&=
	- \bz{s_0^{-1}} \otimes \bz{s_0^{-1}} + \bz{s_0^{-1}} \otimes \bm{1} \cdot e_{\idd} + \bm{1} \otimes \bz{s_0^{-1}} \cdot e_{\idd^{-1}} - \bm{s_0^{-1}} \otimes \bm{1} \cdot e_1
	\\&=
	- \bz{s_0} \otimes \bz{s_0} + \bz{s_0^{-1}} \otimes e_{\idd^{-1}} \bm{1} + \bm{1} \otimes e_{\idd} \bz{s_0^{-1}} - \bm{s_0^{-1}} \otimes e_{\idd^{-2}} \bm{1}
	\\&=
	- \bz{s_0} \otimes \bz{s_0} + e_{\idd} \cdot \bz{s_0^{-1}} \otimes \bm{1} + e_{\idd^{-1}} \cdot \bm{1} \otimes \bz{s_0^{-1}} - e_{1} \cdot \bm{s_0^{-1}} \otimes \bm{1}
	\\&=
	- \bz{s_0} \otimes \bz{s_0} - e_{\idd} \cdot \bz{s_0} \otimes \bm{1} - e_{\idd^{-1}} \cdot \bm{1} \otimes \bz{s_0} + e_{1} \cdot \bm{1} \otimes \bp{s_0}
	\\&=
	-x_0.
\end{align*}
Now it remains to treat $x_1$: we have already recalled that $x_1 = \Gamma_\varpi(x_0)$. We can use the fact that $\Gamma_\varpi$ and $\invol$ commute (on $E^\ast$, as recalled in \eqref{eqCCCNInvolCommute}, and hence also on $T_{E^0}^\ast E^1$), getting that
\[
	\invol(x_1)
	=
	\invol(\Gamma_\varpi( x_0))
	=
	\Gamma_\varpi(\invol(x_0))
	=
	\Gamma_\varpi(-x_0)
	=
	-x_1.
	\qedhere
\]
\end{proof}

\begin{lmm}\label{lmmInvolRep2AfterQuot}
The two composite maps
\[\begin{tikzcd}[row sep = 0em]
	E^2 \ar[r, "{\rep_2}"] & E^1 \otimes_{E^0} E^1 \ar[r, "{\invol}"] & E^1 \otimes_{E^0} E^1 \ar[r, "{\text{quot.}}"] & \quoz{(E^1 \otimes_{E^0} E^1)}{K_2},
	\\
	E^2 \ar[r, "{\invol}"] & E^2 \ar[r, "{\rep_2}"] & E^1 \otimes_{E^0} E^1 \ar[r, "{\text{quot.}}"] & \quoz{(E^1 \otimes_{E^0} E^1)}{K_2}
\end{tikzcd}\]
coincide.
\end{lmm}

\begin{proof}
By Lemma \ref{lmmInvolRep2} the two maps coincide on $F^1E^2$, and by Remark \ref{remDefCandidateKernel2} they also coincide at $\am{1}$ and at $\ap{1}$. For $\omega \in \quoz{T^0}{T^1}$, both maps transform left multiplication by $\tau_\omega$ into right multiplication by $\tau_{\omega^{-1}}$ by Lemma \ref{lmmBimodulesGpAlgebraFiniteTorus}, and so we conclude that the two maps coincide on the whole $E^2$.
\end{proof}

\begin{lmm}\label{lmmLeftModules}
One has that the composite map
\[
		\begin{tikzcd}[column sep = 4em]
			E^2
			\ar[r, "{\rep_2}"]
			&
			E^1 \otimes_{E^0} E^1
			\ar[r, "\text{quot. map}"]
			&
			\quoz{(E^1 \otimes_{E^0} E^1)}{K_2}
		\end{tikzcd}
\]
is a homomorphism of $E^0$-bimodules.
\end{lmm}

\begin{proof}
We will show that the following equalities are true:
\begin{itemize}
\item[(i)]
	$\rep_2(\tau_{s_0} \cdot \alpha_{s_1v}^\sigma) \equiv \tau_{s_0} \cdot \rep_2(\alpha_{s_1v}^\sigma)$ modulo $K_2$ for $v \in \widetilde{W}$ with $\ell(s_1v) = \ell(v) + 1$ and for $\sigma \in \{-,0,+\}$ (we will see that in this case we actually have a true equality in $E^1 \otimes_{E^0} E^1$, with no need to consider the quotient modulo $K_2$);
\item[(ii)]
	$\rep_2(\tau_{s_1} \cdot \alpha_{s_1v}^\sigma) \equiv \tau_{s_1} \cdot \rep_2(\alpha_{s_1v}^\sigma)$ modulo $K_2$ for $v \in \widetilde{W}$ with $\ell(s_1v) = \ell(v) + 1$ and for $\sigma \in \{-,0,+\}$;
\item[(iii)]
	$\rep_2(\tau_{s_i} \cdot \ap{1}) \equiv \tau_{s_i} \cdot \rep_2(\ap{1})$ modulo $K_2$ for $i \in \{0,1\}$.
\end{itemize}
Before checking these three properties, let us show that, if they hold, then the lemma is proved. We saw in Lemma \ref{lmmCommutesCCCN} that $\Gamma_\varpi$ commutes with $\rep_2$ and in Lemma \ref{lmmK2Invariant} that $K_2$ is $\Gamma_\varpi$-invariant, and so by applying $\Gamma_\varpi$ to the equalities in (i), (ii) and (iii) (and using \eqref{eqCCCNOnH} and \eqref{eqCCCNOnE2}), we get:
\begin{itemize}
\item[(i\textquotesingle)]
	$\rep_2(\tau_{s_1} \cdot \alpha_{s_0w}^\sigma) \equiv \tau_{s_1} \cdot \rep_2(\alpha_{s_0w}^\sigma) \mod K_2$ for $w \in \widetilde{W}$ with $\ell(s_0w) = \ell(w) + 1$ and for $\sigma \in \{-,0,+\}$ (as before, we already have equality in $E^1 \otimes_{E^0} E^1$);
\item[(ii\textquotesingle)]
	$\rep_2(\tau_{s_0} \cdot \alpha_{s_0w}^\sigma) \equiv \tau_{s_0} \cdot \rep_2(\alpha_{s_0w}^\sigma) \mod K_2$ for $w \in \widetilde{W}$ with $\ell(s_0w) = \ell(w) + 1$ and for $\sigma \in \{-,0,+\}$;
\item[(iii\textquotesingle)]
	$\rep_2(\tau_{s_i} \cdot \am{1}) \equiv \tau_{s_i} \cdot \rep_2(\am{1}) \mod K_2$ for $i \in \{0,1\}$.
\end{itemize}
Note that, for $\omega \in \quoz{T^0}{T^1}$, in (iii) and (iii\textquotesingle) we can consider the analogous congruences with $\ap{\omega}$ (respectively, of $\am{\omega}$) in place of $\ap{1}$ (respectively, of $\am{1}$): the new congruences are still true because $\rep_2$ is a homomorphism of left $k[\quoz{T^0}{T^1}]$-modules (by Lemma \ref{lmmBimodulesGpAlgebraFiniteTorus}, or just by the definition of $\rep_2$ in Lemma \ref{lmmDefRep2}).

All in all, this shows that the congruence
\begin{equation}\label{eqCongruence}
	\rep_2(x \cdot \alpha) \equiv x \cdot \rep_2(\alpha) \mod K_2
\end{equation}
is true for $\alpha$ running through a $k$-basis of $E^2$ and for $x \in \{\tau_{s_0},\tau_{s_1}\}$. But since $\rep_2$ is a homomorphism of left $k[\quoz{T^0}{T^1}]$-modules, the same congruence is also true for $x = \tau_{\omega}$ for $\omega \in \quoz{T^0}{T^1}$. In other words \eqref{eqCongruence} is true for $\alpha$ running through a $k$-basis of $E^2$ and for $x$ running through a set of generators of $E^0$ as a $k$-algebra, and hence it follows that it must be true for all $\alpha \in E^2$ and all $x \in E^0$, completing the proof that $\rep_2$ becomes a homomorphism of left $E^0$-modules after modding out $K_2$.

To prove the same claim for the right $E^0$-action, since $K_2$ is $\invol$-invariant by Lemma \ref{lmmK2Invariant}, applying $\invol$ to \eqref{eqCongruence} we obtain that
\[
	(\invol \circ \rep_2)(x \cdot \alpha) \equiv (\invol \circ \rep_2)(\alpha) \cdot \invol(x) \mod K_2
\]
for all $x \in E^0$ and all $\alpha \in E^2$. Since the maps $\invol \circ \rep_2$ and $\rep_2 \circ \invol$ coincide after modding out by $K_2$ (Lemma \ref{lmmInvolRep2AfterQuot}), we deduce that
\[
	\rep_2(\invol(\alpha) \cdot \invol(x)) \equiv \rep_2(\invol(\alpha)) \cdot \invol(x) \mod K_2,
\]
completing the proof that $\rep_2$ becomes a homomorphism of right $E^0$-modules after modding out $K_2$.

Now, it remains to prove (i), (ii) and (iii).
\begin{itemize}
\item[(i)]
	One has
	\begin{align*}
		\tau_{s_0} \cdot \rep_2(\am{s_1v})
		&=
		- \tau_{s_0} \cdot \bp{1} \otimes \bz{s_1v}
		\\&=
		0
		&&\text{by \eqref{eqFormulasTauQp}}
		\\&=
		\rep_2(\tau_{s_0} \cdot \am{s_1v})
		&&\text{by \eqref{eqExplicitDeg2AddUp}.}
	\end{align*}
	Moreover,
	\begin{align*}
		\tau_{s_0} \cdot \rep_2(\az{s_1v})
		&=
		\tau_{s_0} \cdot \bp{1} \otimes \bm{s_1v}
		\\&=
		0
		&&\text{by \eqref{eqFormulasTauQp}}
		\\&=
		\rep_2(\tau_{s_0} \cdot \az{s_1v})
		&&\text{by \eqref{eqExplicitDeg2AddUp}.}
	\end{align*}
	And finally,
	\begin{align*}
		\tau_{s_0} \cdot \rep_2(\ap{s_1v})
		&=
		\tau_{s_0} \cdot \bz{s_1} \otimes \bp{v}
		\\&=
		- \bz{s_0s_1} \otimes \bp{v}
		&&\text{by \eqref{eqFormulasTauQp}}
		\\&=
		\bz{s_0} \otimes \bm{s_1v}
		&&\text{by \eqref{eqFormulaRighteasyQp} and \eqref{eqFormulasTauQp}}
		\\&=
		- \rep_2(\am{s_0s_1v})
		\\&=
		\rep_2(\tau_{s_0} \cdot \ap{s_1v})
		&&\text{by \eqref{eqExplicitDeg2AddUp}.}
	\end{align*}
\item[(ii-a)]
	Let us prove that $\rep_2 \big( \tau_{s_1} \cdot \am{s_1v} \big) \equiv \tau_{s_1} \cdot \rep_2 ( \am{s_1v} )$ modulo $K_2$:
	we first assume that $\ell(v) \geqslant 1$, and compute
	\begin{align*}
		\rep_2(\tau_{s_1} \cdot \am{s_1v})
		&=
		- e_1 \rep_2(\am{s_1 v}) - \rep_2(\ap{c_{-1} v})
		\\&\qquad\qquad\qquad\qquad\text{by \eqref{eqTauS1OnE2LeftNotAddUp} and Lemma \ref{lmmBimodulesGpAlgebraFiniteTorus}}
		\\&=
		e_1 \bp{1} \otimes \bz{s_1 v} - \bm{1} \otimes \bz{c_{-1} v},
	\end{align*}
	and
	\begin{align*}
		\tau_{s_1} \cdot \rep_2(\am{s_1v})
		&=
		- \tau_{s_1} \cdot \bp{1} \otimes \bz{s_1v}
		\\&=
		\tau_{s_1} \cdot \bp{s_1} \otimes \bz{v}
		\\&\qquad\qquad\qquad\qquad\text{by \eqref{eqFormulasTauQp} and \eqref{eqFormulaRighteasyQp}}
		\\&=
		-e_1 \bp{s_1}  \otimes \bz{v}
		+ 2e_{\idd^{-1}} \bz{s_1}  \otimes \bz{v}
		\\&\qquad+ e_{\idd^{-2}} \bm{s_1} \otimes \bz{v}
		- \bm{c_{-1}} \otimes \bz{v} 
		\\&\qquad\qquad\qquad\qquad\text{by \eqref{eqTauS1OnE1LeftNotAddUp}}
		\\&=
		e_1 \bp{1}  \otimes \bz{s_1v}
		+ 2e_{\idd^{-1}} \bz{s_1}  \otimes \bz{s_0} \cdot \tau_{s_0^{-1} v}
		\\&\qquad- e_{\idd^{-2}} \tau_{s_1} \cdot \bp{1} \otimes \bz{s_0} \cdot \tau_{s_0^{-1} v}
		- \bm{1} \otimes \bz{c_{-1} v} 
		\\&\qquad\qquad\qquad\qquad\text{by \eqref{eqFormulasTauQp}, \eqref{eqFormulaRighteasyQp}, \eqref{eqOmegaOnE1Left} and \eqref{eqOmegaOnE1Right}.}
	\end{align*}
	Since both $\bz{s_1}  \otimes \bz{s_0}$ and $\bp{1} \otimes \bz{s_0}$ lie in the list of generators defining $K_2$, we conclude that $\rep_2 \big( \tau_{s_1} \cdot \am{s_1v} \big) \equiv \tau_{s_1} \cdot \rep_2 ( \am{s_1v} )$ modulo $K_2$.
	
	We now consider the case $\ell(v) = 0$: by definition of $\rep_2$ we have
	\[
		\rep_2(\ap{c_{-1}v})
		=
		\rep_2 \big( \ap{c_{-1}v} + \tau_{s_1} \cdot \am{s_1 v} \big) - \tau_{s_1} \cdot \rep_2 \big( \am{s_1 v} \big),
	\]
	and then the claimed congruence is actually an equality in $E^1 \otimes_{E^0} E^1$ by $k$-linearity of $\rep_2$.
\item[(ii-b)]
	Let us prove that $\rep_2 \big( \tau_{s_1} \cdot \az{s_1v} \big) \equiv \tau_{s_1} \cdot \rep_2 ( \az{s_1v} )$ modulo $K_2$: we compute
	\begin{align*}
		\rep_2(\tau_{s_1} \cdot \az{s_1v})
		&=
		\begin{cases}
			\begin{matrix*}[l]
				- e_1 \rep_2(\az{s_1v}) -2 e_{\idd^{-1}} \rep_2(\ap{s_1v})
				\\\qquad
				- \rep_2(\az{c_{-1}v})
			\end{matrix*}
			&\text{if $\ell(v) \geqslant 1$}
			\\
			- e_1 \rep_2(\az{s_1v}) -2 e_{\idd^{-1}} \rep_2(\ap{s_1v})
			&\text{if $\ell(v) = 0$}
		\end{cases}
		\\&\qquad\qquad\qquad\qquad\text{by \eqref{eqTauS1OnE2LeftNotAddUp} and Lemma \ref{lmmBimodulesGpAlgebraFiniteTorus}}
		\\&=
		\begin{cases}
			\begin{matrix*}[l]
				- e_1 \bp{1} \otimes \bm{s_1v} -2 e_{\idd^{-1}} \bz{s_1} \otimes \bp{v}
				\\\qquad
				+ \bm{1} \otimes \bp{c_{-1}v}	
			\end{matrix*}
			&\text{if $\ell(v) \geqslant 1$}
			\\
			- e_1 \bp{1} \otimes \bm{s_1v} -2 e_{\idd^{-1}} \bz{s_1} \otimes \bp{v}
			&\text{if $\ell(v) = 0$,}
		\end{cases}
	\end{align*}
	and
	\begin{align*}
		\tau_{s_1} \cdot \rep_2(\az{s_1v})
		&=
		\tau_{s_1} \cdot \bp{1} \otimes \bm{s_1v}
		\\&=
		- \tau_{s_1} \cdot \bp{s_1} \otimes \bp{v}
		\\&\qquad\qquad\qquad\qquad\qquad\text{by \eqref{eqFormulasTauQp} and \eqref{eqFormulaRighteasyQp}}
		\\&=
		e_1 \bp{s_1} \otimes \bp{v}
		- 2e_{\idd^{-1}} \bz{s_1} \otimes \bp{v}
		\\&\qquad - e_{\idd^{-2}} \bm{s_1} \otimes \bp{v}
		+ \bm{c_{-1}} \otimes \bp{v}
		\\&\qquad\qquad\qquad\qquad\qquad\text{by \eqref{eqTauS1OnE1LeftNotAddUp}}
		\\&=
		- e_1 \bp{1} \otimes \bm{s_1v}
		- 2e_{\idd^{-1}} \bz{s_1} \otimes \bp{v}
		\\&\qquad - e_{\idd^{-2}} \bm{s_1} \otimes \bp{v}
		+ \bm{1} \otimes \bp{c_{-1}v}
		\\&\qquad\qquad\qquad\qquad\qquad\text{by \eqref{eqFormulasTauQp}, \eqref{eqFormulaRighteasyQp}, \eqref{eqOmegaOnE1Left} and \eqref{eqOmegaOnE1Right}.}
	\end{align*}
	Now, if $\ell(v) \geqslant 1$, then using \eqref{eqFormulasTauQp} and \eqref{eqFormulaRighteasyQp} we see that
	\[
		\bm{s_1} \otimes \bp{v} = - (\bm{s_1} \tau_{s_0}) \otimes \bm{s_0^{-1} v} = 0,
	\]
	and hence we conclude that $\rep_2(\tau_{s_1} \cdot \az{s_1v}) = \tau_{s_1} \cdot \rep_2(\az{s_1v})$.
	
	If instead $\ell(v) = 0$, then using \eqref{eqFormulasTauQp} and \eqref{eqOmegaOnE1Right} we see that
	\begin{align*}
		\bm{s_1} \otimes \bp{v}
		&=
		- \tau_{s_1} \cdot \bp{1} \otimes \bp{1} \cdot \tau_v
		\in
		K_2,
		\\
		\bm{1} \otimes \bp{c_{-1}v}
		&=
		\bm{1} \otimes \bp{1} \cdot \tau_{c_{-1}v}
		\in
		K_2,
	\end{align*}
	and hence we conclude that $\rep_2(\tau_{s_1} \cdot \az{s_1v}) \equiv \tau_{s_1} \cdot \rep_2(\az{s_1v})$ modulo $K_2$.
\item[(ii-c)]
	Let us prove that $\rep_2 \big( \tau_{s_1} \cdot \ap{s_1v} \big) \equiv \tau_{s_1} \cdot \rep_2 ( \ap{s_1v} )$ modulo $K_2$:
	we compute
		\begin{align*}
			\rep_2 \big( \tau_{s_1} \cdot \ap{s_1v} \big)
			&=
			- \rep_2 \big( e_1 \ap{s_1v} \big)
			&&\text{by \eqref{eqTauS1OnE2LeftNotAddUp}}
			\\&=
			- e_1 \rep_2 \big( \ap{s_1v} \big)
			&&\text{by Lemma \ref{lmmBimodulesGpAlgebraFiniteTorus}}
			\\&=
			- e_1 \bz{s_1} \otimes \bp{v},
		\end{align*}
		and
		\begin{align*}
			\tau_{s_1} \cdot \rep_2 ( \ap{s_1v} )
			&=
			\tau_{s_1} \cdot \bz{s_1} \otimes \bp{v}
			\\&=
			-e_1 \bz{s_1} \otimes \bp{v} - e_{\idd^{-1}}\bm{s_1} \otimes \bp{v}
			&&\text{by \eqref{eqTauS1OnE1LeftNotAddUp}.}
		\end{align*}
	If $\ell(v) \geqslant 1$, we can write $v$ as $s_0v'$ for some $v' \in \widetilde{W}$ such that lengths add up. Using formulas \eqref{eqFormulasTauQp} and \eqref{eqFormulaRighteasyQp}, we obtain
	\begin{align*}
		\bm{s_1} \otimes \bp{v}
		&=
		- (\bm{s_1} \cdot \tau_{s_0}) \otimes \bm{v}
		=
		0.
	\end{align*}
	If instead $\ell(v) = 0$, using formulas \eqref{eqFormulasTauQp} and \eqref{eqOmegaOnE1Right}, we obtain
	\begin{align*}
		- e_{\idd^{-1}}\bm{s_1} \otimes \bp{v}
		&=
		e_{\idd^{-1}} \tau_{s_1} \cdot \bp{1} \otimes \bp{1} \cdot \tau_v
		\in
		K_2.
	\end{align*}
	Therefore, both in the case $\ell(v) \geqslant 1$ and in the case $\ell(v) = 0$, we obtain that $\rep_2 \big( \tau_{s_1} \cdot \ap{s_1v} \big) \equiv \tau_{s_1} \cdot \rep_2 ( \ap{s_1v} )$ modulo $K_2$.
\item[(iii-a)]
	Let us prove that $\rep_2 \big( \tau_{s_1} \cdot \ap{1} \big) \equiv \tau_{s_1} \cdot \rep_2 ( \ap{1} )$ modulo $K_2$:
	we compute
	\begin{align*}
		\tau_{s_1} \cdot \rep_2(\ap{1})
		&=
		\tau_{s_1} \cdot \left( \rep_2 \big( \ap{1} + \tau_{s_1} \cdot \am{s_1^{-1}} \big) - \tau_{s_1} \cdot \rep_2 \big( \am{s_1^{-1}} \big) \right)
		\\&\equiv
		\rep_2 \big( \tau_{s_1} \cdot (\ap{1} + \tau_{s_1} \cdot \am{s_1^{-1}}) \big) - \tau_{s_1}^2 \cdot \rep_2 \big( \am{s_1^{-1}} \big)
		\mod K_2
		\\&\qquad\qquad
		\substack{\text{by (ii), since $\ap{1} + \tau_{s_1} \cdot \am{s_1^{-1}} \in \bigoplus_{\omega \in \quoz{T^0}{T^1}} H^2(I,\xx(s_1\omega))$}}
		\\&\equiv
		\rep_2 \big( \tau_{s_1} \cdot (\ap{1} + \tau_{s_1} \cdot \am{s_1^{-1}}) \big) - \rep_2 \big( \tau_{s_1}^2 \cdot \am{s_1^{-1}} \big)
		\mod K_2
		\\&\qquad\qquad
		\substack{\text{again by (ii)}}
		\\&=
		\rep_2(\tau_{s_1} \cdot \ap{1}).
	\end{align*}
\item[(iii-b)]
	Let us prove that $\rep_2 \big( \tau_{s_0} \cdot \ap{1} \big) \equiv \tau_{s_0} \cdot \rep_2 ( \ap{1} )$ modulo $K_2$. Since $K_2$ is $\invol$-invariant, this is equivalent to showing that
	\[
		(\invol \circ \rep_2) \big( \tau_{s_0} \cdot \ap{1} \big) \equiv (\invol \circ \rep_2) ( \ap{1} ) \cdot \tau_{s_0^{-1}}
		\mod K_2.
	\]
	Recall from Lemma \ref{lmmInvolRep2AfterQuot} that the maps $\invol \circ \rep_2$ and $\rep_2 \circ \invol$ coincide after post-composing with the quotient map onto $\quoz{(E^1 \otimes_{E^0} E^1)}{K_2}$. Therefore our claim is equivalent to showing that
	\[
		(\rep_2 \circ \invol) \big( \tau_{s_0} \cdot \ap{1} \big) \equiv (\rep_2 \circ \invol) ( \ap{1} ) \cdot \tau_{s_0^{-1}}
		\mod K_2.
	\]
	I.e., using the formula \eqref{eqInvolOnE2} for the action of $\invol$ on $E^2$, this is equivalent to showing that
	\begin{equation}\label{eqAp1Claim}
		\rep_2 \big( \ap{1} \cdot \tau_{s_0^{-1}} \big) \equiv \rep_2 ( \ap{1} ) \cdot \tau_{s_0^{-1}}
		\mod K_2.
	\end{equation}
	By definition of $\rep_2$, the right hand side of the above equation is
	\begin{equation}\label{eqAp1Reversed}
		\rep_2 ( \ap{1} ) \cdot \tau_{s_0^{-1}}
		=
		\rep_2 \big( \ap{1} + \tau_{s_1} \cdot \am{s_1^{-1}} \big)
		\cdot \tau_{s_0^{-1}}
		- \tau_{s_1} \cdot \rep_2 \big( \am{s_1^{-1}} \big)
		\cdot \tau_{s_0^{-1}}.
	\end{equation}
	From (i) we know that
	\[
		\tau_{s_0} \cdot \rep_2(\alpha) = \rep_2(\tau_{s_0} \cdot \alpha)
	\]
	for all $\alpha \in \oplus_{\omega \in \quoz{T^0}{T^1}} H^2(I,\xx(\omega s_1))$, and therefore for the same $\alpha$'s we have
	\[
		(\invol \circ \rep_2)(\alpha) \cdot \tau_{s_0^{-1}} = (\invol \circ \rep_2)(\tau_{s_0} \cdot \alpha).
	\]
	As explained above, we can replace $\invol \circ \rep_2$ with $\rep_2 \circ \invol$ if we content ourselves with a congruence modulo $K_2$. As the subspace $\oplus_{\omega \in \quoz{T^0}{T^1}} H^2(I,\xx(\omega s_1))$ is $\invol$-invariant, we then obtain 
	\[
		\rep_2(\alpha) \cdot \tau_{s_0^{-1}} \equiv \rep_2(\alpha \cdot \tau_{s_0^{-1}}) \mod K_2
	\]
	for all $\alpha \in \oplus_{\omega \in \quoz{T^0}{T^1}} H^2(I,\xx(\omega s_1))$. Since both $\ap{1} + \tau_{s_1} \cdot \am{s_1^{-1}}$ and $\am{s_1^{-1}}$ lie in such subspace, applying this latter fact to \eqref{eqAp1Reversed} we obtain
	\begin{equation}\label{eqAp1Reversed2}\begin{aligned}
		\rep_2 ( \ap{1} ) \cdot \tau_{s_0^{-1}}
		&=
		\rep_2 \big( \ap{1} + \tau_{s_1} \cdot \am{s_1^{-1}} \big)
		\cdot \tau_{s_0^{-1}}
		\\&\qquad - \tau_{s_1} \cdot \rep_2 \big( \am{s_1^{-1}} \big)
		\cdot \tau_{s_0^{-1}}
		\\&\equiv
		\rep_2 \big( \ap{1} \cdot \tau_{s_0^{-1}} + \tau_{s_1} \cdot \am{s_1^{-1}} \cdot \tau_{s_0^{-1}} \big)
		\\&\qquad - \tau_{s_1} \cdot \rep_2 \big( \am{s_1^{-1}} \cdot \tau_{s_0^{-1}} \big)
		\mod K_2.
	\end{aligned}\end{equation}
	Now, since $\am{s_1^{-1}} \cdot \tau_{s_0^{-1}}$ lies in $H^2(I,\xx(s_1s_0))$, by (ii) we have that
	\[
		\tau_{s_1} \cdot \rep_2 \big( \am{s_1^{-1}} \cdot \tau_{s_0^{-1}} \big)
		\equiv
		\rep_2 \big( \tau_{s_1} \cdot \am{s_1^{-1}} \cdot \tau_{s_0^{-1}} \big)
		\mod K_2.
	\]
	Therefore, continuing the computation in \eqref{eqAp1Reversed2}, we obtain
	\[
		\rep_2 ( \ap{1} ) \cdot \tau_{s_0^{-1}}
		\equiv
		\rep_2 \big( \ap{1} \cdot \tau_{s_0^{-1}} \big)
		\mod K_2,
	\]
	as we wanted to show in \eqref{eqAp1Claim}.
\qedhere
\end{itemize}
\end{proof}

\begin{prop}\label{propKerDeg2}
The kernel of the degree $2$ multiplication map
\[
	\funcInline{\mult_2}{T_{E_0}^2 E^1 = E^1 \otimes_{E^0} E^1}{E^2}
\]
is $K_2$ (defined in Remark \ref{remDefCandidateKernel2}).
\end{prop}

\begin{proof}
We consider the maps
\[\begin{tikzcd}[column sep = huge]
	\quoz{(E^1 \otimes_{E^0} E^1)}{K_2}
	\ar[r, shift left, "{\overline{\mult_2}}"]
	&
	E^2,
	\ar[l, shift left, "{\overline{\rep_2}}"]		
\end{tikzcd}\]
induced by $\mult_2$ and $\rep_2$ (note that $\overline{\mult_2}$ is well defined because in Remark \ref{remDefCandidateKernel2} we have observed that $K_2 \subseteq \ker(\mult_2)$).

To prove the proposition it suffices to show that $\overline{\rep_2} \circ \overline{\mult_2} = \id_{\quoz{(E^1 \otimes_{E^0} E^1)}{K_2}}$: this can be checked on a set of generators of $E^1 \otimes_{E^0} E^1$ as an $E^0$-bimodule, since Lemma \ref{lmmLeftModules} states that $\overline{\rep_2}$ is a homomorphism of $E^0$-bimodules: then the claim is clear from Lemma \ref{lmmSomeElementsInTheKer} and the definition of $K_2$.
\end{proof}

\subsection{The \texorpdfstring{$3$\textsuperscript{rd}}{3rd} graded piece of the kernel}

To study the kernel of the multiplication map $\funcInline{\mult_3}{T^3_{E^0} E^1}{E^3}$, we follow the same strategy as for the kernel of $\mult_2$. Since $E^3$ admits a somewhat simpler basis than $E^2$, some things will be easier.

\begin{lmm}\label{lmmDefRep3}
The multiplication map $\funcInline{\mult_3}{T^3_{E^0} E^1}{E^3}$ is surjective, and the following is a section of it as a homomorphism of $k$-vector spaces.
\[\begin{tikzcd}[row sep = -0.6em]
\rep_3 \colon &[-3.5em] E^3 \ar[r] &[-1em] T_{E^0}^3E^1 &[-1em]
\\
& \phi_{(s_1s_0)^i\omega} \ar[r, mapsto] & - \bp{1} \otimes \bz{(s_1s_0)^i} \otimes \bm{1} \cdot \tau_\omega & \scriptstyle\text{for $i \in \zpiu$, $\omega \in \quoz{T^0}{T^1}$,}
\\
& \phi_{(s_0s_1)^i\omega} \ar[r, mapsto] & \bm{1} \otimes \bz{(s_0s_1)^i} \otimes \bp{1} \cdot \tau_\omega  & \scriptstyle\text{for $i \in \zpiu$, $\omega \in \quoz{T^0}{T^1}$,}
\\
& \phi_{s_0(s_1s_0)^i\omega} \ar[r, mapsto] & - \bm{1} \otimes \bz{s_0(s_1s_0)^i} \otimes \bm{1} \cdot \tau_\omega  & \scriptstyle\text{for $i \in \nn$, $\omega \in \quoz{T^0}{T^1}$,}
\\
& \phi_{s_1(s_0s_1)^i\omega} \ar[r, mapsto] & \bp{1} \otimes \bz{s_1(s_0s_1)^i} \otimes \bp{1} \cdot \tau_\omega  & \scriptstyle\text{for $i \in \nn$, $\omega \in \quoz{T^0}{T^1}$,}
\\
& \phi_{\omega} \ar[r, mapsto] & (\tau_{s_0} + e_1) \cdot \rep_3(\phi_{s_0^{-1}}) \cdot \tau_\omega  & \scriptstyle\text{for $\omega \in \quoz{T^0}{T^1}$.}
\end{tikzcd}\]
It can also be rewritten as
\[\begin{tikzcd}[row sep = -0.6em]
\rep_3 \colon &[-3.5em] E^3 \ar[r] &[-1em] T_{E^0}^3E^1 &[-1em]
\\
& \phi_{s_1v} \ar[r, mapsto] & \bp{1} \otimes \bz{s_1} \otimes \bp{v} &\scriptstyle\text{if $\ell(s_1v) = \ell(v) + 1$,}
\\
& \phi_{s_0w} \ar[r, mapsto] & - \bm{1} \otimes \bz{s_0} \otimes \bm{w} &\scriptstyle\text{if $\ell(s_0w) = \ell(w) + 1$,}
\\
& \phi_{\omega} \ar[r, mapsto] & (\tau_{s_0} + e_1) \cdot \rep_3(\phi_{s_0^{-1}\omega}) & \scriptstyle\text{for $\omega \in \quoz{T^0}{T^1}$.}
\end{tikzcd}\]
Moreover, $\rep_3$ is a homomorphism of $k[\quoz{T^0}{T^1}]$-bimodules. Finally, defining
\[
	F^1E^3
	\defeq
	\bigoplus_{\mathclap{\substack{w \in \widetilde{W} \\\text{s.t. } \ell(w) \geqslant 1}}} H^3(I,\xx(w)),
\]
one has that the following diagrams are commutative:
\begin{align*}
&\begin{tikzcd}[row sep = 2.7em, column sep = 4em, ampersand replacement = \&]
	F^1E^3
	\ar[r, "{ \restr{\rep_3}{F^1E^3}}"]
	\ar[d, "{ \Gamma_\varpi}"']
	\&
	T^3_{E^0} E^1
	\ar[d, "{ \Gamma_\varpi}"]
	\\
	F^1E^3
	\ar[r, "{ \restr{\rep_3}{F^1E^3}}"']
	\&
	T^3_{E^0} E^1,
\end{tikzcd}
&&\begin{tikzcd}[row sep = 2.7em, column sep = 4em, ampersand replacement = \&]
	F^1E^3
	\ar[r, "{ \restr{\rep_3}{F^1E^3}}"]
	\ar[d, "{ \invol}"']
	\&
	T^3_{E^0} E^1
	\ar[d, "{ \invol}", "{\substack{ \beta \otimes \beta' \\ \downmapsto \\ - \invol(\beta') \otimes \invol(\beta) }}"']
	\\
	F^1E^3
	\ar[r, "{ \restr{\rep_3}{F^1E^3}}"']
	\&
	T^3_{E^0} E^1.
\end{tikzcd}
\end{align*}
\end{lmm}
Note the last line in each of the definitions of $\rep_3$ makes sense because $\rep_3(\phi_{s_0^{-1}})$ is already defined lines before.

\begin{proof}
Let us first check that $\rep_3$, defined in the first way, is a section of $\mult_3$ (and hence, in particular, $\mult_3$ is surjective). Using the computations \eqref{eqFirstBigComputation}, \eqref{eqSecondBigComputation}, \eqref{eqThirdBigComputation} and \eqref{eqFourthBigComputation} in the proof of Proposition \ref{propAlmostGenByE1}, one obtains
\begin{align*}
	\bm{1} \cdot \bz{s_0(s_1s_0)^i} \cdot \bm{1}
	&=
	- \bm{s_0(s_1s_0)^i} \cupprod \bz{s_0(s_1s_0)^i} \cupprod \bp{s_0(s_1s_0)^i},
	\\
	\bp{1} \cdot \bz{s_1(s_0s_1)^i} \cdot \bp{1}
	&=
	- \bp{s_1(s_0s_1)^i} \cupprod \bz{s_1(s_0s_1)^i} \cupprod \bm{s_1(s_0s_1)^i},
	\\
	\bm{1} \cdot \bz{(s_0s_1)^i} \cdot \bp{1}
	&=
	\bm{(s_0s_1)^i} \cupprod \bz{(s_0s_1)^i} \cupprod \bp{(s_0s_1)^i},
	\\
	\bp{1} \cdot \bz{(s_1s_0)^i} \cdot \bm{1}
	&=
	\bp{(s_1s_0)^i} \cupprod \bz{(s_1s_0)^i} \cupprod \bm{(s_1s_0)^i}.
\end{align*}
Recalling from \eqref{eqCond2} that $\phi_w = \bm{w} \cupprod \bz{w} \cupprod \bp{w}$ for all $w \in \widetilde{W}$ such that $\ell(w) \geqslant 1$, we see that $\mult_3 \circ \rep_3$ is the identity map on $F^1 E^3$. For the last line in the definition of $\rep_3$ one uses instead the formula $(\tau_{s_0} + e_1) \cdot \phi_{s_0^{-1}} = \phi_1$ (immediate from \eqref{eqFormulasTopDegSL2}), and so we conclude that $\rep_3$ is indeed a section of $\mult_3$.

The claim about the second description of $\rep_3$ follows from the formulas of Lemma \ref{lmmDeg1GeneratedByTheFourElements}: for example, for all $i \in \zpiu$ and all $\omega \in \quoz{T^0}{T^1}$, we have
\begin{align*}
	- \bp{1} \otimes \bz{(s_1s_0)^i} \otimes \bm{1} \cdot \tau_\omega
	&=
	- \bp{1} \otimes \bz{s_1} \otimes (\tau_{s_0(s_1s_0)^{i-1}} \cdot \bm{1}) \cdot \tau_\omega
	\\&=
	\bp{1} \otimes \bz{s_1} \otimes \bp{s_0(s_1s_0)^{i-1}} \cdot \tau_\omega
	\\&=
	\bp{1} \otimes \bz{s_1} \otimes \bp{s_0(s_1s_0)^{i-1} \omega}.
\end{align*}

The fact that $\rep_3$ is a homomorphism of right $K[\quoz{T^0}{T^1}]$-modules is basically true by construction (using the formula $\phi_{w \omega} = \phi_w \cdot \tau_{\omega}$ for all $w \in \widetilde{W}$ and all $\omega \in \quoz{T^0}{T^1}$, see \eqref{eqFormulasTopDegSL2}).

The fact that $\rep_3$ is also a homomorphism of left $K[\quoz{T^0}{T^1}]$-modules can be proved exactly as in Lemma \ref{lmmBimodulesGpAlgebraFiniteTorus}, replacing the references to the formulas involving $E^2$ with those involving $E^3$ (namely, \eqref{eqFormulasTopDegSL2}).

The fact that the restriction of $\rep_3$ to $F^1E^3$ commutes with $\invol$ and $\Gamma_\varpi$ can be proved in the same way as in Lemmas \ref{lmmCommutesCCCN} and \ref{lmmInvolRep2}, of course replacing the references to the formulas for $E^2$ with the formulas for $E^3$: namely, \eqref{eqCCCNOnE3}, \eqref{eqInvolOnE3} and \eqref{eqFormulasTopDegSL2}. Actually, the first description of $\rep_3$ makes the computations for $\invol$ easier than the analogous proof for $\rep_2$.
\end{proof}

\begin{lmm}\label{lmmCommutesCCCNInvolDeg3}
The map
\[
\func
{\rep_3'}
{E^3}
{T_{E^0}^3 E^1 = E^1 \otimes_{E^0} E^1 \otimes_{E^0} E^1}
{\phi_w}
{\begin{cases}
	\rep_3(\phi_w)
	&\text{if $\ell(w) \geqslant 1$,}
	\\
	\begin{matrix*}[l]
	\scriptstyle \frac{1}{4} \big( \rep_3(\phi_w) + (\Gamma_\varpi \circ \rep_3 \circ \Gamma_\varpi)(\phi_w)
	\\\scriptstyle \quad + (\invol \circ \rep_3 \circ \invol)(\phi_w) + (\Gamma_\varpi \circ \invol \circ \rep_3 \circ \invol \circ \Gamma_\varpi)(\phi_w) \big)
	\\\scriptstyle \quad\quad=\frac{1}{4} \big( \rep_3(\phi_1) + \Gamma_\varpi(\rep_3(\phi_1))
	\\\scriptstyle \quad\quad\quad + \invol(\rep_3(\phi_1)) + \Gamma_\varpi (\invol(\rep_3(\phi_1))) \big) \cdot \tau_w
	\end{matrix*}
	&\text{if $\ell(w) = 0$.}
	\end{cases}}
\]
is well defined, is a section of the multiplication map $\mult_3$ as a homomorphism of $k \left[ \quoz{T^0}{T^1} \right]$-bimodules and commutes with $\Gamma_\varpi$ and with $\invol$.
\end{lmm}

\begin{proof}
Let us consider
\begin{align*}
	F^1E^3
	&\defeq
	\bigoplus_{\mathclap{\substack{w \in \widetilde{W} \\\text{s.t. } \ell(w) \geqslant 1}}} H^3(I,\xx(w)),
\\
	F_0E^3
	&\defeq
	\bigoplus_{\mathclap{\omega \in \quoz{T^0}{T^1}}} H^3(I,\xx(\omega)).
\end{align*}
The fact that $\restr{\rep_3'}{F^1E^3}$ is a homomorphism of $k[\quoz{T^0}{T^1}]$-bimodules, the fact that $\mult_3 \circ \restr{\rep_3'}{F^1E^3}$ is the identity and the fact that $\restr{\rep_3'}{F^1E^3}$ commutes with $\Gamma_\varpi$ and $\invol$ are all clear from the previous lemma, and we now have to study what happens on $F_0E^3$.

We will check in the end that the two definitions of $\rep_3'$ on $F_0E^3$ coincide, but for the moment we stick to the first definition: using that $\mult_3 \circ \restr{\rep_3}{F_0E^3}$ is the identity and that $\mult_3$ commutes with $\Gamma_\varpi$ and $\invol$, we see that $\mult_3 \circ \restr{\rep_3'}{F_0E^3}$ is the identity. Let $\omega \in \quoz{T^0}{T^1}$. Since $\restr{\rep_3}{F_0E^3}$ is a homomorphism of $k[\quoz{T^0}{T^1}]$-bimodules, since $\Gamma_\varpi$ transforms the left (respectively, right) action of $\tau_\omega$ into the left (respectively, right) action by $\tau_{\omega^{-1}}$ and since $\invol$ transforms the left (respectively, right) action of $\tau_\omega$ into the right (respectively, left) action by $\tau_{\omega^{-1}}$, we conclude that $\restr{\rep_3'}{F_0E^3}$ is a homomorphism of $k[\quoz{T^0}{T^1}]$-bimodules (also using \eqref{eqFormulasTopDegSL2}).

The fact that $\restr{\rep_3'}{F_0E^3}$ commutes with $\Gamma_\varpi$ and $\invol$ is clear from the construction and the fact that $\Gamma_\varpi$ and $\invol$ commute by \eqref{eqCCCNInvolCommute} (on $E^\ast$ and hence on $T^\ast_{E^0} E^1$).

For $\omega \in \quoz{T^0}{T^1}$, the alternative description
\[
	\rep_3'(\phi_\omega)
	=
	\frac{1}{4} \Big( \rep_3(\phi_1) + \Gamma_\varpi(\rep_3(\phi_1))
	+ \invol(\rep_3(\phi_1)) + \Gamma_\varpi (\invol(\rep_3(\phi_1))) \Big) \cdot \tau_\omega
\]
is clear from the fact that $\restr{\rep_3'}{F_0E^3}$ is a homomorphism of right $k[\quoz{T^0}{T^1}]$-modules, together with the fact that both $\Gamma_\varpi$ and $\invol$ fix $\phi_1$ (see \eqref{eqCCCNOnE3} and \eqref{eqInvolOnE3}).
\end{proof}

\begin{rem}\label{remRep3Phi1Explicit}
Using the formulas for the action of $\Gamma_\varpi$ and $\invol$ on $E^1$ (respectively, \eqref{eqCCCNOnE1} and \eqref{eqInvolOnE1}) one readily computes explicitly the four summands in the definition of $\rep_3'(\phi_1)$:
\begin{align*}
	\rep_3(\phi_1)
	&=
	- (\tau_{s_0} + e_1) \cdot \bm{1} \otimes \bz{s_0^{-1}} \otimes \bm{1},
\\	
	\Gamma_\varpi(\rep_3(\phi_1))
	&=
	(\tau_{s_1} + e_1) \cdot \bp{1} \otimes \bz{s_1^{-1}} \otimes \bp{1},
\\
	\invol(\rep_3(\phi_1))
	&=
	- \bm{1} \otimes \bz{s_0} \otimes \bm{1} \cdot (\tau_{s_0^{-1}} + e_1),
\\
	\Gamma_\varpi (\invol(\rep_3(\phi_1)))
	&=
	\bp{1} \otimes \bz{s_1} \otimes \bp{1} \cdot (\tau_{s_1^{-1}} + e_1).
\end{align*}
\end{rem}

\begin{lmm}\label{lmmTwoCongruences}
Let $K_{2,3}$ be the sub-$E^0$-bimodule of $T_{E^0}^3 E^1$ generated by the kernel of the degree $2$ multiplication map, i.e.,
\[
	K_{2,3} \defeq \overline{\ker(\mult_2) \otimes_{E^0} E^1} + \overline{E^1 \otimes_{E^0} \ker(\mult_2)} \subseteq T_{E^0}^3 E^1,
\]
where $\overline{(?)}$ denotes the image of $(?)$ in $T_{E^0}^3 E^1$. One has the following congruences:
\begin{align*}
	\Gamma_\varpi (\invol(\rep_3(\phi_1))) &\equiv \rep_3(\phi_1) \mod K_{2,3},
	\\
	\invol(\rep_3(\phi_1)) &\equiv \Gamma_\varpi(\rep_3(\phi_1)) \mod K_{2,3}.
\end{align*}
\end{lmm}

\begin{proof}
First of all we note that the second congruence follows from the first one, since $K_{2,3}$ is $\Gamma_\varpi$-invariant (which is true because the multiplication map commutes with $\Gamma_\varpi$, see \eqref{eqMultiplicationAndCCCNAndInvol}). To show the first congruence, we first compute a couple of useful equalities and congruences: we have
\begin{align*}
	(\tau_{s_0} + e_1) \cdot \bm{1} \cdot \tau_{s_0}
	&=
	(\tau_{s_0} + e_1) \cdot \bm{s_0}
	\\&=
	- 2 e_{\idd} \bz{s_0} + e_{\idd^2} \bp{s_0} - \bp{c_{-1}},
\end{align*}
by \eqref{eqFormulaRighteasyQp} and \eqref{eqTauS0OnE1LeftNotAddUp}.
Since both $\bz{s_0} \otimes \bz{s_1}$ and $\bp{s_0} \otimes \bz{s_1} = - \tau_{s_0} \cdot \bm{1} \otimes \bz{s_1}$ lie in $\ker(\mult_2)$ (see, e.g., Lemma \ref{lmmSomeElementsInTheKer} and see \eqref{eqTauS0OnE1LeftNotAddUp} for the claimed equality), tensoring both sides of the last equality by $\bz{s_1}$ on the right, we get the congruence
\begin{equation}\label{eqCongOne}
	((\tau_{s_0} + e_1) \cdot \bm{1} \cdot \tau_{s_0}) \otimes \bz{s_1}
	\equiv
	- \bp{1} \otimes \bz{s_1^{-1}}
	\mod \ker(\mult_2)
\end{equation}
(also using the formulas for the left and right action of $\tau_{c_{-1}}$ on $E^1$: see \eqref{eqOmegaOnE1Left} and \eqref{eqOmegaOnE1Right}).
Now we apply $\Gamma_\varpi$ to the last congruence, obtaining again a congruence since $\ker(\mult_2)$ is $\Gamma_\varpi$-invariant:
\[
	- ((\tau_{s_1} + e_1) \cdot \bp{1} \cdot \tau_{s_1}) \otimes \bz{s_0}
	\equiv
	\bm{1} \otimes \bz{s_0^{-1}}
	\mod \ker(\mult_2),
\]
where we have used the formulas for the action of $\Gamma_\varpi$ on $E^0$ and $E^1$ (see \eqref{eqCCCNOnH} and \eqref{eqCCCNOnE2}).
Now we apply $\invol$ (or rather $-\invol$) to the last congruence, obtaining again a congruence since $\ker(\mult_2)$ is $\invol$-invariant (by \eqref{eqMultiplicationAndCCCNAndInvol}):
\[
	- \invol(\bz{s_0}) \otimes \invol\big((\tau_{s_1} + e_1) \cdot \bp{1} \cdot \tau_{s_1}\big)
	\equiv
	\invol(\bz{s_0^{-1}}) \otimes \invol(\bm{1})
	\mod \ker(\mult_2),
\]
i.e.,
\[
	\bz{s_0^{-1}} \otimes (\tau_{s_1^{-1}} \cdot \bp{1} \cdot (\tau_{s_1^{-1}} + e_1))
	\equiv
	- \bz{s_0} \otimes \bm{1}
	\mod \ker(\mult_2),
\]
using the formulas for the action of $\invol$ on $E^0$ and $E^1$ (see \eqref{eqInvolTauw} and \eqref{eqInvolOnE1}). Equivalently, using the formulas for the action of $\tau_{c_{-1}}$ (\eqref{eqOmegaOnE1Left} and \eqref{eqOmegaOnE1Right}),
\begin{equation}\label{eqCongTwo}
	\bz{s_0} \otimes (\tau_{s_1} \cdot \bp{1} \cdot (\tau_{s_1} + e_1))
	\equiv
	- \bz{s_0^{-1}} \otimes \bm{1}
	\mod \ker(\mult_2),
\end{equation}
We now get the desired congruence $\Gamma_\varpi (\invol(\rep_3(\phi_1))) \equiv \rep_3(\phi_1)$ modulo $K_{2,3}$, by putting together \eqref{eqCongOne} and \eqref{eqCongTwo}: indeed denoting by $\equiv$ congruences modulo $K_{2,3}$, we have
\begin{align*}
	&\Gamma_\varpi (\invol(\rep_3(\phi_1)))
	\\&\quad=
	\bp{1} \otimes \bz{s_1} \otimes \bp{1} \cdot (\tau_{s_1^{-1}} + e_1)
	&&\text{by Remark \ref{remRep3Phi1Explicit}}
	\\&\quad=
	\bp{1} \otimes \bz{s_1} \otimes \bp{1} \cdot \tau_{c_{-1}} \cdot (\tau_{s_1} + e_1)
	\\&\quad=
	\bp{1} \otimes \bz{s_1^{-1}} \otimes \bp{1} \cdot (\tau_{s_1} + e_1)
	&&\text{by \eqref{eqOmegaOnE1Left} and \eqref{eqOmegaOnE1Right}}
	\\&\quad\equiv
	- ((\tau_{s_0} + e_1) \cdot \bm{1} \cdot \tau_{s_0}) \otimes \bz{s_1} \otimes \bp{1} \cdot (\tau_{s_1} + e_1)
	&&\text{by \eqref{eqCongOne}}
	\\&\quad=
	(\tau_{s_0} + e_1) \cdot \bm{1} \otimes \bz{s_0} \otimes (\tau_{s_1} \cdot \bp{1} \cdot (\tau_{s_1} + e_1))
	&&\text{by \eqref{eqFormulasTauQp} and \eqref{eqFormulaRighteasyQp}}
	\\&\quad\equiv
	- (\tau_{s_0} + e_1) \cdot \bm{1} \otimes \bz{s_0^{-1}} \otimes \bm{1}
	&&\text{by \eqref{eqCongTwo}}
	\\&\quad=
	\rep_3(\phi_1).
&&\qedhere
\end{align*}
\end{proof}

\begin{prop}\label{propKerDeg3}
As in the last lemma, let $K_{2,3}$ be the sub-$E^0$-bimodule of $T_{E^0}^3 E^1$ generated by the kernel of the degree $2$ multiplication map, i.e.,
\[
	K_{2,3} \defeq \overline{\ker(\mult_2) \otimes_{E^0} E^1} + \overline{E^1 \otimes_{E^0} \ker(\mult_2)} \subseteq T_{E^0}^3 E^1,
\]
where $\overline{(?)}$ denotes the image of $(?)$ in $T_{E^0}^3 E^1$. Furthermore, let $K_{\operatorname{extra},3}$ be the sub-left-$k \left[ \quoz{T^0}{T^1} \right]$-module generated by the following element:
\begin{align*}
	&\Gamma_\varpi(\rep_3(\phi_1)) - \rep_3(\phi_1)
	\\&\qquad=
	(\tau_{s_1} + e_1) \cdot \bp{1} \otimes \bz{s_1^{-1}} \otimes \bp{1}
	+
	(\tau_{s_0} + e_1) \cdot \bm{1} \otimes \bz{s_0^{-1}} \otimes \bm{1}.	
\end{align*}
One has that the kernel $\ker(\mult_3)$ of the multiplication map in degree $3$ coincides with $K_{2,3} + K_{\operatorname{extra},3}$. Furthermore, $K_{\operatorname{extra},3}$ is also the sub-right-$k \left[ \quoz{T^0}{T^1} \right]$-module generated by the above element.
\end{prop}

\begin{proof}
The last claim is easy to see using that $\rep_3$ is a homomorphism of $k \left[ \quoz{T^0}{T^1} \right]$-bimodules (Lemma \ref{lmmCommutesCCCNInvolDeg3}), that $\Gamma_\varpi$ transforms multiplication by $\tau_\omega$ into multiplication by $\tau_{\omega^{-1}}$ and that $k \left[ \quoz{T^0}{T^1} \right]$ centralizes $\phi_1$.

We now turn to the proof of the main claim. Let us define
\[
	V \defeq K_{2,3} + K_{\operatorname{extra},3} + \image(\rep_3') \subseteq T_{E^0}^3 E^1,
\]
where the map $\rep_3'$ was defined in Lemma \ref{lmmCommutesCCCNInvolDeg3}.
To prove the proposition, it suffices to show that $V = T_{E^0}^3 E^1$ (i.e., the inclusion from right to left): indeed assuming this equality, we obtain
\[
	K_{2,3} + K_{\operatorname{extra},3} + \image(\rep_3') = \ker(\mult_3) \oplus  \image(\rep_3'),
\]
and this, together with the \virgolette{easy inclusion} $K_{2,3} + K_{\operatorname{extra},3} \subseteq \ker(\mult_3)$, shows that we have $K_{2,3} + K_{\operatorname{extra},3} = \ker(\mult_3)$, as we wanted (for the inclusion $K_{\operatorname{extra},3} \subseteq \ker(\mult_3)$, recall from \eqref{eqMultiplicationAndCCCNAndInvol} that $\Gamma_\varpi$ commutes with $\mult$ and from \eqref{eqCCCNOnE3} that $\Gamma_\varpi$ fixes $\phi_1$).

To reduce the amount of computations, we first make some preliminary observations.
\begin{enumerate}[label=(\alph*)]
\item\label{item1PropDeg3}
	We remark that one has the congruences
	\[
		\Gamma_\varpi(\invol(\phi_1)) \equiv \rep_3(\phi_1) \equiv \Gamma_\varpi(\rep_3(\phi_1)) \equiv \invol(\rep_3(\phi_1))
		\mod K_{2,3} + K_{\operatorname{extra},3}.
	\]
	Indeed the first and last congruence follow from Lemma \ref{lmmTwoCongruences}, while the second congruence holds by definition of $K_{\operatorname{extra},3}$. Looking at the definition of $\rep_3'$, this shows that
	\begin{equation}\label{eqRep3Rep3Prime}\begin{aligned}
		\rep_3'(\phi_1)
		&\equiv
		\rep_3(\phi_1)
		\equiv
		\Gamma_\varpi(\rep_3(\phi_1))
		\\&\equiv
		\invol(\rep_3(\phi_1))
		\equiv
		\Gamma_\varpi(\invol(\rep_3(\phi_1)))
		\mod K_{2,3} + K_{\operatorname{extra},3}.
	\end{aligned}\end{equation}
\item\label{item2PropDeg3}
	We remark that $V$ is invariant for $\Gamma_\varpi$ and $\invol$.
	\begin{itemize}
		\item The term $K_{2,3}$ is invariant for $\Gamma_\varpi$ and $\invol$ because both involutions commute with $\mult_2$ (see \eqref{eqMultiplicationAndCCCNAndInvol}).
		\item The term $\image(\rep_3')$ is invariant for $\Gamma_\varpi$ and $\invol$ because these involutions commute with $\rep_3'$ (Lemma \ref{lmmCommutesCCCNInvolDeg3}).
		\item The term $K_{\operatorname{extra},3}$ is visibly invariant for $\Gamma_\varpi$. Moreover, applying $\invol$ to the difference $\Gamma_\varpi(\rep_3(\phi_1)) - \rep_3(\phi_1)$ and using that $\Gamma_\varpi$ and $\invol$ commute on $E^\ast$ by \eqref{eqCCCNInvolCommute} and hence also on $T_{E^0}^3 E^1$, we obtain the element $\Gamma_\varpi(\invol(\rep_3(\phi_1))) - \invol(\rep_3(\phi_1))$, which lies in $K_{2,3} + K_{\operatorname{extra},3}$ by part \ref{item1PropDeg3}. Also considering the behaviour of $\invol$ with respect to multiplication by $\tau_\omega$ for $\omega \in \quoz{T^0}{T^1}$, we obtain that $\invol(K_{\operatorname{extra},3}) \subseteq K_{2,3} + K_{\operatorname{extra},3}$.
	\end{itemize}
\item\label{item3PropDeg3}
	We further remark that $V$ is a sub-$k \left[ \quoz{T^0}{T^1} \right]$-bimodule of $T_{E^0}^3 E^1$: indeed the term $K_{2,3}$ is clearly even a sub-$E^0$-bimodule, the term $K_{\operatorname{extra},3}$ is a sub-$k \left[ \quoz{T^0}{T^1} \right]$-bimodule as already discussed and the term $\image(\rep_3')$ is a sub-$k \left[ \quoz{T^0}{T^1} \right]$-bimodule because $\rep_3'$ is a homomorphism of $k \left[ \quoz{T^0}{T^1} \right]$-bimodules (Lemma \ref{lmmCommutesCCCNInvolDeg3}).
\end{enumerate}

Now we can start the actual proof of the proposition. Let us consider $x \in T_{E^0}^3 E^1$ and let us prove that $x \in V$. Without loss of generality, we can of course assume that $x$ is of the form $y \otimes z$ for $y \in E^1$ and $z \in E^1 \otimes_{E^0} E^1$. Considering the equality
\[
	x = y \otimes z = y \otimes (z - \rep_2(\mult_2(z))) + y \otimes \rep_2(\mult_2(z)),
\]
we see that the first summand on the right hand side lies in $\overline{E^1 \otimes_{E^0} \ker(\mult_2)}$ and hence in $V$. Therefore, without loss of generality, we can assume that $x$ is of the form $y \otimes z'$ for some $y \in E^1$ and some $z' \in \image(\rep_2)$. From Lemma \ref{lmmInvolRep2} we obtain that $\rep_2(F^1E^2) = \invol(\rep_2(F^1E^2))$. Using the formulas for the action of $\invol$ on $E^1$ \eqref{eqInvolOnE1} and looking at the explicit definition of $\rep_2$ in Lemma \ref{lmmDefRep2} we then see that every $z' \in \image(\rep_2)$ can be written as a sum of simple tensors of the following forms:
\begin{align*}
	&u \otimes \bm{1},
	&
	&u \otimes \bp{1},
	&
	&u \otimes \bz{s_0},
	&
	&u \otimes \bz{s_1},
\end{align*}
for $u \in E^1$.
Hence, without loss of generality, we can assume that $z'$ is of one of those forms. To simplify a little further, we note that since $V$ is invariant for $\Gamma_\varpi$, it follows that it suffices to prove that the elements of the following forms lie in $V$:
\begin{enumerate}[label=(\roman*)]
\item\label{itemiPropDeg3}
	$y \otimes u \otimes \bm{1}$ for some $u \in E^1$,
\item\label{itemiiPropDeg3}
	$y \otimes u \otimes \bz{s_0}$ for some $u \in E^1$,
\end{enumerate}
because by applying $\Gamma_\varpi$ to such elements we immediately obtain what is left (see \eqref{eqCCCNOnE1}). We treat separately elements of the forms \ref{itemiPropDeg3} and \ref{itemiiPropDeg3}.
\begin{enumerate}
\item[\ref{itemiPropDeg3}]
Let us start the proof that every element of the form
\begin{align*}
	&y \otimes u \otimes \bm{1} &&\text{(for some $y,u \in E^1$)}
\end{align*}
lies in $V$. Similarly to what we did before, we can consider the equality
\[
	y \otimes u \otimes \bm{1} = (y \otimes u - \rep_2(\mult_2(y \otimes u)))  \otimes \bm{1} + \rep_2(\mult_2(y \otimes u)) \otimes \bm{1},
\]
from which we see that we may consider instead elements of the form $r \otimes \bm{1}$ for $r \in \image (\rep_2)$, and, looking at the explicit definition of $\rep_2$ (Lemma \ref{lmmDefRep2}), we can further assume that $r \otimes \bm{1}$ is of one of the following forms:
\begin{align*}
&\bp{1} \otimes \bz{s_1v} \otimes \bm{1} &&\text{for $v \in \widetilde{W}$ such that $\ell(s_1v) = \ell(v) + 1$,}
\\
&\bp{1} \otimes \bm{s_1v} \otimes \bm{1} &&\text{for $v \in \widetilde{W}$ such that $\ell(s_1v) = \ell(v) + 1$,}
\\
&\bz{s_1} \otimes \bp{v} \otimes \bm{1} &&\text{for $v \in \widetilde{W}$ such that $\ell(s_1v) = \ell(v) + 1$,}
\\
&\bz{s_0} \otimes \bm{w} \otimes \bm{1} &&\text{for $w \in \widetilde{W}$ such that $\ell(s_0w) = \ell(w) + 1$,}
\\
&\bm{1} \otimes \bp{s_0w} \otimes \bm{1} &&\text{for $w \in \widetilde{W}$ such that $\ell(s_0w) = \ell(w) + 1$,}
\\
&\bm{1} \otimes \bz{s_0w} \otimes \bm{1} &&\text{for $w \in \widetilde{W}$ such that $\ell(s_0w) = \ell(w) + 1$,}
\\
&\tau_{s_0} \cdot \rep_2 \big( \ap{s_0^{-1}\omega} \big) \otimes \bm{1} &&\text{for $\omega \in \quoz{T^0}{T^1}$,}
\\
&\tau_{s_1} \cdot \rep_2 \big( \am{s_1^{-1}\omega} \big) \otimes \bm{1} &&\text{for $\omega \in \quoz{T^0}{T^1}$.}
\end{align*}
If $v$ is as above, from Lemma \ref{lmmDeg1GeneratedByTheFourElements} and from \eqref{eqOmegaOnE1Left} one sees that either $\bm{s_1v} \in k^\times \tau_{s_1v} \bm{1}$ or $\bm{s_1v} \in k^\times \tau_{s_1v} \bp{1}$. This justifies the first line below:
\begin{align*}
	\bm{s_1v} \otimes \bm{1} &\in E^0 \cdot \{\bm{1} \otimes \bm{1}, \bp{1} \otimes \bm{1}\},
	\\
	\bp{v} \otimes \bm{1} &\in E^0 \cdot \{\bm{1} \otimes \bm{1}, \bp{1} \otimes \bm{1}\},
	\\
	\bm{w} \otimes \bm{1} &\in E^0 \cdot \{\bm{1} \otimes \bm{1}, \bp{1} \otimes \bm{1}\},
	\\
	\bp{s_0w} \otimes \bm{1} &\in E^0 \cdot \{\bm{1} \otimes \bm{1}, \bp{1} \otimes \bm{1}\},
\end{align*}
and the other lines can be justified with a completely analogous argument.

Since the two elements $\bm{1} \otimes \bm{1}$ and $\bp{1} \otimes \bm{1}$ lie in $\ker(\mult_2)$ (see, e.g., Lemma \ref{lmmSomeElementsInTheKer}), we see that the elements $\bp{1} \otimes \bm{s_1v} \otimes \bm{1}$, $\bz{s_1} \otimes \bp{v} \otimes \bm{1}$, $\bz{s_0} \otimes \bm{w} \otimes \bm{1}$ and $\bm{1} \otimes \bp{s_0w} \otimes \bm{1}$ in the above list all lie in $V$. Also recalling that $V$ is a left-$k \left[ \quoz{T^0}{T^1} \right]$-submodule, we see that we are reduced to showing that the following elements lie in $V$:
\begin{equation}\label{eqNonSo}\begin{aligned}
&\bp{1} \otimes \bz{(s_1s_0)^i} \otimes \bm{1} && \text{for $i \in \zpiu$,}
\\
&\bp{1} \otimes \bz{s_1(s_0s_1)^i} \otimes \bm{1} && \text{for $i \in \nn$,}
\\
&\bm{1} \otimes \bz{(s_0s_1)^i} \otimes \bm{1} && \text{for $i \in \zpiu$,}
\\
&\bm{1} \otimes \bz{s_0(s_1s_0)^i} \otimes \bm{1} && \text{for $i \in \nn$,}
\\
&\tau_{s_0} \cdot \bm{1} \otimes \bz{s_0^{-1}} \otimes \bm{1}, &&
\\
&\tau_{s_1} \cdot \bp{1} \otimes \bz{s_1^{-1}} \otimes \bm{1}. &&
\end{aligned}\end{equation}
Up to a sign, the first line (respectively, the fourth line) is $\rep_3'(\phi_{(s_1s_0)^i})$ (respectively, $\rep_3'(\phi_{s_0(s_1s_0)^i})$), and hence they both lie in $V$. Moreover, we know that $\bz{s_1} \otimes \bm{1} \in \ker(\mult_2)$ (see for example Lemma \ref{lmmSomeElementsInTheKer}), from which it follows that
\begin{align*}
	\bz{(s_0s_1)^i} \otimes \bm{1} &= - \tau_{(s_0s_1)^{i-1}s_0} \cdot \bz{s_1} \otimes \bm{1} \in \ker(\mult_2) &&\text{for all $i \in \zpiu$,}
	\\
	\bz{s_1(s_0s_1)^i} \otimes \bm{1} &= \tau_{(s_1s_0)^i} \cdot \bz{s_1} \otimes \bm{1} \in \ker(\mult_2) &&\text{for all $i \in \nn$}
\end{align*}
(where we used \eqref{eqH}).
Therefore, it follows that also the second, the third and the sixth line in \eqref{eqNonSo} lie in $V$. It remains to consider the element $\tau_{s_0} \cdot \bm{1} \otimes \bz{s_0^{-1}} \otimes \bm{1}$. First of all, we see that there is no harm in considering instead the element $(\tau_{s_0} + e_1) \cdot \bm{1} \otimes \bz{s_0^{-1}} \otimes \bm{1} = - \rep_3(\phi_1)$, since $e_1 \cdot \bm{1} \otimes \bz{s_0^{-1}} \otimes \bm{1} \in \image(\rep_3') \subseteq V$. But we have seen in formula \eqref{eqRep3Rep3Prime} that $\rep_3(\phi_1)$ is congruent to $\rep_3'(\phi_1)$ modulo $K_{2,3} + K_{\operatorname{extra},3}$, and hence $\rep_3(\phi_1) \in V$, thus completing the proof that all the lines in \eqref{eqNonSo} lie in $V$.
\item[\ref{itemiiPropDeg3}]
We have to consider the elements of the form
\begin{align*}
	&y \otimes u \otimes \bz{s_0} = y \otimes u \otimes \bz{s_0} &&\text{(for some $y,u \in E^1$)}
\end{align*}
and prove that they lie in $V$. As before, we see that we may only treat the elements of the form $r \otimes \bz{s_0}$ for $r \in \image (\rep_2)$, and, looking at the explicit definition of $\rep_2$ (Lemma \ref{lmmDefRep2}), we can further assume that $r \otimes \bz{s_0}$ is of one of the following forms:
\begin{align*}
&\bp{1} \otimes \bz{s_1v} \otimes \bz{s_0} &&\text{for $v \in \widetilde{W}$ such that $\ell(s_1v) = \ell(v) + 1$,}
\\
&\bp{1} \otimes \bm{s_1v} \otimes \bz{s_0} &&\text{for $v \in \widetilde{W}$ such that $\ell(s_1v) = \ell(v) + 1$,}
\\
&\bz{s_1} \otimes \bp{v} \otimes \bz{s_0} &&\text{for $v \in \widetilde{W}$ such that $\ell(s_1v) = \ell(v) + 1$,}
\\
&\bz{s_0} \otimes \bm{w} \otimes \bz{s_0} &&\text{for $w \in \widetilde{W}$ such that $\ell(s_0w) = \ell(w) + 1$,}
\\
&\bm{1} \otimes \bp{s_0w} \otimes \bz{s_0} &&\text{for $w \in \widetilde{W}$ such that $\ell(s_0w) = \ell(w) + 1$,}
\\
&\bm{1} \otimes \bz{s_0w} \otimes \bz{s_0} &&\text{for $w \in \widetilde{W}$ such that $\ell(s_0w) = \ell(w) + 1$,}
\\
&\tau_{s_0} \cdot \rep_2 \big( \ap{s_0^{-1}\omega} \big) \otimes \bz{s_0} &&\text{for $\omega \in \quoz{T^0}{T^1}$,}
\\
&\tau_{s_1} \cdot \rep_2 \big( \am{s_1^{-1}\omega} \big) \otimes \bz{s_0} &&\text{for $\omega \in \quoz{T^0}{T^1}$.}
\end{align*}
In part \ref{itemiPropDeg3} we have proved that all the elements of the form $x \otimes y \otimes \bm{1}$ (for some $x,y \in E^1$) lie in $V$. As we have already said that $V$ is $\invol$-invariant and $\Gamma_\varpi$-invariant, we deduce that the elements of the following forms lie in $V$:
\begin{align*}
	&\bm{1} \otimes x \otimes y &&\text{for some $x,y \in E^1$,}
	\\
	&\bp{1} \otimes x \otimes y &&\text{for some $x,y \in E^1$.}
\end{align*}
This shows that most of the elements in the above list are in $V$. For the remaining ones, we first use Lemma \ref{lmmDeg1GeneratedByTheFourElements} and \eqref{eqOmegaOnE1Left} as in part \ref{itemiPropDeg3} to reduce to checking that the following elements lie in $V$:
\begin{align*}
& \bz{(s_1s_0)^i} \otimes \bm{1} \otimes \bz{s_0} && \text{for $i \in \zpiu$,}
\\
& \bz{s_1(s_0s_1)^i} \otimes \bp{1} \otimes \bz{s_0} && \text{for $i \in \nn$,}
\\
& \bz{(s_0s_1)^i} \otimes \bp{1} \otimes \bz{s_0} && \text{for $i \in \zpiu$,}
\\
& \bz{s_0(s_1s_0)^i} \otimes \bm{1} \otimes \bz{s_0} && \text{for $i \in \nn$,}
\\
&\tau_{s_0} \cdot \bm{1} \otimes \bz{s_0^{-1}} \otimes \bz{s_0}, &&
\\
&\tau_{s_1} \cdot \bp{1} \otimes \bz{s_1^{-1}} \otimes \bz{s_0}. &&
\end{align*}
We recall from Lemma \ref{lmmSomeElementsInTheKer} that $\bp{1} \otimes \bz{s_0} \in \ker(\mult_2)$ and that the same is true for $\bz{s_1} \otimes \bz{s_0}$, and so it only remains to consider the following elements:
\begin{align}
& \bz{(s_1s_0)^i} \otimes \bm{1} \otimes \bz{s_0} && \text{for $i \in \zpiu$,}
\label{eqBetaS0One}
\\
& \bz{s_0(s_1s_0)^i} \otimes \bm{1} \otimes \bz{s_0} && \text{for $i \in \nn$,}
\label{eqBetaS0Two}
\\
&\tau_{s_0} \cdot \bm{1} \otimes \bz{s_0^{-1}} \otimes \bz{s_0}. &&
\label{eqBetaS0Three}
\end{align}
Let us treat the last line, where, multiplying by $\tau_{c_{-1}}$, we can replace $s_0^{-1}$ by $s_0$. We consider the following element of $\ker(\mult_2)$ (Lemma \ref{lmmSomeElementsInTheKer}):
\[
	\bz{s_0} \otimes \bz{s_0} + e_{\idd^{-1}} \cdot \bm{1} \otimes \bz{s_0} + e_{\idd} \cdot \bz{s_0} \otimes \bm{1} - e_1 \cdot \bm{1} \otimes \bp{s_0}.
\]
Tensoring by $\tau_{s_0} \cdot \bm{1}$ on the left we get
\begin{align*}
	&\tau_{s_0} \cdot \bm{1} \otimes \bz{s_0} \otimes \bz{s_0}
	+ \tau_{s_0} \cdot \bm{1} \otimes e_{\idd^{-1}} \bm{1} \otimes \bz{s_0}
	\\&\qquad+ \tau_{s_0} \cdot \bm{1} \otimes e_{\idd} \bz{s_0} \otimes \bm{1}
	- \tau_{s_0} \cdot \bm{1} \otimes e_1 \bm{1} \otimes \bp{s_0}
	\\&\qquad\qquad\in K_{2,3}.
\end{align*}
Since $\bm{1} \otimes \bm{1}$ lies in $\ker(\rep_2)$, we can delete the two terms where this element appears, getting that
\[
	\tau_{s_0} \cdot \bm{1} \otimes \bz{s_0} \otimes \bz{s_0}
	+ \tau_{s_0} \cdot \bm{1} \otimes e_{\idd} \bz{s_0} \otimes \bm{1}
	\in K_{2,3}.
\]
In part \ref{itemiPropDeg3} we have seen that the element $\tau_{s_0} \cdot \bm{1} \otimes \bz{s_0^{-1}} \otimes \bm{1}$ lies in $V$, and so this proves that \eqref{eqBetaS0Three} lies in $V$.

It remains to prove that the elements in \eqref{eqBetaS0One} and in \eqref{eqBetaS0Two} lie in $V$. We claim that $\bz{s_0} \otimes \bm{1} \otimes \bz{s_0}$ lies in $K_{2,3}$. If we show this, then we are done because $K_{2,3}$ is a sub-$E^0$-bimodule, and using \eqref{eqH} we get that the elements in the lines \eqref{eqBetaS0One} and \eqref{eqBetaS0Two} lie in $K_{2,3}$ as well and hence also to $V$. So, let us show that $\bz{s_0} \otimes \bm{1} \otimes \bz{s_0}$ lies in $K_{2,3}$.
\begin{equation}\label{eqbzbmbz}\begin{aligned}
	\bz{s_0} \otimes \bm{1} \otimes \bz{s_0}
	&=
	\rep_2(- \am{s_0}) \otimes \bz{s_0}
	\\&=
	\rep_2(\tau_{s_0} \cdot \ap{1}) \otimes \bz{s_0}
	\\&\qquad\qquad\qquad\qquad\text{(by \eqref{eqExplicitDeg2AddUp})}
	\\&\equiv
	\big( \tau_{s_0} \cdot \rep_2(\ap{1}) \big) \otimes \bz{s_0}
	\mod K_{2,3}
\end{aligned}\end{equation}
We now compute $\rep_2(\ap{1})$ explicitly
\begin{align*}
	\rep_2(\ap{1})
	&=
	\rep_2 \big( \ap{1} + \tau_{s_1} \cdot \am{s_1^{-1}} \big) - \tau_{s_1} \cdot \rep_2 \big( \am{s_1^{-1}} \big)
	\\&\qquad\qquad\qquad\text{by Lemma \ref{lmmDefRep2}}
	\\&=
	\rep_2 \big( - e_1 \am{s_1^{-1}} + e_{\idd^{-1}} \az{s_1^{-1}} + e_{\idd^{-2}} \ap{s_1^{-1}} \big) - \tau_{s_1} \cdot \rep_2 \big( \am{s_1^{-1}} \big)
	\\&\qquad\qquad\qquad\text{by \eqref{eqTauS1OnE2LeftNotAddUp}}
	\\&=
	e_1 \cdot \bp{1} \otimes \bz{s_1^{-1}}
	+ e_{\idd^{-1}} \cdot \bp{1} \otimes \bm{s_1^{-1}}
	\\&\qquad+ e_{\idd^{-2}} \cdot \bz{s_1} \otimes \bp{c_{-1}}
	+ \tau_{s_1} \cdot \bp{1} \otimes \bz{s_1^{-1}}
	\\&\qquad\qquad\qquad\text{by Lemma \ref{lmmDefRep2} and Lemma \ref{lmmBimodulesGpAlgebraFiniteTorus}}
	\\&=
	e_1 \cdot \bp{c_{-1}} \otimes \bz{s_1}
	- e_{\idd^{-1}} \cdot (\bp{1} \cdot \tau_{s_1^{-1}}) \otimes \bp{1}
	\\&\qquad+ e_{\idd^{-2}} \cdot \bz{s_1^{-1}} \otimes \bp{1}
	+ \tau_{s_1} \cdot \bp{c_{-1}} \otimes \bz{s_1}
	\\&\qquad\qquad\qquad\text{by \eqref{eqOmegaOnE1Left}, \eqref{eqOmegaOnE1Right}, \eqref{eqF}}
\end{align*}
Going back to \eqref{eqbzbmbz}, and recalling from Lemma \ref{lmmSomeElementsInTheKer} that both $\bp{1} \otimes \bz{s_0}$ and $\bz{s_1} \otimes \bz{s_0}$ lie in $\ker(\mult_2)$, we conclude that $\bz{s_0} \otimes \bm{1} \otimes \bz{s_0}$ lies in $K_{2,3}$, as claimed. This concludes the proof that all the elements of the forms \eqref{eqBetaS0One} and \eqref{eqBetaS0Two} lie in $V$, and with it the proof that all the elements of the form \ref{itemiiPropDeg3} lie in $V$.
\qedhere
\end{enumerate}
\end{proof}

\subsection{The \texorpdfstring{$4$\textsuperscript{th}}{4th} graded piece of the kernel}

Since $E^4 = 0$, the kernel of the multiplication map $\funcInline{\mult_4}{T_{E^0}^4 E^1}{E^4}$ is of course the whole $T_{E^0}^4 E^1$. As we computed generators for $\ker(\mult_2)$ as an $E^0$-bimodule (Proposition \ref{propKerDeg2}), and we computed $\ker(\mult_3)$ in terms of $\ker(\mult_2)$ and an additional generator (Proposition \ref{propKerDeg3}), we now would like to compute $\ker(\mult_4) = T_{E^0}^4 E^1$ in terms of $\ker(\mult_3)$ (and, a priori, some other generators). The following result achieves this, showing that no further generators are needed and hence that $\ker(\mult)$ is generated as a two-sided ideal by its $2$\textsuperscript{nd} and $3$\textsuperscript{rd} graded pieces.

\begin{prop}\label{propDeg4}
Let $\funcInline{\mult}{T_{E^0}^\ast E^1}{E^\ast}$ be the multiplication map, and let
\[
	K_{3,4} \defeq \overline{\ker(\mult_3) \otimes_{E^0} E^1} + \overline{E^1 \otimes_{E^0} \ker(\mult_3)} \subseteq T_{E^0}^4 E^1,
\]
where $\overline{(?)}$ denotes the image of $(?)$ in $T_{E^0}^4 E^1$. One has that $K_{3,4} = T_{E^0}^4 E^1$.
\end{prop}

\begin{proof}
We have to prove that every element of $ T_{E^0}^4 E^1$ lies in $K_{3,4}$, and clearly it suffices to prove this for \virgolette{simple tensors} of the form $x \otimes y$ for $x \in E^1$ and $y \in T_{E^0}^3 E^1$. Considering the equality
\[
	x \otimes y = x \otimes \left( y - \rep_3(\mult_3(y)) \right) + x \otimes \rep_3(\mult_3(y)),
\]
we see that it suffices to prove our claim for elements of the form $x \otimes y'$ for $x \in E^1$ and $y' \in \image(\mult_3)$. Now looking at the explicit definition of $\rep_3$ in Lemma \ref{lmmDefRep3}, we see that every $y' \in \image(\mult_3)$ can be written as a $k$-linear combination of tensors of the form $z \otimes t \otimes \bm{1}$ for some $z,t \in E^1$ and of the form $z \otimes t \otimes \bp{1}$ for some $z,t \in E^1$ (the presence of $\tau_\omega$ in the definition of $\rep_3$ is irrelevant, since it commutes, up to a scalar, with both $\bm{1}$ and $\bp{1}$). We are thus reduced to showing that the elements of the form $u \otimes \bm{1}$ and $u \otimes \bp{1}$, for $u \in T_{E^0}^3E^1$, lie in $K_{3,4}$. With the same trick as before, we see that we can assume that $u \in \image(\mult_3)$. Using the same property of $\image(\mult_3)$ as above, we are thus reduced to showing that the elements of the form
\begin{align*}
	&z \otimes t \otimes \bm{1} \otimes \bm{1},
	&
	&z \otimes t \otimes \bp{1} \otimes \bm{1},
	\\
	&z \otimes t \otimes \bm{1} \otimes \bp{1},
	&
	&z \otimes t \otimes \bp{1} \otimes \bp{1},
\end{align*}
(for $z,t \in E^1$) all lie in $K_{3,4}$. This is clear, since all of the elements $\bm{1} \otimes \bm{1}$, $\bp{1} \otimes \bm{1}$, $\bm{1} \otimes \bp{1}$, and $\bp{1} \otimes \bp{1}$ lie in $\ker(\mult_2)$ (see, e.g., Lemma \ref{lmmSomeElementsInTheKer}).
\end{proof}

\subsection{First main result: finite presentation over the tensor algebra of \texorpdfstring{$E^1$}{E\textasciicircum{}1}}

We are now going to state the first main result of this section, consisting in a presentation of the $\ext$-algebra $E^\ast$ in terms of the tensor algebra $T_{E^0}^\ast E^1$.

\begin{thm}\label{thmFinalResultKernel}
Recall the assumptions stated at the beginning of \S\ref{sectionFinPres}.
The multiplication map
\[
	\funcInline{\mult}{T_{E^0}^\ast E^1}{E^\ast}
\]
is surjectve and its kernel is generated by its $2$\textsuperscript{nd} and $3$\textsuperscript{rd} graded pieces. Furthermore, the kernel is finitely generated as a two-sided ideal and the following is an explicit list of generators:
\begin{align*}
&\begin{aligned}
	&\bm{1} \otimes \bm{1},
	&&\quad\bp{1} \otimes \bm{1},
	&&\quad\bz{s_1} \otimes \bm{1},
	&&\quad\bp{1} \otimes \bz{s_0},
	&&\quad\bz{s_1} \otimes \bz{s_0},
	\\
	&\bm{1} \otimes \bp{1},
	&&\quad\bp{1} \otimes \bp{1},
	&&\quad\bz{s_0} \otimes \bp{1},
	&&\quad\bm{1} \otimes \bz{s_1},
	&&\quad\bz{s_0} \otimes \bz{s_1},
\end{aligned}
\\
	&\bz{s_0} \otimes \bz{s_0} + e_{\idd^{-1}} \cdot \bm{1} \otimes \bz{s_0} + e_{\idd} \cdot \bz{s_0} \otimes \bm{1} - e_1 \cdot \bm{1} \otimes \bp{s_0}, 
\\
	&\bz{s_1} \otimes \bz{s_1} - e_{\idd} \cdot \bp{1} \otimes \bz{s_1} - e_{\idd^{-1}} \cdot \bz{s_1} \otimes \bp{1} - e_1 \cdot \bp{1} \otimes \bm{s_1},
\\
	&\bp{s_0} \otimes \bz{s_0} + \bz{s_0} \otimes \bm{s_0} = - \tau_{s_0} \cdot \bm{1} \otimes \bz{s_0} + \bz{s_0} \otimes \bm{1} \cdot \tau_{s_0},
\\
	&\bm{s_1} \otimes \bz{s_1} + \bz{s_1} \otimes \bp{s_1} = - \tau_{s_1} \cdot \bp{1} \otimes \bz{s_1} + \bz{s_1} \otimes \bp{1} \cdot \tau_{s_1},
\\
	&(\tau_{s_1} + e_1) \cdot \bp{1} \otimes \bz{s_1^{-1}} \otimes \bp{1}
	+
	(\tau_{s_0} + e_1) \cdot \bm{1} \otimes \bz{s_0^{-1}} \otimes \bm{1}.	
\end{align*}
\end{thm}

\begin{proof}
Surjectivity has been proved in Lemmas \ref{lmmDefRep2} and \ref{lmmDefRep3}, while the description of the kernel follows from Propositions \ref{propKerDeg2}, \ref{propKerDeg3} and \ref{propDeg4}.
\end{proof}

\begin{rem}
The statement of the theorem does not exclude the possibility that $\ker(\mult)$ is actually generated by its $2$\textsuperscript{nd} graded piece. It is actually possible to show that this is not the case (see \cite[\S4.9]{thesis}). However, the proof is quite long and we do not present it here.
\end{rem}

\subsection{Second main result: finite presentation as a \texorpdfstring{$k$-algebra}{k-algebra}}

The purpose of this section is to show that the $\ext$-algebra $E^\ast$ is finitely presented a $k$-algebra, and to compute an explicit presentation (Theorem \ref{thmFinitePreskalg}). For this, besides Theorem \ref{thmFinalResultKernel} we will need a presentation of $E^0$ as a $k$-algebra and a presentation of $E^1$ as an $E^0$-bimodule (where by $E^0$-bimodule we mean $E^0 \otimes_k (E^0)^{\opp}$-left-module and not $E^0 \otimes_{\zz} (E^0)^{\opp}$-left-module). We begin by studying the latter problem.

For later purposes, we fix a generator $\omega_0$ of the cyclic group $T^0/T^1 \cong \ff_p^\times$:
\begin{equation}\label{eqOmegaZero}
	T^0/T^1 = \{ 1, \omega_0, \omega_0^2, \dots, \omega_0^{p-2} \}.
\end{equation}

\begin{lmm}\label{lmmPresentationE1}
Let $M$ be the free $E^0 \otimes_k (E^0)^{\opp}$-left-module of rank $4$ generated by the four symbols $\widehat{\bm{1}}$, $\widehat{\bp{1}}$, $\widehat{\bz{s_0}}$ and $\widehat{\bz{s_1}}$. Let us consider the submodule $N$ generated by the following elements:
\begin{align}
	&\tau_{s_1} \cdot \widehat{\bm{1}},
	&
	&\tau_{s_0} \cdot \widehat{\bp{1}},
	\label{eqBigListLine1}
	\\
	&\widehat{\bp{1}} \cdot \tau_{s_0},
	&
	&\widehat{\bm{1}} \cdot \tau_{s_1},
	\label{eqBigListLine2}
	\\
	&\begin{matrix*}[l]
		(\tau_{s_0} + e_1) \cdot \widehat{\bm{1}} \cdot (\tau_{s_0} + e_1)
		\\\qquad
		{} + 2 e_{\idd} \widehat{\bz{s_0}} + \tau_{\omega_0}^{\frac{p-1}{2}} \cdot \widehat{\bp{1}},
	\end{matrix*}
	&
	&\begin{matrix*}[l]
		(\tau_{s_1} + e_1) \cdot \widehat{\bp{1}} \cdot (\tau_{s_1} + e_1)
		\\\qquad
		{} - 2 e_{\idd^{-1}} \widehat{\bz{s_1}} + \tau_{\omega_0}^{\frac{p-1}{2}} \cdot \widehat{\bm{1}},
	\end{matrix*}
	\label{eqBigListLine3}
	\\
	&\tau_{s_0} \cdot \widehat{\bz{s_1}} + \widehat{\bz{s_0}} \cdot \tau_{s_1},
	&
	&\tau_{s_1} \cdot \widehat{\bz{s_0}} + \widehat{\bz{s_1}} \cdot \tau_{s_0},
	\label{eqBigListLine4}
	\\
	&(\tau_{s_0} + e_1) \cdot \widehat{\bz{s_0}} + e_{\idd} \tau_{s_0} \cdot \widehat{\bm{1}},
	&
	&(\tau_{s_1} + e_1) \cdot \widehat{\bz{s_1}} - e_{\idd^{-1}} \tau_{s_1} \cdot \widehat{\bp{1}},
	\label{eqBigListLine5}
	\\
	&\widehat{\bz{s_0}} \cdot (\tau_{s_0} + e_1) + e_{\idd^{-1}} \widehat{\bm{1}} \cdot \tau_{s_0},
	&
	&\widehat{\bz{s_1}} \cdot (\tau_{s_1} + e_1) - e_{\idd} \widehat{\bp{1}} \cdot \tau_{s_1},
	\label{eqBigListLine6}
	\\
	&\tau_{\omega_0} \cdot \widehat{\bm{1}} - u_{\omega_0}^{-2} \widehat{\bm{1}} \cdot \tau_{\omega_0},
	&
	&\tau_{\omega_0} \cdot \widehat{\bp{1}} - u_{\omega_0}^{2} \widehat{\bp{1}} \cdot \tau_{\omega_0},
	\label{eqBigListLine7}
	\\
	&\tau_{\omega_0} \cdot \widehat{\bz{s_0}} - \widehat{\bz{s_0}} \cdot \tau_{\omega_0^{-1}},
	&
	&\tau_{\omega_0} \cdot \widehat{\bz{s_1}} - \widehat{\bz{s_1}} \cdot \tau_{\omega_0^{-1}}.
	\label{eqBigListLine8}
\end{align}
For $\sigma \in \{ \bm{1}, \bp{1}, \bz{s_0}, \bz{s_1} \}$, we denote by $\widetilde{\sigma}$ the image of $\widehat{\sigma}$ in $\quoz{M}{N}$. One has an isomorphism of $E^0 \otimes_k (E^0)^{\opp}$-left-modules
\[
	\funcNN
		{\quoz{M}{N}}
		{E^1}
		{
			\begin{array}{c}
				\widetilde{\sigma}
				\\
				\scriptstyle\text{(for $\sigma \in \{ \bm{1}, \bp{1}, \bz{s_0}, \bz{s_1} \}$)}
			\end{array}
		}
		{\sigma.}
\]
\end{lmm}

\begin{proof}
To show that the claimed isomorphism is well-defined, we need to show that the elements of $E^1$ that we obtain from the elements in \eqref{eqBigListLine1}-\eqref{eqBigListLine8} by replacing $\widehat{\sigma}$ with $\sigma$ (for $\sigma \in \{ \bm{1}, \bp{1}, \bz{s_0}, \bz{s_1} \}$) are all zero. For line \eqref{eqBigListLine1} the claim follows from \eqref{eqFormulasTauQp}, for line \eqref{eqBigListLine2} the claim follows from \eqref{eqFormulaRighteasyQp}. Regarding line \eqref{eqBigListLine3}, using again such results and also \eqref{eqOmegaOnE1Left}, \eqref{eqIdempotentsE1} and \eqref{eqTauS0OnE1LeftNotAddUp}, we compute
\begin{align*}
	&(\tau_{s_0} + e_1) \cdot \bm{1} \cdot (\tau_{s_0} + e_1)
	\\&\qquad\qquad=
	\tau_{s_0} \cdot \bm{1} \cdot \tau_{s_0}
	+
	e_1 \cdot \bm{1} \cdot \tau_{s_0}
	+
	\tau_{s_0} \cdot \bm{1} \cdot e_1
	+
	e_1 \cdot \bm{1} \cdot e_1
	\\&\qquad\qquad=
	\tau_{s_0} \cdot \bm{s_0}
	+
	e_1 \bm{s_0}
	-
	e_{\idd^2} \bp{s_0}
	\\&\qquad\qquad=
	\left(
		-
		e_1 \bm{s_0}
		- 2
		e_{\idd} \bz{s_0}
		+
		e_{\idd^2} \bp{s_0}
		-
		\bp{c_{-1}}
	\right)
	+
	e_1 \bm{s_0}
	-
	e_{\idd^2} \bp{s_0}
	\\&\qquad\qquad=
	- 2
	e_{\idd} \bz{s_0}
	-
	\tau_{\omega_0}^{\frac{p-1}{2}} \cdot \bp{1}.
\end{align*}
This shows the claim for the first part of line \eqref{eqBigListLine3}, and for the second part one can just apply $\Gamma_\varpi$ and use the pertaining formulas.

To show that the elements in line \eqref{eqBigListLine4} are zero it suffices to use again the formulas \eqref{eqFormulasTauQp} and \eqref{eqFormulaRighteasyQp}. Similarly, for line \eqref{eqBigListLine5} we use \eqref{eqTauS0OnE1LeftNotAddUp} and \eqref{eqTauS1OnE1LeftNotAddUp} and for line \eqref{eqBigListLine6} we use instead \eqref{eqRightBadLengthS0} and \eqref{eqRightBadLengthS1}. Finally, for lines \eqref{eqBigListLine7} and \eqref{eqBigListLine8} we use \eqref{eqOmegaOnE1Left} and \eqref{eqOmegaOnE1Right}.

Now, it remains to prove injectivity and surjectivity of our homomorphism. Surjectivity is clear from the fact that the elements $\bm{1}$, $\bp{1}$, $\bz{s_0}$ and $\bz{s_1}$ generate $E^1$ as an $E^0$-bimodule (see Lemma \ref{lmmDeg1GeneratedByTheFourElements}). To prove injectivity, we fix a $k$-basis $\mathcal{B}$ of $E^1$, and, using surjectivity, for all $b \in \mathcal{B}$ we consider a preimage $m_b \in \quoz{M}{N}$, to be chosen later. If we prove that the elements $m_b$ generate $\quoz{M}{N}$ as a $k$-vector space then injectivity of $P$ follows.

To pursue the above strategy, let us consider the following table: the elements of $\quoz{M}{N}$ in the first column get mapped to the elements of the second column by the formulas in Lemma \ref{lmmDeg1GeneratedByTheFourElements} and \eqref{eqOmegaOnE1Right}. Moreover, the elements in the second column form a basis of $E^1$.
\[\begin{tikzcd}[row sep = -0.6em]
	\widetilde{\bz{s_1}} \cdot \tau_{w}
	\ar[r, mapsto]&
	\bz{s_1w}
	&\scriptstyle \text{for $w \in \widetilde{W}$ such that $\ell(s_1w) = \ell(w) + 1$,}
	\\
	\widetilde{\bz{s_0}} \cdot \tau_{w}
	\ar[r, mapsto]&
	\bz{s_0w}
	&\scriptstyle \text{for $w \in \widetilde{W}$ such that $\ell(s_0w) = \ell(w) + 1$,}
	\\
	\widetilde{\bm{1}} \cdot \tau_{w}
	\ar[r, mapsto]&
	\bm{w}
	&\scriptstyle \text{for $w \in \widetilde{W}$ such that $\ell(s_1w) = \ell(w) + 1$,}
	\\
	\widetilde{\bp{1}} \cdot \tau_{w}
	\ar[r, mapsto]&
	\bp{w}
	&\scriptstyle \text{for $w \in \widetilde{W}$ such that $\ell(s_0w) = \ell(w) + 1$,}
	\\
	\tau_{(s_1s_0)^j} \cdot \widetilde{\bm{1}} \cdot \tau_{\omega}
	\ar[r, mapsto]&
	 \bm{(s_1s_0)^j\omega}
	&\scriptstyle \text{for $\omega \in \quoz{T^0}{T^1}$ and $j \in \zpiu$,}
	\\
	\tau_{s_0(s_1s_0)^j} \cdot \widetilde{\bm{1}} \cdot \tau_{\omega}
	\ar[r, mapsto]&
	 - \bp{s_0(s_1s_0)^j\omega}
	&\scriptstyle \text{for $\omega \in \quoz{T^0}{T^1}$ and $j \in \nn$,}
	\\
	\tau_{(s_0s_1)^j} \cdot \widetilde{\bp{1}} \cdot \tau_{\omega}
	\ar[r, mapsto]&
	\bp{(s_0s_1)^j\omega}
	&\scriptstyle \text{for $\omega \in \quoz{T^0}{T^1}$ and $j \in \zpiu$,}
	\\
	\tau_{s_1(s_0s_1)^j} \cdot \widetilde{\bp{1}} \cdot \tau_{\omega}
	\ar[r, mapsto]&
	 - \bm{s_1(s_0s_1)^j\omega}
	&\scriptstyle \text{for $\omega \in \quoz{T^0}{T^1}$ and $j \in \nn$.}
\end{tikzcd}\]
Therefore, it remains to show that the elements in the first column generate $\quoz{M}{N}$ as a $k$-vector space. For this, it suffices to prove that, for all $v, w \in \widetilde{W}$, the following elements lie in the sub-$k$-vector space $V$ generated by the first column:
\begin{align*}
	&\tau_v \cdot \widetilde{\bm{1}} \cdot \tau_w,
	&
	&\tau_v \cdot \widetilde{\bp{1}} \cdot \tau_w,
	&
	&\tau_v \cdot \widetilde{\bz{s_0}} \cdot \tau_w,
	&
	&\tau_v \cdot \widetilde{\bz{s_1}} \cdot \tau_w.
\end{align*}
We are going to prove this claim by induction on $\ell(v) + \ell(w)$.

Let us consider the case $\ell(v) + \ell(w) = 0$. This condition means that $v$ and $w$ lie in $\quoz{T^0}{T^1}$, and so that they are powers of $\omega_0$. Using \eqref{eqBigListLine7}, we see that the elements $\tau_v \cdot \widetilde{\bm{1}} \cdot \tau_w$ (respectively, $\tau_v \cdot \widetilde{\bp{1}} \cdot \tau_w$) can be rewritten, up to a coefficient, in the form $\widetilde{\bm{1}} \cdot \tau_\omega$ (respectively, $\widetilde{\bp{1}} \cdot \tau_\omega$) for a suitable $\omega \in \quoz{T^0}{T^1}$, and so it lies in $V$, as we wanted to show. In the same way one deals with the elements $\tau_v \cdot \widetilde{\bz{s_0}} \cdot \tau_w$ and $\tau_v \cdot \widetilde{\bz{s_1}} \cdot \tau_w$, this time using \eqref{eqBigListLine8}.

Now it remains to consider the induction step. We distinguish some cases.
\begin{itemize}
\item
	Let us consider first $\tau_v \cdot \widetilde{\bm{1}} \cdot \tau_w$ and $\tau_v \cdot \widetilde{\bp{1}} \cdot \tau_w$.
	
	Since both $\widehat{\bp{1}} \cdot \tau_{s_0}$ and $\widehat{\bm{1}} \cdot \tau_{s_1}$ lie in $N$ by \eqref{eqBigListLine2}, we see that it suffices to consider the following cases:
	\begin{align*}
		&\tau_v \cdot \widetilde{\bm{1}} \cdot \tau_\omega &&\text{for some $\omega \in \quoz{T^0}{T^1}$,}
		\\
		&\tau_v \cdot \widetilde{\bp{1}} \cdot \tau_\omega &&\text{for some $\omega \in \quoz{T^0}{T^1}$,}
		\\
		&\tau_v \cdot \widetilde{\bm{1}} \cdot \tau_{s_0 w'} &&\text{for some $w' \in \widetilde{W}$ such that $\ell(s_0 w') = \ell(w') + 1$,}
		\\
		&\tau_v \cdot \widetilde{\bp{1}} \cdot \tau_{s_1 w'} &&\text{for some $w' \in \widetilde{W}$ such that $\ell(s_1 w') = \ell(w') + 1$.}
	\end{align*}
	Let us look at the first two elements. Since both $\tau_{s_1} \cdot \widehat{\bm{1}}$ and $\tau_{s_0} \cdot \widehat{\bp{1}}$ lie in $N$ by \eqref{eqBigListLine1}, it suffices to consider the following cases:
	\begin{align*}
		&\tau_{(s_1s_0)^j \omega'} \cdot \widetilde{\bm{1}} \cdot \tau_\omega &&\text{for some $j \in \nn$ and some $\omega,\omega' \in \quoz{T^0}{T^1}$,}
		\\
		&\tau_{s_0(s_1s_0)^j \omega'} \cdot \widetilde{\bm{1}} \cdot \tau_\omega &&\text{for some $j \in \nn$ and some $\omega,\omega' \in \quoz{T^0}{T^1}$,}
		\\
		&\tau_{(s_0s_1)^j \omega'} \cdot \widetilde{\bp{1}} \cdot \tau_\omega &&\text{for some $j \in \nn$ and some $\omega,\omega' \in \quoz{T^0}{T^1}$,}
		\\
		&\tau_{s_1(s_0s_1)^j \omega'} \cdot \widetilde{\bp{1}} \cdot \tau_\omega &&\text{for some $j \in \nn$ and some $\omega,\omega' \in \quoz{T^0}{T^1}$.}
	\end{align*}
	If $\omega'=1$ then these elements are in $V$ because they are in the list of generators of $V$, and in the general case we can reduce to this situation exactly as we did in the case $\ell(v) + \ell(w) = 0$.
	
	Now we have to consider the elements
	\begin{align*}
		&\tau_v \cdot \widetilde{\bm{1}} \cdot \tau_{s_0} \cdot \tau_{w'} &&\text{for some $w' \in \widetilde{W}$ such that $\ell(s_0 w') = \ell(w') + 1$,}
		\\
		&\tau_v \cdot \widetilde{\bp{1}} \cdot \tau_{s_1} \cdot \tau_{w'} &&\text{for some $w' \in \widetilde{W}$ such that $\ell(s_1 w') = \ell(w') + 1$.}
	\end{align*}
	If $\ell(v) = 0$, we reduce as usual to the case $v=1$, in which case we see that the above elements are in the list of generators of $V$. So we might assume that $\ell(v) \geqslant 1$. In this case, using again that both $\tau_{s_1} \cdot \widehat{\bm{1}}$ and $\tau_{s_0} \cdot \widehat{\bp{1}}$ lie in $N$, we are reduced to considering elements of the following forms:
	\begin{align*}
		&\tau_{v'} \cdot \tau_{s_0} \cdot \widetilde{\bm{1}} \cdot \tau_{s_0} \cdot \tau_{w'},
		\\
		&\tau_{v'} \cdot \tau_{s_1} \cdot \widetilde{\bp{1}} \cdot \tau_{s_1} \cdot \tau_{w'},
	\end{align*}
	where $v'$ and $w'$ are elements of $\widetilde{W}$ such that lengths add up in the corresponding multiplications (i.e., $\tau_{v'} \cdot \tau_{s_0}$ and $\tau_{s_0} \cdot \tau_{w'}$ in the case of the first line).
	Now, looking at \eqref{eqBigListLine3}, we see that we can write these two elements as linear combinations of elements of the form
	\begin{align*}
		&\tau_{v''} \cdot \widetilde{\bm{1}} \cdot \tau_{w''},
		&
		&\tau_{v''} \cdot \widetilde{\bp{1}} \cdot \tau_{w''},
		&
		&\tau_{v''} \cdot \widetilde{\bz{s_0}} \cdot \tau_{w''},
		&
		&\tau_{v''} \cdot \widetilde{\bz{s_1}} \cdot \tau_{w''}
	\end{align*}
	for $v'',w'' \in \widetilde{W}$ such that $\ell(v'') + \ell(w'') < \ell(v) + \ell(w)$. The induction hypothesis then allow us to conclude that the elements $\tau_{v' s_0} \cdot \widetilde{\bm{1}} \cdot \tau_{s_0 w'}$ and $\tau_{v' s_1} \cdot \widetilde{\bp{1}} \cdot \tau_{s_1 w'}$ lie in $V$.
\item
	Now we consider the elements $\tau_v \cdot \widetilde{\bz{s_i}} \cdot \tau_w$ (for $i \in \{0,1\}$) under the additional assumption that $\ell(vs_i) = \ell(v) + 1$ and $\ell(s_iw) = \ell(w) + 1$.
	
	Making $v$ explicit, we see that we are dealing with the following elements:
	\begin{align*}
		\tau_{(s_0s_1)^i\omega} \cdot \widetilde{\bz{s_0}} \cdot \tau_w
		&&\scriptstyle{\text{for $i \in \nn$, $\omega \in \quoz{T^0}{T^1}$, $w \in \widetilde{W}$ with $\ell(s_0w) = \ell(w) + 1$,}}
		\\
		\tau_{s_1(s_0s_1)^i\omega} \cdot \widetilde{\bz{s_0}} \cdot \tau_w
		&&\scriptstyle{\text{for $i \in \nn$, $\omega \in \quoz{T^0}{T^1}$, $w \in \widetilde{W}$ with $\ell(s_0w) = \ell(w) + 1$,}}
		\\
		\tau_{(s_1s_0)^i\omega} \cdot \widetilde{\bz{s_1}} \cdot \tau_w
		&&\scriptstyle{\text{for $i \in \nn$, $\omega \in \quoz{T^0}{T^1}$, $w \in \widetilde{W}$ with $\ell(s_1w) = \ell(w) + 1$,}}
		\\
		\tau_{s_0(s_1s_0)^i\omega} \cdot \widetilde{\bz{s_1}} \cdot \tau_w
		&&\scriptstyle{\text{for $i \in \nn$, $\omega \in \quoz{T^0}{T^1}$, $w \in \widetilde{W}$ with $\ell(s_1w) = \ell(w) + 1$.}}
	\end{align*}
	As before, we see that we can assume without loss of generality that $\omega = 1$, and then, using \eqref{eqBigListLine4} repeatedly, we see that, up to a sign, the above elements can be rewritten in the form $\widetilde{\bz{s_i}} \cdot \tau_v$ for some $i \in \{0,1\}$ and some $v \in \widetilde{W}$ such that $\ell(s_i v) = \ell(v) + 1$.
\item
	Now we consider the elements $\tau_v \cdot \widetilde{\bz{s_i}} \cdot \tau_w$ (for $i \in \{0,1\}$) under the additional assumption that $\ell(s_iw) = \ell(w) - 1$.
	
	We only deal with the case $i=0$, the other case being completely analogous. Using \eqref{eqBigListLine6}, we see that
	\begin{align*}
		\tau_v \cdot \widetilde{\bz{s_0}} \cdot \tau_w
		&=
		\tau_v \cdot \widetilde{\bz{s_0}} \cdot \tau_{s_0} \cdot \tau_{s_0^{-1}w}
		\\&=
		\tau_v \cdot \left( - \widetilde{\bz{s_0}} \cdot e_1 - e_{\idd^{-1}} \widetilde{\bm{1}} \cdot \tau_{s_0} \right) \cdot \tau_{s_0^{-1}w}
		\\&=
		- \tau_v \cdot \widetilde{\bz{s_0}} \cdot \tau_{s_0^{-1}w}e_1
		- \tau_v e_{\idd^{-1}}  \cdot \widetilde{\bm{1}} \cdot \tau_{w}.
	\end{align*}
	The first term in the last line is a linear combination of terms of the form $\tau_{v'} \cdot \widetilde{\bz{s_0}} \cdot \tau_{w'}$ for $v', w' \in \widetilde{W}$ such that $\ell(v') + \ell(w') = \ell(v) + \ell(w) - 1$, and so by inductive hypothesis it lies in $V$. On the other hand, the second term is a linear combination of terms of the form $\tau_{v'} \cdot \widetilde{\bm{1}} \cdot \tau_{w'}$ for $v', w' \in \widetilde{W}$ such that $\ell(v') + \ell(w') = \ell(v) + \ell(w)$, and these also lie in $V$, as already proved.
\item	
	Now it remains to consider the elements $\tau_v \cdot \widetilde{\bz{s_i}} \cdot \tau_w$ (for $i \in \{0,1\}$) under the additional assumption that $\ell(vs_i) = \ell(v) - 1$. This case is completely analogous to the previous one, using \eqref{eqBigListLine5} in place of \eqref{eqBigListLine6}.
\qedhere
\end{itemize}
\end{proof}

In the next lemma we compute a presentation of the Hecke algebra $E^0=H$. This is straightforward (and probably well-known) under our assumptions. Also compare the finite presentation of $H$ in the case $G = \GL_n(\field)$ determined in \cite[\S2.1]{grosse}. The following lemma does not use anything about the $\ext$-algebra, and it is actually valid in the case $G = \SL_2(\field)$, without assumptions on $\field$ or $p$, but replacing the occurrences of $p$ with $q$.

\begin{lmm}\label{lmmPresentationE0}
The pro-$p$ Iwahori--Hecke algebra $H=E^0$ can be expressed by generators and relations as follows. Let
\[
	k \left\langle \widehat{\tau_{\omega_0}}, \widehat{\tau_{s_0}}, \widehat{\tau_{s_1}} \right\rangle
\]
be the ring of non-commutative polynomials in three indeterminates called $\widehat{\tau_{\omega_0}}$, $\widehat{\tau_{s_0}}$ and $\widehat{\tau_{s_1}}$. Furthermore, let $I$ be the two-sided ideal of $k \left\langle \widehat{\tau_{\omega_0}}, \widehat{\tau_{s_0}}, \widehat{\tau_{s_1}} \right\rangle$ generated by the following elements
\begin{align}
	&\widehat{\tau_{\omega_0}}^{p-1} - 1,
	\label{eqSmallList1}
	\\
	&\widehat{\tau_{\omega_0}} \cdot \widehat{\tau_{s_0}} - \widehat{\tau_{s_0}} \cdot \widehat{\tau_{\omega_0}}^{p-2},
	&
	&\widehat{\tau_{\omega_0}} \cdot \widehat{\tau_{s_1}} - \widehat{\tau_{s_1}} \cdot \widehat{\tau_{\omega_0}}^{p-2},
	\label{eqSmallList2}
	\\
	&\widehat{\tau_{s_0}}^2 - \sum_{i=0}^{p-2} \widehat{\tau_{\omega_0}}^{i} \cdot \widehat{\tau_{s_0}},
	&
	&\widehat{\tau_{s_1}}^2 - \sum_{i=0}^{p-2} \widehat{\tau_{\omega_0}}^{i} \cdot \widehat{\tau_{s_1}}.
	\label{eqSmallList3}
\end{align}
Let $\widetilde{\tau_{\omega_0}}$, $\widetilde{\tau_{s_0}}$ and $\widetilde{\tau_{s_1}}$ be respectively the images of $\widehat{\tau_{\omega_0}}$, $\widehat{\tau_{s_0}}$ and $\widehat{\tau_{s_1}}$ in the quotient algebra $\quoz{k \left\langle \widehat{\tau_{\omega_0}}, \widehat{\tau_{s_0}}, \widehat{\tau_{s_1}} \right\rangle}{I}$. One has an isomorphism of $k$-algebras
\[
	\funcAboveTriple
		{}
		{\quoz{k \left\langle \widehat{\tau_{\omega_0}}, \widehat{\tau_{s_0}}, \widehat{\tau_{s_1}} \right\rangle}{I}}
		{H=E^0}
		{\widetilde{\tau_{\omega_0}}}
		{\tau_{\omega_0},}
		{\widetilde{\tau_{s_0}}}
		{\tau_{s_0},}
		{\widetilde{\tau_{s_1}}}
		{\tau_{s_1}.}
\]
\end{lmm}

\begin{proof}
The braid relations, the quadratic relations and the properties of the group $\widetilde{W}$ imply that the homomorphism in the statement is indeed well-defined and surjective.

To show injectivity, we proceed as in Lemma \ref{lmmPresentationE1}. Let us consider the following table, where the elements in the first column are mapped to the elements on the right column by our $k$-algebra homomorphism
$
	\funcInlineNN
		{\quoz{k \left\langle \widehat{\tau_{\omega_0}}, \widehat{\tau_{s_0}}, \widehat{\tau_{s_1}} \right\rangle}{I}}
		{E^0}
$, and where the elements in the right column form a $k$-basis of $E^0$:
\[\begin{tikzcd}[row sep = -0.6em]
	\widetilde{\tau_{\omega_0}}^i
	\ar[r, mapsto]
	&\tau_{\omega_0^i}
	&\substack{\text{for $i \in \{0, \dots, p-2\}$,}}
	\\
	\widetilde{\tau_{\omega_0}}^i \cdot \widetilde{\tau_{s_1}} \cdot (\widetilde{\tau_{s_0}} \cdot \widetilde{\tau_{s_1}})^j
	\ar[r, mapsto]
	&\tau_{\omega_0^i s_1(s_0s_1)^j}
	&\substack{\text{for $i \in \{0, \dots, p-2\}$ and $j \in \nn$,}}
	\\
	\widetilde{\tau_{\omega_0}}^i \cdot \widetilde{\tau_{s_0}} \cdot (\widetilde{\tau_{s_1}} \cdot \widetilde{\tau_{s_0}})^j
	\ar[r, mapsto]
	&\tau_{\omega_0^i s_0(s_1s_0)^j}
	&\substack{\text{for $i \in \{0, \dots, p-2\}$ and $j \in \nn$,}}
	\\
	\widetilde{\tau_{\omega_0}}^i \cdot (\widetilde{\tau_{s_0}} \cdot \widetilde{\tau_{s_1}})^j
	\ar[r, mapsto]
	&\tau_{\omega_0^i (s_0s_1)^j}
	&\substack{\text{for $i \in \{0, \dots, p-2\}$ and $j \in \zpiu$,}}
	\\
	\widetilde{\tau_{\omega_0}}^i \cdot (\widetilde{\tau_{s_1}} \cdot \widetilde{\tau_{s_0}})^j
	\ar[r, mapsto]
	&\tau_{\omega_0^i (s_1s_0)^j}
	&\substack{\text{for $i \in \{0, \dots, p-2\}$ and $j \in \zpiu$.}}
\end{tikzcd}\]

As in Lemma \ref{lmmPresentationE1}, to show injectivity of our homomorphism, it suffices to show that the $k$-vector space $V$ generated by the elements in the left column is the full $\quoz{k \left\langle \widehat{\tau_{\omega_0}}, \widehat{\tau_{s_0}}, \widehat{\tau_{s_1}} \right\rangle}{I}$. So, it suffices to prove that every element of the form
\[
	\widetilde{\tau_{w_1}} \cdots \widetilde{\tau_{w_n}},
\]
for $n \in \nn$ and $w_1, \dots, w_n \in \{ \omega_0, s_0, s_1 \}$, lies in $V$. Using the relations \eqref{eqSmallList1} and \eqref{eqSmallList2}, we see that we may further reduce to elements of the form
\[
	\widetilde{\tau_{\omega_0}}^i \cdot \widetilde{\tau_{s_{l_1}}} \cdots \widetilde{\tau_{s_{l_m}}}
\]
for $i \in \{0, \dots, p-2\}$, for $m \in \nn$ and for $l_1, \dots, l_m \in \{0,1\}$. We now prove that an element of this form lies in $V$ by induction on $m$. If $m = 0$, then the result is clear. Furthermore, for general $m$ the result is clear in the case that there are no consecutive indices $l_{j}$ and $l_{j+1}$ both equal to $0$ or both equal to $1$. If instead there is at least one such pair of indices, then we use \eqref{eqSmallList3} (together with \eqref{eqSmallList1} and \eqref{eqSmallList2} again) to reduce to the case of smaller $m$.
\end{proof}

\begin{lmm}\label{lmmTecnicoPresentation}
Let $A^\ast$ be a positively graded algebra over a field $F$ (associative, with $1$, not necessarily commutative) such that the natural map
$
	\funcInlineNN{T^\ast_{A^0} A^1}{A^\ast}
$
is surjective. Let us consider generators the kernel of the above map (as a two-sided ideal), obtaining a presentation
\[
	\funcInlineNN{\frac{T^\ast_{A^0} A^1}{(a_i)_{i \in \mathbf{I}}}}{A^\ast}.
\]
Let us consider a presentation of $A^1$ as an $A^0 \otimes_k (A^0)^{\opp}$-left module
\[
	\funcInlineNN{\frac{\bigoplus_{j \in \mathbf{J}} \big( A^0 \otimes_k (A^0)^{\opp} \big) X_j}{(b_k)_{k \in \mathbf{K}}}}{A^1},
\]
and let us consider a presentation of $A^0$ as a $k$-algebra
\[
	\funcInlineNN{\frac{F \langle Y_l \rangle_{l \in \mathbf{L}}}{(c_m)_{m \in \mathbf{M}}}}{A^0}.
\]
For all $i \in \mathbf{I}$ (respectively, for all $k \in \mathbf{K}$) let us choose a lift $\dot{a}_i$ (respectively, $\dot{b}_k$) of $a_i$ (respectively, of $b_k$) in $F \langle X_j, Y_l \rangle_{j \in \mathbf{J}, l \in \mathbf{L}}$.

One then has an induced presentation of $A^\ast$ as a $k$-algebra
\[
	\funcInlineNN
		{\frac
			{F \langle X_j, Y_l \rangle_{j \in \mathbf{J}, l \in \mathbf{L}}}
			{\ideal{ \dot{a}_i, \dot{b}_k, c_m }{ i \in \mathbf{I}, k \in \mathbf{K}, m \in \mathbf{M} }}
		}{A^\ast}.
\]
\end{lmm}

\begin{proof}
For the moment let us not assume that the map
$
	\funcInlineNN{T^\ast_{A^0} A^1}{A^\ast}
$
is surjective.

Let $\funcInline{\iota_0}{A^0}{T_{A^0}^\ast A^1}$ and $\funcInline{\iota_1}{A^1}{T_{A^0}A^1}$ be the canonical inclusions. It is easy to check that the triple $(T^\ast_{A^0}A^1, \iota_0, \iota_1)$ enjoys the following universal property: for all triples $(R,\xi_0,\xi_1)$ consisting of a $F$-algebra $R$ (associative, not necessarily commutative), a homomorphism of $F$-algebras $\funcInline{\xi_0}{A^0}{R}$ and a homomorphism of left $A^0 \otimes_F (A^0)^{\opp}$-modules $\funcInline{\xi_1}{A^1}{R}$ (where $R$ is a left $A^0 \otimes_F (A^0)^{\opp}$-module via $\xi_0$), there exists a unique homomorphism of $F$-algebras $\funcInline{\eta}{T_{A^0}^\ast A^1}{R}$ making the following diagrams commute:
\begin{align*}
&\begin{tikzcd}[row sep = 0.2em, ampersand replacement = \&]
	\&
	T_{A^0}^\ast A^1
	\ar[dd, "{\eta}"]
	\\
	A^0
	\ar[ur, "{\iota_0}"]
	\ar[dr, "{\xi_0}"']
	\&
	\\
	\&
	R,
\end{tikzcd}
&\begin{tikzcd}[row sep = 0.2em, ampersand replacement = \&]
	\&
	T_{A^0}^\ast A^1
	\ar[dd, "{\eta}"]
	\\
	A^1
	\ar[ur, "{\iota_1}"]
	\ar[dr, "{\xi_1}"']
	\&
	\\
	\&
	R.
\end{tikzcd}
\end{align*}
This universal property provides an inverse to the obvious map
\[
	\funcInlineNN
		{\frac
			{F \langle X_j, Y_l \rangle_{j \in \mathbf{J}, l \in \mathbf{L}}}
			{\ideal{ \dot{b}_k, c_m }{ k \in \mathbf{K}, m \in \mathbf{M} }}
		}{T^\ast_{A^0}A^1},
\]
and the claim of the lemma follows (assuming that the map
$
	\funcInlineNN{T^\ast_{A^0} A^1}{A^\ast}
$
is surjective).
\end{proof}

The preceding lemma directly allows us to compute a presentation of $E^\ast$, since we have already computed the three presentations required in the statement of the lemma, in Theorem \ref{thmFinalResultKernel}, Lemma \ref{lmmPresentationE0} and Lemma \ref{lmmPresentationE1} respectively.

Let us work with the ring of noncommutative polynomials in seven indeterminates
$
	k \left\langle \widehat{\tau_{\omega_0}}, \widehat{\tau_{s_0}}, \widehat{\tau_{s_1}}, \widehat{\bm{1}}, \widehat{\bp{1}}, \widehat{\bz{s_0}}, \widehat{\bz{s_1}} \right\rangle
$
(where the indeterminates have been labelled in accordance with the notation of Lemma \ref{lmmPresentationE1} and Lemma \ref{lmmPresentationE0}).

For all $\lambda \in \widehat{\quoz{T^0}{T^1}}$, let us give the following definition, mimicking the definition of $e_\lambda$:
\[
	\varepsilon_\lambda \defeq - \sum_{i=0}^{p-1} \lambda(\omega_0^{-i}) \widehat{\tau_{\omega_0}}^i \in k[\widehat{\tau_{\omega_0}}].
\]

Let us consider the following elements of
$
	k \left\langle \widehat{\tau_{\omega_0}}, \widehat{\tau_{s_0}}, \widehat{\tau_{s_1}}, \widehat{\bm{1}}, \widehat{\bp{1}}, \widehat{\bz{s_0}}, \widehat{\bz{s_1}} \right\rangle
$:
\begin{equation}\label{eqList3}\begin{aligned}
&\begin{aligned}
	&\widehat{\bm{1}} \cdot \widehat{\bm{1}}, &&\qquad\qquad\qquad\widehat{\bp{1}} \cdot \widehat{\bm{1}}, &&\qquad\qquad\qquad\widehat{\bz{s_1}} \cdot \widehat{\bm{1}},
	\\
	&\widehat{\bm{1}} \cdot \widehat{\bp{1}}, &&\qquad\qquad\qquad\widehat{\bp{1}} \cdot \widehat{\bp{1}}, &&\qquad\qquad\qquad\widehat{\bz{s_0}} \cdot \widehat{\bp{1}},
	\\
	&\widehat{\bp{1}} \cdot \widehat{\bz{s_0}}, &&\qquad\qquad\qquad\widehat{\bz{s_1}} \cdot \widehat{\bz{s_0}},
	\\
	&\widehat{\bm{1}} \cdot \widehat{\bz{s_1}}, &&\qquad\qquad\qquad\widehat{\bz{s_0}} \cdot \widehat{\bz{s_1}},
\end{aligned}
\\
	&\widehat{\bz{s_0}} \cdot \widehat{\bz{s_0}} + \varepsilon_{\idd^{-1}} \cdot \widehat{\bm{1}} \cdot \widehat{\bz{s_0}} + \varepsilon_{\idd} \cdot \widehat{\bz{s_0}} \cdot \widehat{\bm{1}} + \varepsilon_1 \cdot \widehat{\bm{1}} \cdot \widehat{\tau_{s_0}} \cdot \widehat{\bm{1}}, 
\\
	&\widehat{\bz{s_1}} \cdot \widehat{\bz{s_1}} - \varepsilon_{\idd} \cdot \widehat{\bp{1}} \cdot \widehat{\bz{s_1}} - \varepsilon_{\idd^{-1}} \cdot \widehat{\bz{s_1}} \cdot \widehat{\bp{1}} + \varepsilon_1 \cdot \widehat{\bp{1}} \cdot \widehat{\tau_{s_1}} \cdot \widehat{\bp{1}},
\\
	&\widehat{\bz{s_0}} \cdot \widehat{\bm{1}} \cdot \widehat{\tau_{s_0}} - \widehat{\tau_{s_0}} \cdot \widehat{\bm{1}} \cdot \widehat{\bz{s_0}},
\\
	&\widehat{\bz{s_1}} \cdot \widehat{\bp{1}} \cdot \widehat{\tau_{s_1}} - \widehat{\tau_{s_1}} \cdot \widehat{\bp{1}} \cdot \widehat{\bz{s_1}},
\\
	&(\widehat{\tau_{s_1}} + \varepsilon_1) \cdot \widehat{\bp{1}} \cdot \widehat{\bz{s_1}} \cdot \widehat{\bp{1}}
	+
	(\widehat{\tau_{s_0}} + \varepsilon_1) \cdot \widehat{\bm{1}} \cdot \widehat{\bz{s_0}} \cdot \widehat{\bm{1}}.	
\end{aligned}\end{equation}
We claim that this is a set of lifts to
$
	k \left\langle \widehat{\tau_{\omega_0}}, \widehat{\tau_{s_0}}, \widehat{\tau_{s_1}}, \widehat{\bm{1}}, \widehat{\bp{1}}, \widehat{\bz{s_0}}, \widehat{\bz{s_1}} \right\rangle
$
of a set of generators of the kernel of the multiplication map $\funcInline{\mult}{T^\ast_{E^0} E^1}{E^\ast}$ (as a two-sided ideal). Indeed, this list has been obtained in the following way from the list of generators of the $\ker(\mult)$ given in Theorem \ref{thmFinalResultKernel}: the occurrences of $\bm{1}$ have been replaced by $\widehat{\bm{1}}$, the occurrences of $\bp{1}$ have been replaced by $\widehat{\bp{1}}$, the occurrences of $\bz{s_i}$ have been replaced by $\widehat{\bz{s_i}}$ for $i \in \{0,1\}$, the occurrences of $\tau_{s_i}$ have been replaced by $\widehat{\tau_{s_i}}$ for $i \in \{0,1\}$ and the occurrences of $e_\lambda$ have been replaced by $\varepsilon_\lambda$ for $\lambda \in \{ 1, \idd, \idd^{-1} \}$. Furthermore, the occurrences of $\beta^\pm_{s_i}$ (for $i \in \{0,1\}$) have first been replaced according to the following rules (see the formulas in Lemma \ref{lmmDeg1GeneratedByTheFourElements}):
\begin{align*}
	\bp{s_0} &= - \tau_{s_0} \cdot \bm{1},
	&
	\bm{s_1} &= - \tau_{s_1} \cdot \bp{1},
	\\
	\bm{s_0} &= \bm{1} \cdot \tau_{s_0},
	&
	\bp{s_1} &= \bp{1} \cdot \tau_{s_1},
\end{align*}
and then, of course, according to the rules above. Finally, as regards the generator
$
	(\tau_{s_1} + e_1) \cdot \bp{1} \otimes \bz{s_1^{-1}} \otimes \bp{1}
	+
	(\tau_{s_0} + e_1) \cdot \bm{1} \otimes \bz{s_0^{-1}} \otimes \bm{1}
$,
for simplicity we have also deleted the exponent \virgolette{$-1$} (which can be achieved up to multiplying by $\tau_{c_{-1}}$, thus not affecting the fact that we are dealing with a set of generators).
This shows that \eqref{eqList3} is a set of lifts to
$
	k \left\langle \widehat{\tau_{\omega_0}}, \widehat{\tau_{s_0}}, \widehat{\tau_{s_1}}, \widehat{\bm{1}}, \widehat{\bp{1}}, \widehat{\bz{s_0}}, \widehat{\bz{s_1}} \right\rangle
$
of a set of generators of $\ker(\mult)$, as claimed. In the notation of the previous lemma, the list \eqref{eqList3} corresponds to the family $(\dot{a}_i)_{i \in \mathbf{I}}$.

Furthermore, let consider the following list of elements, obtaining from the list in Lemma \ref{lmmPresentationE1} by replacing $\tau_{\omega_0}$ with $\widehat{\tau_{\omega_0}}$, by replacing $\tau_{\omega_0^{-1}}=\tau_{\omega_0}^{p-2}$ with $\widehat{\tau_{\omega_0}}^{p-2}$ by replacing $\tau_{s_i}$ with $\widehat{\tau_{s_i}}$ for $i \in \{0,1\}$ and by replacing $e_\lambda$ by $\varepsilon_\lambda$ for $\lambda \in \{ 1, \idd, \idd^{-1} \}$:
\begin{equation}\label{eqListE1}\begin{aligned}
	&\widehat{\tau_{s_1}} \cdot \widehat{\bm{1}},
	&
	&\qquad\qquad\widehat{\tau_{s_0}} \cdot \widehat{\bp{1}},
	\\
	&\widehat{\bp{1}} \cdot \widehat{\tau_{s_0}},
	&
	&\qquad\qquad\widehat{\bm{1}} \cdot \widehat{\tau_{s_1}},
	\\
	&
	\begin{matrix*}[l]
		(\widehat{\tau_{s_0}} + \varepsilon_1) \cdot \widehat{\bm{1}} \cdot (\widehat{\tau_{s_0}} + \varepsilon_1)
		\\\qquad
		+ 2 \varepsilon_{\idd} \widehat{\bz{s_0}} + \widehat{\tau_{\omega_0}}^{\frac{p-1}{2}} \cdot \widehat{\bp{1}},
	\end{matrix*}
	&
	&\qquad\qquad
	\begin{matrix*}[l]
		(\widehat{\tau_{s_1}} + \varepsilon_1) \cdot \widehat{\bp{1}} \cdot (\widehat{\tau_{s_1}} + \varepsilon_1)
		\\\qquad
		- 2 \varepsilon_{\idd^{-1}} \widehat{\bz{s_1}} + \widehat{\tau_{\omega_0}}^{\frac{p-1}{2}} \cdot \widehat{\bm{1}},
	\end{matrix*}
	\\
	&\widehat{\tau_{s_0}} \cdot \widehat{\bz{s_1}} + \widehat{\bz{s_0}} \cdot \widehat{\tau_{s_1}},
	&
	&\qquad\qquad\widehat{\tau_{s_1}} \cdot \widehat{\bz{s_0}} + \widehat{\bz{s_1}} \cdot \widehat{\tau_{s_0}},
	\\
	&(\widehat{\tau_{s_0}} + \varepsilon_1) \cdot \widehat{\bz{s_0}} + \varepsilon_{\idd} \widehat{\tau_{s_0}} \cdot \widehat{\bm{1}},
	&
	&\qquad\qquad(\widehat{\tau_{s_1}} + \varepsilon_1) \cdot \widehat{\bz{s_1}} - \varepsilon_{\idd^{-1}} \widehat{\tau_{s_1}} \cdot \widehat{\bp{1}},
	\\
	&\widehat{\bz{s_0}} \cdot (\widehat{\tau_{s_0}} + \varepsilon_1) + \varepsilon_{\idd^{-1}} \widehat{\bm{1}} \cdot \widehat{\tau_{s_0}},
	&
	&\qquad\qquad\widehat{\bz{s_1}} \cdot (\widehat{\tau_{s_1}} + \varepsilon_1) - \varepsilon_{\idd} \widehat{\bp{1}} \cdot \widehat{\tau_{s_1}},
	\\
	&\widehat{\tau_{\omega_0}} \cdot \widehat{\bm{1}} - u_{\omega_0}^{-2} \widehat{\bm{1}} \cdot \widehat{\tau_{\omega_0}},
	&
	&\qquad\qquad\widehat{\tau_{\omega_0}} \cdot \widehat{\bp{1}} - u_{\omega_0}^{2} \widehat{\bp{1}} \cdot \widehat{\tau_{\omega_0}},
	\\
	&\widehat{\tau_{\omega_0}} \cdot \widehat{\bz{s_0}} - \widehat{\bz{s_0}} \cdot \widehat{\tau_{\omega_0}}^{p-2},
	&
	&\qquad\qquad\widehat{\tau_{\omega_0}} \cdot \widehat{\bz{s_1}} - \widehat{\bz{s_1}} \cdot \widehat{\tau_{\omega_0}}^{p-2}.
\end{aligned}\end{equation}
In the notation of the previous lemma, this list corresponds to the family $(\dot{b}_k)_{k \in \mathbf{K}}$.

Furthermore, let us recall the list of elements Lemma \ref{lmmPresentationE0} (which we simply rewrite using the notation $\varepsilon_1$):
\begin{equation}\label{eqListE0}\begin{aligned}
	&\widehat{\tau_{\omega_0}}^{p-1} - 1,
	\\
	&\widehat{\tau_{\omega_0}} \cdot \widehat{\tau_{s_0}} - \widehat{\tau_{s_0}} \cdot \widehat{\tau_{\omega_0}}^{p-2},
	&
	&\widehat{\tau_{\omega_0}} \cdot \widehat{\tau_{s_1}} - \widehat{\tau_{s_1}} \cdot \widehat{\tau_{\omega_0}}^{p-2},
	\\
	&\widehat{\tau_{s_0}}^2 + \varepsilon_1 \widehat{\tau_{s_0}},
	&
	&\widehat{\tau_{s_1}}^2 + \varepsilon_1 \widehat{\tau_{s_1}}.
\end{aligned}\end{equation}
In the notation of the previous lemma, this list corresponds to the family $(c_m)_{m \in \mathbf{M}}$.

\begin{thm}\label{thmFinitePreskalg}
Recall the assumptions stated at the beginning of the current \S\ref{sectionFinPres}.
The $\ext$-algebra $E^\ast$ is finitely presented as a $k$-algebra. More precisely, it can be expressed as the quotient of the ring of non-commutative polynomials in seven indeterminates
$
	k \left\langle \widehat{\tau_{\omega_0}}, \widehat{\tau_{s_0}}, \widehat{\tau_{s_1}}, \widehat{\bm{1}}, \widehat{\bp{1}}, \widehat{\bz{s_0}}, \widehat{\bz{s_1}} \right\rangle
$
by the two-sided ideal generated by the elements in the lists \eqref{eqList3}, \eqref{eqListE1} and \eqref{eqListE0}.
\end{thm}

\begin{proof}
This follows from Lemma \ref{lmmTecnicoPresentation} and the subsequent discussion (which relied on the presentations obtained in Theorem \ref{thmFinalResultKernel}, Lemma \ref{lmmPresentationE0} and Lemma \ref{lmmPresentationE1}).
\end{proof}

\begin{rem}
The relation
\[(\tau_{s_1} + e_1) \cdot \bp{1} \cdot (\tau_{s_1} + e_1) - 2 e_{\idd^{-1}} \bz{s_1} + \tau_{\omega_0}^{\frac{p-1}{2}} \cdot \bm{1} = 0\]
(which we used in to produce the corresponding element in the list \eqref{eqListE1}) allows us to express $\bm{1}$ in terms of $\bp{1}$, $\bz{s_1}$, $\tau_{s_1}$ and $\tau_{\omega_0}$.
This shows that $E^\ast$ is actually generated by $\tau_{\omega_0}$, $\tau_{s_0}$, $\tau_{s_1}$, $\bp{1}$, $\bz{s_0}$ and $\bz{s_1}$ as a $k$-algebra, without the need to add $\bm{1}$. Using the presentation we obtained in the last theorem, it is then immediate (but cumbersome) to obtain a presentation of $E^\ast$ as a quotient of $k \left\langle \widehat{\tau_{\omega_0}}, \widehat{\tau_{s_0}}, \widehat{\tau_{s_1}}, \widehat{\bp{1}}, \widehat{\bz{s_0}}, \widehat{\bz{s_1}} \right\rangle$ by replacing $\widehat{\bm{1}}$ with
\[ - \widehat{\tau_{\omega_0}}^{\frac{p-1}{2}} \cdot (\widehat{\tau_{s_1}} + \varepsilon_1) \cdot \widehat{\bp{1}} \cdot (\widehat{\tau_{s_1}} + \varepsilon_1) + 2 \widehat{\tau_{\omega_0}}^{\frac{p-1}{2}} \cdot \varepsilon_{\idd^{-1}} \widehat{\bz{s_1}}\]
wherever it appears in the lists \eqref{eqList3}, \eqref{eqListE1} and \eqref{eqListE0}.
\end{rem}

\bibliographystyle{alpha}
\bibliography{biblio}
\addcontentsline{toc}{section}{References}

\begin{flushleft}
\textsc{%
University of British Columbia,\\
Mathematics Department,\\
1984 Mathematics Road,\\
Vancouver, BC V6T 1Z2, Canada
}\\
E-mail address:
\texttt{e.bodon@math.ubc.ca}
\end{flushleft}
\end{document}